\documentclass[10pt]{amsart}

\topskip  \usepackage{auto-pst-pdf}
\usepackage{amsmath,amsthm,cite}
\usepackage{amssymb,longtable,multido,pstricks,pstricks-add}
\usepackage{mathdots}
\usepackage{color}
\usepackage{xspace}
\usepackage[colorlinks=true,
linkcolor=green,
filecolor=brown,
citecolor=green]{hyperref}

\newtheorem{theorem}{Theorem}
\newtheorem{lemma}[theorem]{Lemma}
\newtheorem{proposition}[theorem]{Proposition}
\newtheorem{corollary}[theorem]{Corollary}
\def\p{\pi}
\def\gf{generating function\xspace}
\def\al{\alpha}

\begin{document}
\title{Wilf classification of triples of 4-letter patterns}

\author[D. Callan]{David Callan}
\address{Department of Statistics, University of Wisconsin, Madison, WI 53706}
\email{callan@stat.wisc.edu}
\author[T.~Mansour]{Toufik Mansour}
\address{Department of Mathematics, University of Haifa, 31905 Haifa, Israel}
\email{tmansour@univ.haifa.ac.il}
\author[M.~Shattuck]{Mark Shattuck}
\address{Department of Mathematics, University of Tennessee, Knoxville, TN 37996}
\email{shattuck@math.utk.edu}
\maketitle

\begin{abstract}
We determine all 242 Wilf classes of triples of 4-letter patterns by showing that there are 32 non-singleton Wilf classes. There are 317 symmetry classes of triples of 4-letter patterns and after computer calculation of 
initial terms,  the problem reduces to showing that counting sequences that appear to be the same (agree in the first 16 terms) are in fact identical. The insertion encoding algorithm (INSENC) accounts for many of these and some others have been previously counted; in this paper, we find the generating function for each of the remaining 36 triples and it turns out to be algebraic in every case.
Our methods are both combinatorial and analytic, including decompositions by left-right maxima and by initial letters. Sometimes this leads to an algebraic equation for the generating function, sometimes to a functional equation or a multi-index recurrence that succumbs to the kernel method. A particularly nice so-called cell decomposition is used in one case and a bijection is used for another.
\bigskip

\end{abstract}
\maketitle

\section{Introduction}

In recent decades pattern avoidance has received a lot of attention. It has a prehistory in the work of MacMahon \cite{macmahon1915} and Knuth \cite{K} but the paper that really sparked the current interest is by Simion and Schmidt \cite{SiS}. They thoroughly analyzed 3-letter patterns, including a bijection between 123- and 132-avoiding permutations, thereby explaining the first 
(nontrivial) instance of what is, in modern terminology, a Wilf class. Since then the problem has been
addressed on several other discrete structures, such as compositions, $k$-ary words, and set partitions; see, e.g.,
the texts \cite{SHM,TM} and references contained therein.

Permutations avoiding a single 4-letter pattern have been well studied (see, e.g., \cite{St0,St,West}). There are 56 symmetry classes of pairs of 4-letter patterns, for all but 8 of which the avoiders have been enumerated \cite{wikipermpatt}. Le \cite{L} established that these $56$ symmetry classes form $38$ distinct Wilf classes. Vatter \cite{V} showed that of these $38$, $12$ can be enumerated with regular insertion encodings (the INSENC algorithm). Some of these generating functions were computed by hand by Kremer and Shiu \cite{KS}.

Much less is known about larger sets/longer patterns. Here, we consider the 317 symmetry classes of triples of 4-letter patterns and determine their Wilf classes.
First, we used the software of Kuszmaul \cite{Kus} to compute the initial terms $\{|S_n(T)|\}_{n=1}^{16}$ for each symmetry class of $3$ patterns in $S_4$, see Table \ref{long3s}.
There are 242 distinct 16-term sequences among the 317, giving a lower bound of 242 on the number of Wilf classes. 
We will show that whenever such sequences agree in the first 16 terms, they are identical, and so 242 is also an upper bound.  To do so, we find the \gf for every triple whose 16-term counting sequence is repeated in Table \ref{long3s}. Thirty-eight of them can be found by INSENC and some others have already been counted; these are referenced in Table \ref{long3} at the end of this introduction. Here we count the rest, amounting to 36 triples with 15 distinct counting sequences. Table \ref{long3} is a compendium of the results. Summarizing, there are 242 Wilf classes of triples of 4-letter patterns, of which 210 are singleton (trivial) and 32 contain more than one symmetry class (nontrivial). Enumeration of the trivial Wilf classes will be treated in a forthcoming paper.

\begin{theorem}[\bf Main Theorem]
There are exactly $242$ Wilf classes $(32$ nontrivial Wilf classes$)$ of triples of $4$-letter patterns in permutations,
see Tables \ref{long3} and \ref{long3s}.
\end{theorem}

\vspace*{5mm}

{\footnotesize\begin{longtable}[c]{|l|l|l|l|}
\caption{Nontrivial Wilf classes of three 4-letter patterns with numbering taken from Table \ref{long3s}.\label{long3}}\\ \hline
\multicolumn{4}{| c |}{Start of Table}\\ \hline
No. &$T$ & $\sum_{n\geq0}|S_n(T)|x^n$ & Reference\\ \hline
\endfirsthead  \hline
\multicolumn{4}{|c|}{Continuation of Table \ref{long3}}\\ \hline
No. &$T$ & $\sum_{n\geq0}|S_n(T)|x^n$ & Reference \\ \hline
\endhead \hline
\endfoot \hline
\multicolumn{4}{| c |}{End of Table}\\ \hline\hline
\endlastfoot
6&\{1432,2134,3412\}, \{1234,1432,3412\} &$\frac{1-6x+16x^2-22x^3+21x^4-8x^5+2x^6}{(1-x)^7}$&INSENC\\\hline
50&\{2143,3412,2341\}, \{2143,2341,4231\}&&    \\
&\{3412,1432,1243\}&  \raisebox{1.5ex}[0pt]{$\frac{1-6x+13x^2-11x^3+5x^4}{(1-x)^2(1-2x)(1-3x+x^2)}$}    & \raisebox{1.5ex}[0pt]{Thm. \ref{th50A1}, \ref{th50A2}, \ref{th50A3}} \\\hline
55&\{1342,3124,4213\}, \{1234,2143,4123\} & & INSENC\\
 &\{1324,2143,2341\} &  \raisebox{1.5ex}[0pt]{$\frac{1-6x+12x^2-8x^3+3x^4-x^5}{(1-x)(1-3x+x^2)^2}$} & Thm. \ref{th55A3}\\\hline
56&\{1342,2143,4123\}, \{1432,3412,4123\} &$\frac{1-5x+8x^2-5x^3+3x^4}{(1-x)^2(1-4x+3x^2-x^3)}$&INSENC\\\hline
78&\{1324,1342,4312\}, \{1234,2413,3412\} &$\frac{1-8x+26x^2-42x^3+36x^4-14x^5}{(1-x)^3(1-2x)^3}$&INSENC\\\hline
94&\{1432,2134,4132\}, \{1234,1432,4132\}&$\frac{1-8x+25x^2-37x^3+27x^4-9x^5}{(1-3x+x^2)(1-x)^2(1-2x)^2}$&INSENC\\\hline
108&\{1432,3124,4123\}, \{1324,1432,4123\} &$\frac{1-7x+17x^2-15x^3+4x^4-2x^5}{(1-2x)(1-3x+x^2)^2}$&INSENC\\\hline
112&\{1243,1432,2134\}, \{1432,3142,4123\} & ${\frac{1-4x+4x^2-3x^3+x^4}{(1-x+x^2)(1-4x+2x^2)}}$&INSENC\\\hline
126&\{2431,4213,1324\}, \{3142,4123,1234\}&$\frac{1-7x+18x^2-20x^3+11x^4-2x^5}{(1-x)^2(1-3x+x^2)^2}$&INSENC\\\hline
127& \{1342,2413,4312\}, \{1243,2341,2413\}&$\frac{1-8x+24x^2-33x^3+22x^4-7x^5}{(1-x)^3(1-3x)(1-3x+x^2)}$&INSENC\\\hline
129&\{1432,2413,3124\}, \{1432,2143,3124\}&$\frac{(1-x)^2(1-3x+x^2)}{1-6x+12x^2-11x^3+3x^4-x^5}$&INSENC\\\hline
157&\{1324,1342,4213\}, \{1432,3124,4132\} & $\frac{(1-x)(1-6x+11x^2-4x^3)}{(1-2x)(1-3x+x^2)^2}$&INSENC\\\hline
166& \{3412,3142,1243\}, \{3412,3142,1324\}& $\frac{1-9x+30x^2-44x^3+27x^4-7x^5}{(1-3x)(1-x)(1-3x+x^2)^2}$  & Thm. \ref{th166A1}, \ref{th166A2}\\\hline
170&\{3142,4231,4321\}, \{1234,1243,2413\} &$\frac{(1-3x+x^2)^2}{(1-x)(1-2x)(1-4x+2x^2)}$& INSENC\\\hline
171&\{3124,1342,4123\}, \{1342,1324,4123\}&$\frac{1-4x+5x^2-x^3+(1-4x+3x^2-x^3)\sqrt{1-4x}}{(1-x)(1-3x+x^2)(1-2x+\sqrt{1-4x})}$  & Thm. \ref{th171A1}, \ref{th171A2}\\\hline
173&\{1342,1423,4213\}, \{1423,2341,2431\} & $\frac{1-5x+8x^2-7x^3+2x^4}{1-6x+12x^2-13x^3+6x^4-x^5}$&INSENC\\\hline
174&\{1432,3412,3421\}, \{1342,3412,4312\} && INSENC\\
&\{2134,2341,2413\}, \{1342,3142,4312\} &$\frac{1-6x+10x^2-3x^3+x^4}{(1-3x+x^2)(1-4x+2x^2)}$& Thm. \ref{th174A3}, INSENC\\
&\{2143,2314,2341\}, \{2143,2314,2431\} & & Thm. \ref{th174A5}, \ref{th174A6}\\\hline
177&\{2143,2341,2413\}, \{2143,2341,3241\}& $\frac{1-4x+3x^2-x^3}{1-5x+6x^2-3x^3}$   & Thm. \ref{th177A1}, \ref{th177A2}\\\hline
191&\{1342,2134,2413\}, \{1324,1432,3124\} & & Thm. \ref{th191A1}, INSENC\\
&\{1423,2134,2413\}, \{3142,4132,4321\} &  \raisebox{1.5ex}[0pt]{$\frac{(1-x)(1-2x)(1-3x)}{1-7x+16x^2-14x^3+3x^4}$} &INSENC\\\hline
196&\{1342,3142,4213\}, \{1324,1432,2134\} &&INSENC\\
&\{2143,2431,3241\}, \{1342,3142,4123\} &   \raisebox{1.5ex}[0pt]{$\frac{1-5x+7x^2-4x^3}{1-6x+11x^2-9x^3+2x^4}$}& Thm. \ref{th196A3}, \ref{th196A4}\\\hline
201&\{3142,1324,1243\}, \{1342,1423,2314\}& $\frac{1-3x+x^2}{1-x}C(x)^3$   & Thm. \ref{th201A1}, \ref{th201A2}\\\hline
203&\{3142,1432,1324\}, \{3124,1423,1234\}& $\frac{1-x}{2-2x-(1-x-x^2)C(x)}$   & Thm. \ref{th203A1}, \ref{th203A2}\\\hline
215&\{1243,2134,2143\}, \{1234,1243,2143\} &&\\
&\{1423,2314,2413\}, \{1423,1432,4123\} & \raisebox{1.5ex}[0pt]{$\frac{1-4x+2x^2}{(1-x)(1-4x+x^2)}$} & \raisebox{1.5ex}[0pt]{INSENC} \\\hline
218&\{1342,2314,2413\}, \{3142,1324,1423\}&   & \\
&\{3124,1423,1243\}& \raisebox{1.5ex}[0pt]{$\frac{(1-2x)(1+\sqrt{1-4x})}{x^2+(2-4x+x^2)\sqrt{1-4x}}$}   & \raisebox{1.5ex}[0pt]{Thm. \ref{th218A1}, \ref{th218A2}, \ref{th218A3}} \\\hline
221&\{2413,3142,1324\}, \{2143,3142,1324\}& &\\
&\{2143,1324,1423\}, \{3142,4132,1243\}&    &\\
&\{3142,4123,1423\}, \{4132,1432,1243\}&  \raisebox{1.5ex}[0pt]{$1+ \frac{1 - 2 x}{2 (1 - x)} \left(\frac{1}{\sqrt{1 - 4 x}} -1\right)$}    & \raisebox{1.5ex}[0pt]{\cite{CM2}} \\
&\{4132,1342,1324\}&    &\\\hline
229&\{2413,3142, 2341\}, \{2143,1342,1423\}&   &\\
&\{2134,1342,1423\}& \raisebox{1.5ex}[0pt]{$\frac{1-2x+2x^2-\sqrt{1-8x+20x^2-24x^3+16x^4-4x^5}}{2x(1-x+x^2)}$}    & \raisebox{1.5ex}[0pt]{Thm. \ref{th229A1}, \ref{th229A2}, \ref{th229A3}} \\\hline
233&\{2143,1324,1243\}, \{2134,1324,1243\}&&\\
&\{2134,1243,1234\}, \{3142,4132,1432\}&&\\
&\{3142,4132,1342\}, \{3142,4132,1423\}&&\\
&\{3142,1342,1324\}, \{3124,1342,1324\}& \raisebox{1.5ex}[0pt]{$\frac{2(1-4x)}{2-9x+4x^2-x\sqrt{1-4x}}$}  & \raisebox{1.5ex}[0pt]{\cite{CMS}} \\
&\{3124,1324,1423\}, \{4132,1432,1324\}&&\\
&\{4132,4123,1423\}, \{1342,4123,1423\}&&\\\hline
234&\{2143,2413,2314\}, \{3142,1342,1243\}& $\frac{(1-x)^2-\sqrt{(1-x)^4-4x(1-2x)(1-x)}}{2x(1-x)}$& Thm. \ref{th234A1}, \ref{th234A2}\\\hline
235&\{1423,1432,2143\}, \{3142,1432,1423\}& $F_T(x)=1-x+xF_T(x)$  &\\
&\{1234,1243,2314\}& $\qquad+x(1-2x)F_T^2(x)+x^2F_T^3(x)$  & \raisebox{1.5ex}[0pt]{Thm. \ref{th235A1}, \ref{th235A2}, \ref{th235A3}}\\\hline
236 &\{1423,3124,4123\}, \{1342,1432,4132\} &&\\
&\{1324,1423,1432\}, \{1243,1324,1423\}&$\frac{1-5 x+(1+x)\sqrt{1-4x}}{1-5x+(1-x)\sqrt{1-4x}}$&\cite{CM}\\
&\{1234,1243,1423\} &&\\\hline
238&\{1423,2413,3142\}, \{2134,2143,2413\}&&   \\
&\{1342,1423,1234\}, \{1342,1423,1324\}&$\frac{3-2x-\sqrt{1-4x}-\sqrt{2-16x+4x^2+(2+4x)\sqrt{1-4x}}}{2(1-\sqrt{1-4x})}$& Thm. \ref{th238A1}, \ref{th238A2}, \ref{th238A3}, \ref{th238A4}, \ref{th238A5}  \\
&\{1342,1423,1243\}& &\\\hline
239&\{2413,3412,3142\}, \{4312,3412,4132\}&&    \\
&\{3412,3142,1342\}, \{3142,1432,1342\}&&   \\
&\{3142,1342,1423\}, \{3124,1324,1243\}&  \raisebox{1.5ex}[0pt]{$\frac{2}{1 + x + \sqrt{1 - 6x + 5x^2}}$ }  & \raisebox{1.5ex}[0pt]{\cite{MS}}    \\
&\{1432,1423,1243\}, \{1324,1423,1234\}&&   \\
&\{4123,1423,1243\}& &\\\hline
\end{longtable}}

\section{Preliminaries and Notation}
We say a permutation is \emph{standard} if its support set is an initial segment of the positive integers, and for a permutation $\p$ whose support is any set of positive integers, St($\p$) denotes the standard permutation obtained by replacing the smallest entry of $\p$ by 1, next smallest by 2, and so on. As usual, a standard permutation $\p$ \emph{avoids} a standard permutation $\tau$ if there is no subsequence $\rho$ of $\p$ for which St($\rho) =\tau$. In this context, $\tau$ is a pattern, and for a list $T$ of patterns, $S_n(T)$ denotes the set of permutations of $[n]=\{1,2,\dots,n\}$ that avoid all the patterns in $T$.

A permutation has an obvious representation as a matrix diagram,
\begin{center}
\begin{pspicture}(-1,-.5)(3,1.5)
\psset{unit =.5cm, linewidth=.5\pslinewidth}
\psline(0,0)(3,0)(3,1)(0,1)(0,2)(3,2)(3,3)(0,3)(0,0)
\psline(1,0)(1,3)(2,3)(2,0)(3,0)(3,3)
\rput(.5,2.5){\small $\bullet$}
\rput(1.5,0.5){\small $\bullet$}
\rput(2.5,1.5){\small $\bullet$}
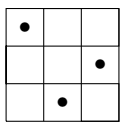
\rput(1.5,-0.8){\small matrix diagram of the permutation 312}
\end{pspicture}
\end{center}
and it will often be convenient to use such diagrams where shaded areas always indicate regions that contain no entries (blank regions may generally contain entries but in a few cases, as noted and clear from the context, they don't).

The eight symmetries of a square, generated by rotation and reflection, partition patterns and sets of patterns into symmetry classes on each of which the counting sequence for avoiders is obviously constant. Thus if $\pi$ avoids $\tau$ then, for example, $\pi^{-1}$ avoids $\tau^{-1}$ since inversion corresponds to flipping the matrix diagram across a diagonal. It sometimes happens (and remarkably often) that different symmetry classes have the same counting sequence, and all symmetry classes with a given counting sequence form a \emph{Wilf class}. Thus Wilf classes correspond to counting sequences.

Throughout, $C(x)=\frac{1-\sqrt{1-4x}}{2x}$ denotes the \gf for the Catalan numbers $C_n:=\frac{1}{n+1}\binom{2n}{n}=\binom{2n}{n}-\binom{2n}{n-1}$. As is well known \cite{wikispecific}, $C(x)$ is the \gf for $(|S_n(\pi)|)_{n\ge 0}$ where $\pi$ is any one of the six 3-letter patterns. Occasionally, we need the \gf for avoiders of a 3-letter and a 4-letter pattern; see \cite{wikispecific} for a comprehensive list.  

A permutation $\pi$ expressed as $\pi=i_1\pi^{(1)}i_2\pi^{(2)} \cdots i_m\pi^{(m)}$ where $i_1<i_2<\cdots<i_m$ and $i_j>\max(\pi^{(j)})$ for $1 \leq j \leq m$ is said to have $m$ left-right maxima (at $i_1,i_2,\ldots,i_m$).  Given nonempty sets of numbers $S$ and $T$, we will write $S<T$ to mean $\max(S)<\min(T)$ (with the inequality vacuously holding if $S$ or $T$ is empty).  In this context, we will often denote singleton sets simply by the element in question. Also, for a number $k$, $S-k$ means the set $\{s-k:s\in S\}$. An ascent in $\pi$ is a pair of adjacent increasing entries $\pi_j<\pi_{j+1}$, thus 413625 has 3 ascents, 13, 36, and 25.

Our approach is ultimately recursive. In each case, we examine the structure of an avoider, 
usually by splitting the class of avoiders under consideration into subclasses according to a judicious choice of parameters which may involve, for example, left-right maxima, initial letters, ascents, and whether resulting subpermutations are empty or not. The choice is made so that each member of a subclass can be decomposed into independent parts. The \gf for the subclass (a summand of the full \gf) is then the product of the generating functions for the parts, and we speak of the ``contribution'' of the various parts to the \gf for that subclass.
For Case 238, we use a cell decomposition, described in that subsection. From the structure, we are able to find an equation for the \gf $F_T(x):= \sum_{n\ge 0}|S_n(T)|x^n$, where $T$ is the triple under consideration. This equation is often algebraic and, if linear or quadratic, as it is in all but one case, easy to solve explicitly (the exception is the cubic equation for the second triple in Case 235). It is also often a functional equation requiring the kernel method (see, e.g., \cite{HM} for an exposition). Exceptionally, for one of the triples in Case 50, we use a bijection. In every case, the \gf turns out to be algebraic.

Furthermore, in several cases, especially those where recurrences are made use of, we have in fact counted members of the avoidance class in question according to the distribution of one or more statistics, specific to the class, and have assumed particular values of the parameters to obtain the avoidance result.  In some of these cases, to aid in solving the recurrence, certain auxiliary arrays related to the statistic are introduced.  This leads to systems of linear functional equations to which we apply the kernel method, adapted for a system.  See, for example, the proofs below of the second symmetry class in Case 171 and of the first class in Case 235.  Also, in cases where the kernel method is used, it is usually possible (if desired) to solve the functional equation in its full generality yielding a polynomial generalization of the avoidance result. In other instances, one may extend the result by counting members of the class in question having a fixed number of left-right maxima.  We refer the reader to the discussion following the proof of the first class in Case 50 below.

We now proceed to find all the required generating functions in the 15 cases of repeated counting sequence.

\section{Proofs}
\subsection{Case 50}
The three representative triples $T$ are:

\{2143, 2341, 3412\} (Theorem \ref{th50A1})

\{2143, 2341, 4231\} (Theorem \ref{th50A2})

\{2143, 2341, 3421\} (Theorem \ref{th50A3})

In order to deal with this case, we define the following two generating functions for each triple $T$: $H_T(x)$ is the generating
function for $T$-avoiders with first letter $n-1$ and $J_T(x)$ is the generating function for $T$-avoiders with second letter $n$.

\begin{lemma}\label{lem50}
For each triple $T$ in Case 50, $H_T=J_T.$
\end{lemma}
\begin{proof}
For each pattern in case 50, its matrix diagram is invariant under the involution ``flip in the diagonal line $y=-x$''. Consequently, the set of $T$-avoiders is invariant under this flip. But the flip interchanges the permutations whose first letter is $n-1$ and the permutations whose second letter is $n$.
\end{proof}

\begin{theorem}\label{th50A1}
Let $T=\{2143,2341,3412\}$. Then
$$F_T(x)=\frac{1-6x+13x^2-11x^3+5x^4}{(1-x)^2(1-2x)(1-3x+x^2)}.$$
\end{theorem}
\begin{proof}
Let $G_m(x)$ be the generating function for $T$-avoiders with $m$
left-right maxima. Clearly, $G_0(x)=1$ and $G_1(x)=xF_T(x)$.
For $G_m(x)$ with $m\geq2$, we first need
an equation for $J_T(x)$. Consider a permutation $\pi=in\pi'\in S_n(T)$ counted by $J_T$.
Clearly, the contribution for the case $i=n-1$ is given by
$\frac{x^2}{1-x}$.
If $i=1\ne n-1$, the contribution is $x^2(F_T(x)-1)$. Otherwise, $n \ge 4$
and $1<i<n-1$ and $\pi$ can be written as $in\beta'(n-1)\beta''$ with at least one of
$\beta', \beta''$
nonempty. Consider three cases: (1) $\beta'$ is empty, (2) $\beta''$ is empty,
(3) neither of $\beta', \beta''$ is empty. In each of cases 1 and 2,
the map ``delete $n-1$ and standardize'' is a bijection to the
one-size-smaller permutations counted by $J_T$ that do not start with a
1. Hence, in each of these cases, the contribution is $x\big(J_T(x)-x^2
F_T(x)\big)$. In case 3, $\pi \backslash\{n-1\}$ must have the form $n-2\ n\
n-3\ n-4\ \dots 2\ 1$ with $n-4$ positions available for $n-1$, namely, immediately before $1,2,\dots,n-4$. Hence
the contribution in this case is $x^5/(1-x)^2$.

Adding all the contributions, and solving for $J_T(x)$, yields
\begin{equation}\label{Jeqn50A1}
J_T(x)=x^2F_T(x)+\frac{x^3(1-x+x^2)}{(1-x)^2(1-2x)}\, .
\end{equation}

Now let $m\geq2$ and let us write an equation for $G_m(x)$.
Let $\pi=i_1\pi^{(1)}i_2\pi^{(2)}\cdots i_m\pi^{(m)}\in S_n(T)$ with exactly $m$ left-right maxima. By considering the cases where $\pi^{(1)}$ is not empty or where $\pi^{(1)}$ is empty and $\pi^{(2)}$ either has a letter smaller than $i_1$ or it doesn't (see the next figure), we obtain
\begin{center}
\begin{pspicture}(-0,-1)(12,3.5)
\psset{unit =.8cm, linewidth=.5\pslinewidth}
\psline(0,0)(4,0)(4,1)(0,1)(0,0)\psline(1,0)(4,0)(4,2)(1,2)(1,0)
\psline(1,0)(1,1)\psline(2,0)(2,1)\psline(3,0)(3,1)
\psline[fillstyle=hlines,hatchcolor=lightgray,hatchsep=0.8pt](1,1)(4,1)(4,2)(1,2)(1,1)
\psline[fillstyle=hlines,hatchcolor=lightgray,hatchsep=0.8pt](2,2)(4,2)(4,3)(2,3)(2,2)
\psline[fillstyle=hlines,hatchcolor=lightgray,hatchsep=0.8pt](3,3)(4,3)(4,4)(3,4)(3,3)
\psline[fillstyle=hlines,hatchcolor=lightgray,hatchsep=0.8pt](2,0)(4,0)(4,1)(2,1)(2,0)
\pscircle*(0,1){0.08}\pscircle*(1,2){0.08}\pscircle*(2,3){0.08}\pscircle*(3,4){0.08}
\put(.2,.3){\small$\neq\emptyset$}
\put(0,-.8){\small$x^{m-2}\left(H_T(x)-\frac{x^2}{1-x}\right)$}
\put(5,0){\psset{unit =.8cm, linewidth=.5\pslinewidth}
\psline(0,0)(4,0)(4,1)(0,1)(0,0)\psline(1,0)(4,0)(4,2)(1,2)(1,0)\psline(2,0)(4,0)(4,3)(2,3)(2,0)
\psline(3,0)(4,0)(4,4)(3,4)(3,0)\pscircle*(0,1){0.08}\pscircle*(1,2){0.08}
\pscircle*(2,3){0.08}\pscircle*(3,4){0.08}
\psline[fillstyle=hlines,hatchcolor=lightgray,hatchsep=0.8pt](2,1)(4,1)(4,2)(2,2)(2,1)
\psline[fillstyle=hlines,hatchcolor=lightgray,hatchsep=0.8pt](2,2)(4,2)(4,3)(2,3)(2,2)
\psline[fillstyle=hlines,hatchcolor=lightgray,hatchsep=0.8pt](3,3)(4,3)(4,4)(3,4)(3,3)
\psline[fillstyle=hlines,hatchcolor=lightgray,hatchsep=0.8pt](2,0)(4,0)(4,1)(2,1)(2,0)
\psline[fillstyle=hlines,hatchcolor=lightgray,hatchsep=0.8pt](0,0)(1,0)(1,1)(0,1)(0,0)
\put(1.2,.3){\small$\neq\emptyset$}
\put(0,-.8){\small$x^{m-2}\big(J_T(x)-x^2F(x)\big)$}}
\put(10,0){\psset{unit =.8cm, linewidth=.5\pslinewidth}
\psline(0,0)(4,0)(4,1)(0,1)(0,0)\psline(1,0)(4,0)(4,2)(1,2)(1,0)\psline(2,0)(4,0)(4,3)(2,3)(2,0)
\psline(3,0)(4,0)(4,4)(3,4)(3,0)\pscircle*(0,1){0.08}\pscircle*(1,2){0.08}
\pscircle*(2,3){0.08}\pscircle*(3,4){0.08}
\psline[fillstyle=hlines,hatchcolor=lightgray,hatchsep=0.8pt](0,0)(4,0)(4,1)(0,1)(0,0)
\psline[fillstyle=hlines,hatchcolor=lightgray,hatchsep=0.8pt](3,0)(4,0)(4,2)(3,2)(3,0)
\put(.8,-.8){\small$xG_{m-1}(x)$}}
\end{pspicture}
\end{center}
\begin{equation}\label{GforThm2}
G_m(x)=x^{m-2}\left(H_T(x)-\frac{x^2}{1-x}\right)+x^{m-2}\big(J_T(x)-x^2F_T(x)\big)+xG_{m-1}(x)
\end{equation}
for $m\ge 2$.
By summing  (\ref{GforThm2}) over $m\geq2$ and using the expressions for $G_0(x)$ and $G_1(x)$, we obtain
\begin{equation}\label{Feqn50A1}
F_T(x)=1+\frac{1}{(1-x)^2}\left(H_T(x)-x^2/(1-x)+J_T(x)+x(1-2x)F_T(x)\right)\, .
\end{equation}
Eliminating $J_T=H_T$ from (\ref{Jeqn50A1}) and (\ref{Feqn50A1}) gives the desired expression for $F_T$.
\end{proof}

It is possible to generalize the preceding result as follows.  Define the generating function $G(x,q)=G_T(x,q)=\sum_{m\geq0}G_m(x)q^m$, where $T=\{2143,2341,3412\}$.  Multiplying both sides of \eqref{GforThm2} by $q^m$, and summing over $m\geq2$, implies
$$G(x,q)-1-xqF_T(x)=\frac{q^2}{1-xq}\left(2J_T(x)-x^2F_T(x)-\frac{x^2}{1-x}\right)+xq(G(x,q)-1).$$
Solving for $G(x,q)$ then yields
\begin{align*}
G(x,q)&=1+\frac{q}{(1-xq)^2}\left(xF_T(x)+\frac{2x^3q(1-x+x^2)}{(1-x)^2(1-2x)}-\frac{x^2q}{1-x}\right)\\
&=1+\frac{xq(1-6x+13x^2-11x^3+5x^4-xq(1-8x+20x^2-19x^3+10x^4-2x^5))}{(1-x)^2(1-xq)^2(1-2x)(1-3x+x^2)}.
\end{align*}
Taking $q=1$ in the last formula, and simplifying, recovers Theorem \ref{th50A1}:
$$G_T(x,1)=F_T(x)=\frac{1-6x+13x^2-11x^3+5x^4}{(1-x)^2(1-2x)(1-3x+x^2)}.$$
Extracting the coefficient of $q^m$ from $G(x,q)$ implies the generating function $G_m(x)$ for the number of $T$-avoiders having exactly $m$ left-right maxima is given by
$$G_m(x)=\frac{x^m(1-8x+20x^2-19x^3+10x^4-2x^5)+mx^{m+1}(2-7x+8x^2-5x^3+2x^4)}{(1-x)^2(1-2x)(1-3x+x^2)}, \quad m \geq1.$$
Note that taking $m=1$ in the last equation gives back the obvious formula $G_1(x)=xF_T(x)$.

\emph{Remark:} A comparable formula for $G_T(x,q)$ may be obtained for other $T$ in this paper where left-right maxima are made use of in counting the avoiders in question.

We now turn our attention to the case when $T=\{2143,2341,4231\}$.

\begin{theorem}\label{th50A2}
Let $T=\{2143,2341,4231\}$. Then
$$F_T(x)=\frac{1-6x+13x^2-11x^3+5x^4}{(1-x)^2(1-2x)(1-3x+x^2)}.$$
\end{theorem}
\begin{proof}
The proof follows along the lines of Theorem \ref{th50A1}.
Let $G_m(x)$ be the generating function for $T$-avoiders with $m$ left-right maxima.
Clearly, $G_0(x)=1$ and $G_1(x)=xK$, where $$K=\sum_{n\geq0}|S_n(231,2143)|x^n=\frac{1-2x}{1-3x+x^2}$$ (see \cite[Seq. A001519]{Sl}).

To write an equation for $H_T(x)$, we consider
$\pi=(n-1)\pi'n\pi''\in S_n(T)$. If $\pi'\pi''$ is empty, then the contribution is $x^2$.
Otherwise, we consider the following two cases:
\begin{itemize}
\item the letter $n-2$ belongs to $\pi'$, which implies that $\pi=(n-1)\beta'(n-2)\beta''n
\pi''$.
If $\beta''\pi''$ is empty, we get a contribution of $x^3K$
by deleting the letters $n-1,\,n-2,\,n$ from $\pi$.
If $\beta''\pi''$ is nonempty, then $\beta'<\beta''\pi''$ and we get a contribution of
$\frac{x}{1-x}\big(H_T(x)-x^2\big)$.

\item the letter $n-2$ belongs to $\pi''$, which implies $\pi=(n-1)\pi'n\beta'(n-2)\beta''$. If $\beta''$ is not  empty, then $\pi'n\beta'=12\cdots jn(j+1)(j+2)\cdots i'$ where $\beta''$ is a permutation of $i'+1,\ldots,n-3$, so we have a contribution of $\frac{x^3}{(1-x)^2}(K-1)$. If $\beta''$ is empty, then $\pi'$ is an increasing subsequence, say $\pi'=j_1j_2\cdots j_d$ with $j_1<j_2<\cdots<j_d$. If $d=0$, then the contribution is $x^3K$, otherwise $\pi$ can be written as
    $\pi=(n-1)j_1j_2\cdots j_dn\beta'(n-2)$ with $d\geq1$. Since $\pi$ avoids $2341$ and
    $4231$, we have
    $\pi=(n-1)12\cdots(d-1)j_dn\beta^{(1)}\beta^{(2)}(n-2)$, where $j_d\beta^{(1)}
    \beta^{(2)}$ is a permutation of $d,d+1,\ldots,n-3$ and $\beta^{(1)}<j_d<\beta^{(2)}$.
    By considering whether $\beta^{(1)}$ is empty or not, we get
    $x^{d+3}K+\frac{x^{d+3}}{1-x}(K-1)$. By summing over $d\geq1$, it follows that the contribution in this case is given by $\frac{x^4}{1-x}K+\frac{x^4}{(1-x)^2}(K-1)$.
\end{itemize}
Thus, $H_T(x)=x^2+x^3K+\frac{x}{1-x}(H_T(x)-x^2)+\frac{x^3}{(1-x)^2}(K-1)+x^3K+\frac{x^4}{1-x}K+\frac{x^4}{(1-x)^2}(K-1)$, which implies
$$H_T(x)=\frac{(1-3x+3x^2+x^3)x^2}{(1-2x)(1-3x+x^2)}.$$

Now, we are ready to write an equation for $G_m(x)$, where $m\geq2$. Using the same
decomposition as in the proof of Theorem \ref{th50A1}, we obtain
\begin{eqnarray*}
G_m(x) & = & x^{m-2}(H_T(x)-x^2K)+x^{m-2}(J_T(x)-x^2K)+xG_{m-1}(x) \\
& = & 2x^{m-2}H_T(x) - 2 x^m K +xG_{m-1}(x).
\end{eqnarray*}
Summing over $m\geq2$, we find $\sum_{m\ge2}G_m(x)=(2H_T(x)-x^2(1+x)K)/(1-x)^2$.  Using the
expressions for $G_0(x)$ and $G_1(x)$, it is seen that  $F_T(x)=\sum_{m\ge 0}G_m(x)$
simplifies to the desired expression.
\end{proof}

\begin{theorem}\label{th50A3}
There is a bijection between the set of $\{2143,2341,3412\}$-avoiders and the set of $\{2143,2341,3421\}$-avoiders.
\end{theorem}
\begin{proof}
Let $A_n$ and $B_n$ denote the subsets of $S_n$ whose members
avoid the patterns in the sets $\{2143,2341,3412\}$ and
$\{2143,2341,3421\}$, respectively.  We will define a bijection
$f$ from $A_n$ to $B_n$ as follows.  If $n\leq 3$, we may
clearly take $f$ to be the identity, so assume $n\geq 4$.  Given
$\pi= \pi_1 \pi_2 \cdots \pi_n \in A_n$, let $a=\pi_1$ denote the first letter of $\pi$.
If $a=1$ so that $\pi$ has the form $1\pi'$, define $f$ recursively by
$f(\pi)=1\oplus f(\pi'-1)$. (Here, $\oplus$ is the direct sum, thus $132\oplus 123=132456$.) 
Henceforth, assume $a>1$. We consider cases according to descending values of the position $j$ of $n$ in $\pi$.

If $j=n$ so that $\pi$ has the form $\pi'n$, define $f$, again recursively, by
$f(\pi)=f(\pi')n$.

If $3 \le j \le n-1$, then $\pi$ has the form $a \beta' n \beta''$ with  $\beta'$
and $\beta''$ nonempty. We claim  (i) $a=n-1$, (ii) $\beta'$ avoids 231,
and (iii) $\beta''$ is decreasing.
To establish the claims, we need a preliminary result.

\begin{lemma}
If $\pi \in A_n$ has $\ge 3$ left-right maxima and $1$ is not its first letter, then $n$ is its last letter.
\end{lemma}
\begin{proof}
Suppose $i_1=a,i_2,\dots,i_m=n$ are the left-right maxima of $\pi$ with $m\ge 3$, so that all other entries of $\pi$ lie in the rectangles $R_2,R_3,\dots,R_m$ and $S_1,S_2$ shown in the figure.
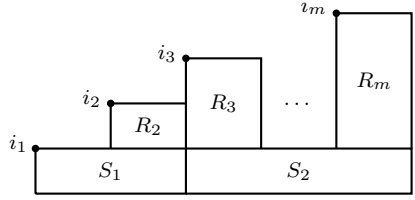
\begin{center}
\begin{pspicture}(0,0)(6,3)
\psset{xunit=.5cm,yunit=.3cm}
\psline(0,0)(0,2)(4,2)(4,0)(0,0)
\psline(2,2)(2,4)(4,4)(4,2)
\psline(4,4)(4,6)(6,6)(6,2)
\psline(8,2)(8,8)(10,8)(10,2)
\psline(4,2)(10,2)(10,0)(4,0)
\rput(-0.45,2.2){\footnotesize{$i_1$}}
\rput(1.5,4.2){\footnotesize{$i_2$}}
\rput(3.5,6.2){\footnotesize{$i_3$}}
\rput(7.4,8.3){\footnotesize{$i_m$}}
\pscircle*(0,2){0.05}
\pscircle*(2,4){0.05}
\pscircle*(4,6){0.05}
\pscircle*(8,8){0.05}
\rput(2,1){\footnotesize{$S_1$}}
\rput(7,1){\footnotesize{$S_2$}}
\rput(3,3){\footnotesize{$R_2$}}
\rput(5,4){\footnotesize{$R_3$}}
\rput(9,5){\footnotesize{$R_m$}}
\rput(7,4){\footnotesize{$\dots$}}
\end{pspicture}
\end{center}
An entry $x\in S_2$ implies $i_1 i_2 i_3 x$ is a forbidden 2341. Hence $S_2=\emptyset$ and so $1 \in S_1$. Now $x \in R_m$ implies $i_1 1 i_m x$ is a forbidden 2143. Hence, $R_m=\emptyset$ and $i_m=n$ is the last letter in $\pi$.
\end{proof}

\begin{corollary}\label{corbij}
If $\pi \in A_n$ and $1$ is not its first letter and $n$ is not its last letter, then $\pi$ has at most two left-right maxima. \qed
\end{corollary}

Now, if claim (i) fails, $a\le n-2$. Since $n$ is not the last letter of $\pi$, Corollary \ref{corbij} implies that $a$ and $n$ are the only two left-right maxima of $\pi$. Consequently, $n-1$ occurs after $n$ and, since $j\ge 3$, there is a letter $x<a$ before $n$. But then $axn(n-1)$ is a forbidden 2143.

If (ii) fails and $cdb$ is a 231 in $\beta'$, take
any $x\in \beta''$. If $x>c$, then $cbnx$ is a forbidden 2143, while if $x<c$,
then $cdnx$ is a forbidden 2341.

If (iii) fails, there are letters $x<y$ in $\beta''$. But then $(n-1)nxy$ is a forbidden 3412.

It follows from (ii) that the initial segment $a \beta' n$ of $\pi$ avoids 3421, since 3421 contains the pattern 231.  Define $f(\pi) =  a \beta' n r(\beta'')$,
where $r(\beta)$ denotes the reverse of a permutation $\beta$. It is clear that
$f(\pi)$ also avoids 3421 and so is in $B_n$.

If $j=2$ so that $\pi$ begins $a n\cdots$, we introduce the condition
\begin{equation}\label{condbij}
\textrm{\emph{ $a+1 \ne \pi_n$  and the successor and predecessor of $a+1$ in $\pi$ are both $<a$,}}
\end{equation}
and consider two cases.
\begin{itemize}
\item Condition (\ref{condbij}) holds. Here, $2\le a \le n-2$ and $\pi$ must have the form
$a\alpha'\beta' (a+1) \beta'' \alpha''$ with $\beta'=(a-1)(a-2)\cdots(b+1)$ and $\beta''=b(b-1)\cdots 1$
for some $b\in[a-2]$, where $\alpha'$ starts with $n$ and is decreasing,
and $\alpha''$ is increasing, possibly empty. Conversely, one may verify that a permutation with a decomposition of this form belongs to $A_n$.  Define $f(\pi) = a \alpha' r(\beta'') (a+1) r(\beta') \alpha''$.

\item Condition (\ref{condbij}) fails. Here, $\pi$ must have the form
$a\alpha'\beta \alpha''$ with $\beta=(a-1)(a-2)\cdots 1$ (hence, nonempty), where $\alpha'$ starts with $n$ and is decreasing, and $\alpha''$ is increasing, possibly empty. Define $f(\pi) = a \alpha' r(\beta) \alpha''$.
\end{itemize}
Finally, if $j=1$ so that $\pi=n\pi'$, define $f$ recursively by $f(\pi)=nf(\pi')$.

The mapping $f$ preserves left-right maxima and their positions. The reader may check that $f$ is reversible and is a bijection from $A_n$ to $B_n$.
\end{proof}

\subsection{Case 55}
\begin{theorem}\label{th55A3}
Let $T=\{1324, 2143, 2341\}$. Then
$$F_T(x)=\frac{1-6x+12x^2-8x^3+3x^4-x^5}{(1-x)(1-3x+x^2)^2}.$$
\end{theorem}
\begin{proof}
Let $a_T(n;i_1,i_2,\ldots,i_s)$ be the number of permutations $i_1i_2\cdots i_s\pi\in S_n(T)$. The initial conditions $a_T(n;n)=a_T(n;n-1)=a_T(n-1)$ and $a_T(n;1)=|S_{n-1}(213,2341)|$ easily follow from the definitions. It is well known that $|S_{n-1}(213,2341)|=F_{2n-3}$ ($F_n$ is the $n$th Fibonacci number defined by $F_n=F_{n-1}+F_{n-2}$, with $F_0=0$ and $F_1=1$).

Now let $2\le i \le n-2$ and let us focus on the second letter of $\pi$. Clearly, $a_T(n;i,n)=a_T(n-1;i)$, and also $a_T(n;i,j)=0$ for all $j=i+2,i+3,\ldots,n-1$.
Note that any permutation $\pi=i(i+1)\pi'\in S_n(T)$ can be written as
$\pi=i(i+1)\pi''(i+2)(i+3)\cdots n$, so $a_T(n;i,i+1)=|S_{i-1}(132,2341)|$
and it is also known that $|S_{i-1}(132,2341)| =F_{2i-3}$.
A permutation $\pi=ij\pi'\in S_n(T)$ with $n-2\geq i>j\geq1$ satisfies
$\pi=ij\pi'(j+2)\cdots(i-1)(i+1)(i+2)\cdots n$, where $\pi'$ is a permutation of $1,2,\ldots,j-1,j+1$ that avoids $132$ and $2341$, and so $a_T(n;i,j)=|S_j(132,2341)|=F_{2j-1}$. Hence
$$a_T(n;i)=F_1+F_3+F_5+\cdots+F_{2i-3}+F_{2i-3}+a_T(n-1;i),$$
which, by the fact that $F_1+F_3+F_5+\cdots+F_{2i-3}=F_{2i-2}$, implies for $2\le i \le n-2$ that
$$a_T(n;i)-a_T(n-1;i)=F_{2i-1}.$$
By summing both sides of the last equation over $i=2,3,\ldots,n-2$ and using the initial conditions, we obtain for $n \geq3$,
\begin{align*}
a_T(n)-a_T(n-1) &=  2a_T(n-1)-a_T(n-2)+F_{2n-3}-F_{2n-5}+\sum_{i=2}^{n-2}F_{2i-1} \\
&=  2a_T(n-1)-a_T(n-2)+F_{2n-3}-F_{2n-5}+F_{2n-4}-1 \\
&=  2a_T(n-1)-a_T(n-2)+2F_{2n-4}-1.
\end{align*}
It is routine to solve this difference equation for the generating function $F_T(x)$ using
$$\sum_{n\geq0}F_{2n-1}x^n=\frac{1-2x}{1-3x+x^2}.$$
\end{proof}

\subsection{Case 166}
The two representative triples $T$ are:

\{1243,3142,3412\} (Theorem \ref{th166A1})

\{1324,3142,3412\} (Theorem \ref{th166A2})

\subsubsection{$\mathbf{T=\{1243,3142,3412\}}$}
Let $b_n(i)=|S_T(n;i,n)|$ and $b_n(i,j)=|S_T(n;i,n,j)|$ so that $b_n(i)$ and $b_n(i,j)$ count $T$-avoiders for which the second letter is $n$ by first letter, $i$, and third letter, $j$. We first obtain a recurrence for $b_n(i,j)$.

Note that when the second letter is $n$, the conditions $i=1$ and $j \in[2,n-2]$ force the last entry to be $n-1$ or 2 for else a forbidden pattern will be present.
Deleting this entry gives contributions of $b_{n-1}(1,j),\ b_{n-1}(1,j-1)$ to $b_n(1,j)$  according as
the last entry is $n-1$ or 2. Clearly, $b_n(1,n-1)=b_{n-1}(1)$.
Hence, for $i=1$, we have the following table of recursive values for $b_n(1,j)$:
\[
\begin{array}{c|cc}
j & \in[2,n-2] & n-1 \\[1mm] \hline
{\rule[-2mm]{0mm}{6mm}b_n(1,j)} & b_{n-1}(1,j)+b_{n-1}(1,j-1) & b_{n-1}(1)
\end{array}
\]

Similarly, the conditions $2\le i \le n-2$ and $j \in[i+1,n-2]$ force the last entry to be $n-1$ or 1 and deleting it gives contributions of $b_{n-1}(i,j),\ b_{n-1}(i-1,j-1)$ according as the last entry is $n-1$ or 1.
Hence, for $2\le i \le n-2$, we have the table (the straightforward verification of entries other than $j \in[i+1,n-2]$ is left to the reader):
\[
\begin{array}{c||c|c|c|c}
j & \in[1,i-2] & i-1 & \in [i+1,n-2] & n-1\\[1mm] \hline
{\rule[-2mm]{0mm}{6mm}b_n(i,j)} & 0 & b_{n-1}(i-1) & b_{n-1}(i,j)+b_{n-1}(i-1,j-1) & b_{n-1}(i)
\end{array}
\]

Lastly, for $i=n-1$, $b_n(n-1)=1$ because $\pi=(n-1)n\pi'\in S_n(T)$ implies $\pi=(n-1)n(n-2)(n-3)\cdots 1$.

The table for $i=1$ yields $b_n(1)=\sum_{j=2}^{n-1}b_n(1,j) = \sum_{j=2}^{n-2}\big(b_{n-1}(1,j)+b_{n-1}(1,j-1)\big)+ b_{n}(1,n-1) = 3b_{n-1}(1)-b_{n-2}(1)$ for $n\ge 4$. Together with the initial conditions $b_2=b_3=1$, this recurrence implies that $b_n(1)=F_{2n-5}$ for $n\ge 2$, where $F_n$ is the $n$th Fibonacci number (defined by $F_n=F_{n-1}+F_{n-2}$ for all $n$, with $F_0=0$ and $F_1=1$).

\begin{lemma}\label{lem166A1x2}
Let $b_n=\sum_{i=1}^{n-1}|S_T(n;i,n)|$. Then the generating function
for the sequence $b_n$ is given by
$$B(x)=\sum_{n\geq2}b_nx^n=\frac{x^2(1-5x+7x^2-x^3)}{(1-3x+x^2)(1-3x)(1-x)}.$$
\end{lemma}
\begin{proof}
Clearly, $b_n=\sum_{i=1}^{n-1}b_n(i)$. Using the preceding results, we have
\begin{align*}
b_n&= 2b_{n-1}+\sum_{i=2}^{n-3}\sum_{j=i+1}^{n-2}b_n(i,j)+\sum_{j=2}^{n-2}b_n(1,j) \\
&=2b_{n-1}+\sum_{i=1}^{n-4}\sum_{j=i+1}^{n-3}b_{n-1}(i,j)
+\sum_{i=2}^{n-3}\sum_{j=i+1}^{n-2}b_{n-1}(i,j)+\sum_{j=2}^{n-2}b_n(1,j)\\
&=2b_{n-1}+b_{n-1}-2b_{n-2}+b_{n-1}-b_{n-2}+F_{2n-7}-F_{2n-7}\\
&=4b_{n-1}-3b_{n-2}+F_{2n-8},
\end{align*}
with $b_1=0$, $b_2=1$ and $b_3=2$. Hence, by using the fact that $\sum_{n\geq4}F_{2n-8}x^n=\frac{x^5}{1-3x+x^2}$, we obtain $B(x)=\frac{x^2(1-5x+7x^2-x^3)}{(1-3x+x^2)(1-3x)(1-x)}$.
\end{proof}

\begin{theorem}\label{th166A1} We have
$$F_T(x)=\frac{1-9x+30x^2-44x^3+27x^4-7x^5}{(1-3x)(1-x)(1-3x+x^2)^2}.$$
\end{theorem}
\begin{proof}
Let $G_m(x)$ be the generating function for $T$-avoiders with $m$ left-right maxima.
Clearly, $G_0(x)=1$ and $G_1(x)=xF_T(x)$. Now let us write an equation for $G_m(x)$ with $m=2$.
Let $\pi=i\pi'n\pi''\in S_n(T)$ with two left-right maxima. If $\pi=i(i-1)\cdots i'n\pi''$, then
the contribution is $B(x)/(1-x)$ (see Lemma \ref{lem166A1x2}). Otherwise, $\pi=i\pi'n\pi''$ where
$\pi'$ is a permutation on $\{i',i'+1,\ldots,i-1\}$ that avoids $T$ and has at least one ascent. Since $\pi$ avoids $1243$, $\pi''=(i'-1)\cdots21$. Thus, the contribution is given by $\frac{x^2}{1-x}\big(F_T(x)-1/(1-x)\big)$. Hence,
$$G_2(x)=\frac{1}{1-x}B(x)+\frac{x^2}{1-x}\left(F_T(x)-\frac{1}{1-x}\right).$$

Now let us write an equation for $G_m(x)$ with $m\geq3$. Let $\pi=i_1\pi^{(1)}\cdots i_m\pi^{(m)}\in
S_n(T)$ with exactly $m$ left-right maxima, $i_1,i_2,\dots,i_m$. Since $\pi$ avoids $1243$, we have
that $\pi^{(2)},\ldots,\pi^{(m)}$ are all $<i_2$.
If $\pi^{(1)}$ has an ascent, then $\pi^{(1)}>\pi^{(2)}>\cdots>\pi^{(m)}$ and $\pi^{(1)}$
avoids $T$ and $\pi^{(2)}\cdots\pi^{(m)}=(i'-1)\cdots1$ with $i'$ the minimal letter of $
\pi^{(1)}$. Thus, the contribution is given by $\frac{x^{m}}{(1-x)^{m-1}}\big(F_T(x)-1/(1-x)\big)$.

From now, we can assume that $\pi^{(1)}=(i-1)\cdots(i'+1)i'$. Let $s\in\{2,3,\ldots,m\}$ be the minimal
number such that $\pi^{(s)}$ contains a letter from the set $[i'-1]$. We have the following cases:
\begin{itemize}
\item $s=2$: Here $\pi'^{(2)}\pi^{(3)}\cdots\pi^{(m)}=(i'-1)\cdots21$, where $\pi'^{(2)}$ is the
subsequence of all letter of $\pi^{(2)}$ that are smaller than $i'$. Hence, by the definition of $B(x)$
(see Lemma \ref{lem166A1x2}), we get a contribution of $\frac{x^{m-2}}{(1-x)^{m-1}}\big(B(x)-x^2K(x)\big)$,
where $K(x)=\sum_{n\geq0}|S_n(132,3412)|x^n=\frac{1-2x}{1-3x+x^2}$ (see \cite[Seq. A001519]{Sl}).

\item $s=3,4,\ldots,m-1$: Here $\pi^{(2)}>\pi^{(3)}>\cdots>\pi^{(s-1)}>i_1>\pi^{(s)}>\pi^{(s+1)}>\cdots>\pi^{(m)}$ and  $\pi^{(s)}\cdots\pi^{(m)}=(i'-1)\cdots21$ and $\pi^{(3)}\cdots\pi^{(s-1)}=(i'_1-1)\cdots(i_1+2)(i_1+1)$, where $i'_1$ is the minimal letter of $\pi^{(2)}$. Moreover, $\pi^{(2)}$ avoids $132$ and $3412$. Thus, the contribution is given by $\frac{x^{m+1}}{(1-x)^m}K(x)$.

\item $s=m$. Here $i'=1$ and $\pi^{(2)}>\pi^{(3)}>\cdots>\pi^{(m)}>i_1$ and $\pi^{(3)}\cdots\pi^{(s-1)}=(i'_1-1)\cdots(i_1+2)(i_1+1)$ and $\pi^{(2)}$ avoids $132$ and $3412$, where $i'_1$ is the minimal letter of $\pi^{(2)}$. Thus, the contribution is given by $\frac{x^{m}}{(1-x)^m}K(x)$.
\end{itemize}
By the preceding cases, we obtain for $m\ge 3$,
\begin{align*}
G_m(x)&=\frac{x^m}{(1-x)^{m-1}}\left(F_T(x)-\frac{1}{1-x}\right)+\frac{x^{m-2}}{(1-x)^{m-1}}\left(B(x)-x^2K(x)\right)\\
&\quad+(m-3)\frac{x^{m+1}}{(1-x)^m}K(x)+\frac{x^{m}}{(1-x)^m}K(x).
\end{align*}
Since $F_T(x)=\sum_{m\geq0}G_m(x)$, we have
\begin{align*}
F_T(x)&=1+xF_T(x)+\frac{1}{1-x}B(x)+\sum_{m\geq2}\frac{x^{m}}{(1-x)^{m-1}}\left(F_T(x)-\frac{1}{1-x}\right)\\
&\quad+\sum_{m\geq3}\frac{x^{m-2}}{(1-x)^{m-1}}(B(x)-x^2K(x))+\sum_{m\geq4}(m-3)\frac{x^{m+1}}{(1-x)^m}K(x)\\
&\quad+\sum_{m\geq3}\frac{x^{m}}{(1-x)^m}K(x).
\end{align*}
After several algebraic operations, and solving for $F_T(x)$, we complete the proof.
\end{proof}

\subsubsection{$\mathbf{T=\{1324,3142,3412\}}$}
Let $b_n(i,j)=|S_T(n;i,n,j)|$. Then, in analogy with the previous subsection,
$$b_n(i,j)=F_{2n-2j-3},\quad 2\leq i+1<j\leq n-2,$$
with $b_n(i+1,i)=b_n(i,i+1)=b_n(i,n-1)=b_{n-1}(i)$ and $b_n(n-2,n-1)=1$.

\begin{lemma}\label{lem166A2x2}
Let $b_n(i)=|S_T(n;i,n)|$, $b_n=\sum_{i=1}^{n-1}b_n(i)$, and $B(x)=\sum_{n\geq2}b_nx^n$. Then
$$B(x)= \frac{x^2(1-5x+7x^2-x^3)}{(1-3x+x^2)(1-3x)(1-x)}.$$
\end{lemma}
\begin{proof}
From the preceding results, we have
$$b_n=3b_{n-1}-1+\sum_{j=1}^{n-4}F_{2j-1}(n-3-j),$$
with $b_2=1$. Thus,
$$b_n-b_{n-1}=3b_{n-1}-3b_{n-2}+F_1+F_3+\cdots+F_{2n-9},$$
which, by the fact that $F_1+F_3+\cdots+F_{2n-11}=F_{2n-8}$, implies
$$b_n=4b_{n-1}-3b_{n-2}+F_{2n-8},$$
with $b_2=1$ and $b_3=2$. This is the same recurrence as in Lemma \ref{lem166A1x2}.
\end{proof}

\begin{theorem}\label{th166A2} We have
$$F_T(x)=\frac{1-9x+30x^2-44x^3+27x^4-7x^5}{(1-3x)(1-x)(1-3x+x^2)^2}.$$
\end{theorem}
\begin{proof}
Let $G_m(x)$ be the generating function for $T$-avoiders with $m$ left-right maxima.
Clearly, $G_0(x)=1$ and $G_1(x)=xF_T(x)$. By using similar arguments as in the proof of Theorem \ref{th166A1}, we obtain that (see Lemma \ref{lem166A2x2}) $G_2(x)=K(x)B(x)$.

Now let us write an equation for $G_m(x)$ with $m\geq3$. Let $\pi=i_1\pi^{(1)}\cdots i_m\pi^{(m)}\in S_n(T)$ with exactly $m$ left-right maxima. Since $\pi$ avoids $1324$, we have that $i_1>\pi^{(1)}>\cdots>\pi^{(m-1)}$.
Let $s\in\{2,3,\ldots,m-1\}$ (maybe $s$ does not exist) be the minimal number such that $\pi^{(s)}\neq\emptyset$. We have the following cases:
\begin{itemize}
\item $s=2,3,\ldots,m-1$: Here $\pi^{(2)}=\pi^{(3)}=\cdots=\pi^{(s-1)}=\emptyset$ and there is no letter  in $\pi^{(m)}$ between $i_{s-1}$ and $i_s$. Thus, the contribution is given by $\frac{x^s}{1-x}G_{m+1-s}(x)$.

\item $s$ does not exist. Here $\pi^{(2)}=\pi^{(3)}=\cdots=\pi^{(m-1)}=\emptyset$. So there two cases, either $\pi^{(m)}$ contains a letter from the set $\{i_1+1,\ldots,i_{m-1}-1\}$ or not. If yes, then $\pi^{(m)}$ can be written as $\pi^{(m)}=\pi'(i_{m-1}-1)\cdots(i_{m-2}+2)(i_{m-2}+1)\cdots(i_1-1)\cdots21$ with $\pi'$ and $\pi^{(1)}$ avoiding $132$ and $3412$. Thus, the contribution is given by $\frac{x^m}{1-x}(1/(1-x)^{m-2}-1)K(x)^2$. Otherwise, $\pi^{(m)}$ does not contain any letter from the set $\{i_1+1,\ldots,i_{m-1}-1\}$, which yields a contribution of $x^{m-2}G_2(x)$.
\end{itemize}
Combining the previous cases, we obtain
\begin{align*}
G_m(x)&=x^{m-2}G_2(x)+\frac{x^m}{1-x}(1/(1-x)^{m-2}-1)K(x)^2+\sum_{s=2}^{m-1}\frac{x^s}{1-x}G_{m+1-s}(x).
\end{align*}
Since $F_T(x)=\sum_{m\geq0}G_m(x)$, we have
\begin{align*}
F_T(x)&=1+xF_T(x)+\frac{1}{1-x}K(x)B(x)+\sum_{m\geq3}\frac{x^m}{1-x}(1/(1-x)^{m-2}-1)K(x)^2\\
&\quad+\frac{x^2}{(1-x)^2}(F_T(x)-1-xF_T(x)).
\end{align*}
Solving for $F_T(x)$ and using Lemma \ref{lem166A2x2}, we complete the proof.
\end{proof}

\subsection{Case 171}
The two representative triples $T$ are:

\{1423,2314,2341\} (Theorem \ref{th171A1})

\{1324,1342,4123\} (Theorem \ref{th171A2})
\subsubsection{$\mathbf{T=\{1423,2314,2341\}}$}
\begin{theorem}\label{th171A1}
Let $T=\{1423,2314,2341\}$. Then
$$F_T(x)=\frac{1-4x+5x^2-x^3+(1-4x+3x^2-x^3)\sqrt{1-4x}}{(1-x)(1-3x+x^2)(1-2x+\sqrt{1-4x})}.$$
\end{theorem}
\begin{proof}
Let $G_m(x)$ be the generating function for $T$-avoiders with $m$ left-right maxima.
Clearly, $G_0(x)=1$ and $G_1(x)=xF_T(x)$. Next, we consider $m\geq3$.
If $\pi=i_1\pi^{(1)}\cdots i_m\pi^{(m)}$ avoids $T$ then $$\pi^{(1)}<\pi^{(2)}<\cdots<\pi^{(m-2)}<\pi^{(m-1)}i_m\pi^{(m)}$$
and $\pi^{(1)}$ avoids $231$ and $1423$ (due to the presence of $i_m$), $\pi^{(j)}=(i_{j+1}-1)\cdots(i_j+2)(i_j+1)$ for $j=2,3,\ldots,m-2$,
and $$i_{m-1}\pi^{(m-1)}i_m\pi^{(m)}=i_{m-1}(i_{m-1}-1)\cdots \ell\: i_m(i_m-1)\cdots (i_{m-1}+1)\,\ell\,(\ell-1)\cdots(i_{m-2}+1)\, .$$
These results are explained in the next figure, where entries are decreasing as indicated by arrows to avoid 1423,
and other shaded regions are empty to avoid the indicated pattern.
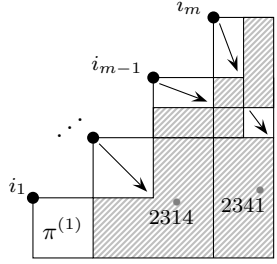
\begin{center}
\begin{pspicture}(0,-.1)(4,3.5)
\psset{unit =.8cm, linewidth=.5\pslinewidth}
\rput(.6,2.3){$\iddots$}
\pspolygon[fillstyle=hlines,hatchcolor=lightgray,hatchsep=0.8pt](1,0)(4,0)(4,2)(2,2)(2,1)(1,1)(1,0)
\psline[fillstyle=hlines,hatchcolor=lightgray,hatchsep=0.8pt](2,2)(3.5,2)(3.5,2.5)(4,2.5)(4,4)(3.5,4)(3.5,3)(3,3)(3,2.5)(2,2.5)(2,2)
\psline(1,0)(0,0)(0,1)(1,1)(1,2)(2,2)(2,3)(3,3)(3,4)(4,4)
\psline(4,2)(4,3)
\psline(3,0)(3,2.5)(3.5,2.5)(3.5,3)
\pscircle*(0,1){0.1}\pscircle*(1,2){0.1}\pscircle*(2,3){0.1}\pscircle*(3,4){0.1}
\psline[arrows=->,arrowsize=3pt 3](1.1,1.9)(1.9,1.1)
\psline[arrows=->,arrowsize=3pt 3](2.1,2.9)(2.9,2.6)
\psline[arrows=->,arrowsize=3pt 3](3.6,2.45)(3.9,2.05)
\psline[arrows=->,arrowsize=3pt 3](3.1,3.9)(3.4,3.1)
\rput(-0.25,1.2){\textrm{{\footnotesize $i_1$}}}
\rput(1.4,3.2){\textrm{{\footnotesize $i_{m-1}$}}}
\rput(2.6,4.2){\textrm{{\footnotesize $i_m$}}}
\rput(0.5,.5){\textrm{{\small $\pi^{(1)}$}}}
\rput(2.3,.8){\textrm{{\footnotesize $23\overset{{\gray \bullet}}{1}4$}}}
\rput(3.5,1){\textrm{{\footnotesize $234\overset{{\gray \bullet}}{1}$}}}

\end{pspicture}
\end{center}
There are $m$ left-right maxima and $m$ regions containing arrows (to be filled with an arbitrary number of ``balls''), and the \gf $K:=\sum_{n\geq0}|S_n(231,1423)|x^n$ for $\{231,1423\}$-avoiders is $\frac{1-2x}{1-3x+x^2}$ \cite[Seq. A001519]{Sl}.
Hence, for $m\ge 3$,
$$G_m(x)=\frac{x^m}{(1-x)^m}K.$$

Now we find an explicit formula for $G_2(x)$. In order to do that we define the following notation. Let $g_k(x)$ denote the generating function for $T$-avoiders in $S_n$ with two left-right maxima and leftmost letter $n-1-k,\ 0\le k \le n-2$. Let $g'_k(x)$ to be the generating function for $T$-avoiders with two left-right maxima $(n-1-k)n$ in the two leftmost positions.

First, we find an equation for $g_k(x)$. Let $\pi=(n-1-k)\pi'n\pi''\in S_n(T)$ with two left-right maxima and leftmost letter $n-1-k$. If $\pi'$ is empty, then the contribution is $g'_k(x)$. Otherwise, since $\pi'$ avoids $2314$, $\pi'$ has the form
$\beta' d\, \beta''$ where $\beta'<\beta''<d$. 
If $\beta'$ is empty, then the contribution is  $x\sum_{j\geq0}g_{k+j}(x)$, upon considering deletion of $n-1-k$.
If $\beta'$ is not empty, then $\pi$ has the form
$$(n-1-k)\beta' d(d-1)\cdots d'n(n-1)\cdots(n-k)(n-2-k)\cdots(d+1)(d'-1)\cdots(d''+1),$$
for some $d'$, where $d''$ is the largest letter of $\beta'$. The contributions are $x^{k+2}$ for $\{n-k-1,n-k,\dots,n\}$, $x$ for $d$, $1/(1-x)^3$ for the 3 decreasing sequences, and $K-1$ for $\beta'$, hence
$\frac{x^{k+3}(K-1)}{(1-x)^3}$ altogether. By combining all the contributions, we obtain
$$g_k(x)=g'_k(x)+x\sum_{j\geq0}g_{k+j}(x)+\frac{x^{k+3}(K-1)}{(1-x)^3}$$
for $k\ge 1$,
with initial condition $g_0(x)=x(F_T(x)-1)$ (delete $n-1$, which can play no role in a forbidden pattern).

Define the \gf $G(x,u)=\sum_{k\geq0}g_k(x)u^k$. Note that $G(x,1)=G_2(x)$ since $G_2(x) =\sum_{k\ge 0}g_k(x)$.
The preceding recurrence for $g_k(x)$ can now be written as
\begin{align}\label{eq171A1x1}
G(x,u)=G'(x,u)+x(F_T(x)-1)-x^2F_T(x)+\frac{xu}{1-u}\big(G(x,1)-G(x,u)\big)+\frac{x^4u(K-1)}{(1-x)^3(1-xu)},
\end{align}
where $G'(x,u)=\sum_{k\geq0}g'_k(x)u^k$.

Next, let us write an equation for $g'_k(x)$. So suppose $\pi=(n-1-k)n\pi'\in S_n(T)$ has two left-right maxima. Clearly, $g'_0(x)=x^2F_T(x)$ (delete the first two letters, $n-1$ and $n$). For $k\geq1$, $\pi$ can be written as $\pi=(n-1-k)n\beta'(n-1)\beta''$. If $\beta'$ is empty, then the contribution is given by $xg'_{k-1}(x)$.
Otherwise, similar to the $g_k$ case, $\beta'$
has the form $\gamma' d \gamma''$ where $\gamma'<\gamma''<d$, and by considering whether $\gamma'$ is empty or not,
we obtain the contribution $x^2\sum_{j\geq0}g_{k-1+j}(x)+\frac{x^{k+3}(K-1)}{(1-x)^3}$.
Combining all the previous cases yields
$$g'_k(x)=xg'_{k-1}(x)+x^2\sum_{j\geq0}g_{k-1+j}(x)+\frac{x^{k+3}(K-1)}{(1-x)^3}$$
for $k\ge 1$, with $g'_0(x)=x^2F_T(x)$. Multiply by $u^k$ and sum over $k\ge 0$ to obtain
\begin{align}\label{eq171A1x2}
G'(x,u)=xuG'(x,u)+x^2F_T(x)+\frac{x^2u}{1-u}\big(G(x,1)-uG(x,u)\big)+\frac{x^4u}{(1-x)^3(1-xu)}(K-1).
\end{align}
Solving \eqref{eq171A1x2} for $G'(x,u)$, and substituting into \eqref{eq171A1x1}, yields
\begin{align*}
\frac{1-u+xu^2}{1-u}G(x,u)&=\frac{x((1-x)^3+xu(2-xu)(x^2-(1-x)^3))}{(1-x)^3(1-xu)}\\
&\quad-x(1-xu+x^2u)F_T(x)-\frac{x^4u(2-xu)K}{(1-x)^3(1-xu)}-\frac{xu(1+x-xu)}{1-u}G(x,1)\, .
\end{align*}
To solve the preceding functional equation, we apply the kernel method and take $u=C(x)$,
which cancels out the $G(x,u)$ term.
A calculation, using the identity $xC(x)^2=C(x)-1$ to simplify the result (best done by computer), now yields
\[
G(x,1) = \frac{x (-1+(2-x) C(x)) }{1+x C(x)}F_T(x)-\frac{x \left(1-5 x+9 x^2-4 x^3+x^4-x^2
   C(x)\right)}{(1-x)^2 \left(1-3 x+x^2\right) \big(1+x (1-C(x))\big)}\, .
\]

Hence, since $F_T(x)=\sum_{m\geq0}G_m(x)$ and $G_2(x)=G(x,1)$, we obtain
$$F_T(x)=1+xF_T(x)+G(x,1)+\frac{x^3K}{(1-x)^2(1-2x)},$$
which leads to
$$F_T(x)=\frac{1-4x+5x^2-x^3+(1-4x+3x^2-x^3)\sqrt{1-4x}}{(1-x)(1-3x+x^2)(1-2x+\sqrt{1-4x})},$$
as required.
\end{proof}

\subsubsection{$\mathbf{T=\{1324,1342,4123\}}$}
For this case, we define $a(n;i_1,i_2,\ldots,i_k)$ for $n\geq k$ to be the number of $T$-avoiding permutations of length $n$ whose first $k$ letters are $i_1,i_2,\ldots,i_k$.  Let $a(n)=\sum_{i=1}^n a(n;i)$ for $n\geq 1$ and $\mathcal{T}_{i,j}$ be the subset of permutations enumerated by $a(n;i,j)$.  It is convenient to consider separately the cases when either the second or third letter equals $n$.  To this end, let $e(n;i)=a(n;i,n)$ for $1 \leq i \leq n-2$ (with $e(n;n-1)$ defined to be zero) and $f(n;i,j)=a(n;i,j,n)$ for $4 \leq i \leq n-1$ and $1 \leq j \leq i-3$.  The arrays $a(n;i,j)$, $e(n;i)$ and $f(n;i,j)$ are determined recursively as follows.

\begin{lemma}\label{171Bl1}
We have
\begin{equation}\label{171Bl11e1}
a(n;i,i-2)=a(n-2;i-2)+e(n-1;i-2)+\sum_{j=1}^{i-3}a(n-1;i-1,j), \qquad 3 \leq i \leq n,
\end{equation}
\begin{equation}\label{171Bl11e2}
a(n;i,j)=f(n;i,j)+\sum_{\ell=1}^{j-1}a(n-1;i-1,\ell), \qquad 4 \leq i \leq n-1 \quad and \quad 1 \leq j \leq i-3,
\end{equation}
\begin{equation}\label{171Bl11e3}
a(n;n,j)=\sum_{\ell=1}^{j}a(n-1;n-1,\ell), \qquad 1 \leq j \leq n-3,
\end{equation}
\begin{equation}\label{171Bl11e4}
e(n;i)=e(n-1;i)+\sum_{j=1}^ia(n-1;n-1,j), \qquad 1 \leq i \leq n-3,
\end{equation}
and
\begin{equation}\label{171B1l1e5}
f(n;i,j)=f(n-1;i,j)+\sum_{\ell=1}^{j-1}a(n-2;n-2,\ell), \qquad 4 \leq i \leq n-2 \quad and \quad 1 \leq j \leq i-3,
\end{equation}
with $e(n;n-2)=C_{n-2}$ for $n\geq 3$ and $f(n;n-1,j)=a(n-1;n-1,j)$ for $1 \leq j \leq n-4$.  Furthermore, we have $a(n;i,j)=0$ if $n \geq4$ and $1 \leq i<j-1<n-1$, $a(n;i,i+1)=a(n-1;i)$ if $1\leq i \leq n-1$, and $a(n;i,i-1)=a(n-1;i-1)$ if $2 \leq i\leq n$.
\end{lemma}
\begin{proof}
The formulas for $a(n;i,i+1)$ and $a(n;i,i-1)$, and for $a(n;i,j)$ when $i<j-1< n-1$, follow from the definitions. In the remaining cases, let $x$ denote the third letter of a $T$-avoiding permutation. For \eqref{171Bl11e1}, note that members of $\mathcal{T}_{i,i-2}$ when $i<n$ must have $x=i-1$, $x=n$ or $x<i-2$, lest there be an occurrence of $1324$ or $1342$.  This is seen to give $a(n-2;i-2)$, $e(n-1;i-2)$ and $\sum_{j=1}^{i-3}a(n-1;i-1,j)$ possibilities, respectively, which implies \eqref{171Bl11e1}.  Observe that \eqref{171Bl11e1} also holds when $i=n$ since $e(n-1;n-2)=0$, by definition.  For \eqref{171Bl11e2}, note that members of $\mathcal{T}_{i,j}$ where $i<n$ and $j \leq i-3$ must have $x=n$ or $x<j$ (as $x=j+1$ is not permitted due to $4123$ and $j+2\leq x \leq n-1$ is not due to $1324$, $1342$).  In the second case, the letter $j$ becomes extraneous and thus may be deleted since $i,x$ imposes a stricter requirement on later letters than does $j,x$ (with $x<j$ making $j$ redundant with respect to $1324$, $1342$). Relation \eqref{171Bl11e2} then follows from the definitions.  For \eqref{171Bl11e3}, note that members of $\mathcal{T}_{n,j}$ where $j\leq n-3$ must have $x=n-1$ or $x<j$ in order to avoid $4123$, which accounts for the $\ell=j$ term and the remaining terms, respectively, in the sum on the right-hand side.

For \eqref{171Bl11e4}, note that members of $\mathcal{T}_{i,n}$ where $i \leq n-3$ must have $x=n-1$, $x<i$ or $x=i+1$.  The letter $n$ may be deleted in the first case, while the $i$ may be deleted in the latter two (as $n,x$ imposes a stricter requirement on subsequent letters than $i,x$).  Thus, there are $e(n-1;i)$, $\sum_{j=1}^{i-1}a(n-1;n-1,j)$ and $a(n-1;n-1,i)$ possibilities, respectively, which implies \eqref{171Bl11e4}.  That $e(n;n-2)=C_{n-2}$ follows from the fact that members of $\mathcal{T}_{n-2,n}$ are synonymous with $123$-avoiding permutations of length $n-2$ which are well known to be enumerated by $C_{n-2}$ (note that $n-2$ is redundant due to $n$).  Finally, to show  \eqref{171B1l1e5}, note that permutations counted by $f(n;i,j)$ must have fourth letter $y$ equal $n-1$ or less than $j$.  If $y=n-1$, then the letter $n$ may be deleted and thus there are $f(n-1;i,j)$ possibilities, by definition.  If $y<j$, then $n,y$ imposes a stricter requirement on the remaining letters with respect to $4123$ than does $i,y$ or $i,j$, with the $i$ and $j$ also redundant with respect to $1324$ or $1342$ due to $y$.  Thus, both $i$ and $j$ may be deleted in this case, yielding $\sum_{\ell=1}^{j-1}a(n-2;n-2,\ell)$ possibilities, which implies \eqref{171B1l1e5}.  That $f(n;n-1,j)=a(n-1;n-1,j)$ holds for $1 \leq j \leq n-4$ since the letter $n$ may be deleted in this case, which completes the proof.
\end{proof}

In order to solve the recurrences of the prior lemma, we introduce the following functions:  $a_n(u)=\sum_{i=1}^na(n;i)u^i$ for $n\geq1$, $b_{n,i}(v)=\sum_{j=1}^{i-2}a(n;i,j)v^j$ for $3 \leq i \leq n$, $b_n(u,v)=\sum_{i=1}^{n-1}b_{n,i}(v)u^i$ for $n\geq3$,
$$c_{n,i}(v)=a(n-1;i)+a(n-1;i-1)+b_{n,i}(v), \qquad 1 \leq i \leq n-1,$$
$c_n(u,v)=\sum_{i=1}^{n-1}c_{n,i}(v)u^i$ for $n\geq2$, $e_n(u)=\sum_{i=1}^{n-2}e(n;i)u^i$ for $n\geq 3$, $f_{n,i}(v)=\sum_{j=1}^{i-3}f(n;i,j)v^j$ for $4 \leq i \leq n-1$, and $f_n(u,v)=\sum_{i=4}^{n-1}f_{n,i}(v)u^i$ for $n\geq 5$.

By the definitions, we have
\begin{equation}\label{171Be1}
a_n(u)=c_n(u,1)+e_n(u)+C_{n-1}u^n, \qquad n \geq 1.
\end{equation}
Assume $b_{n,1}(v)=b_{n,2}(v)=0$.  By \eqref{171Bl11e1} and \eqref{171Bl11e2}, we have for $3 \leq i \leq n-1$,
\begin{align*}
b_{n,i}(v)&=(a(n-2;i-2)+b_{n-1,i-1}(1)+e(n-1;i-2))v^{i-2}+f_{n,i}(v)\\
&\quad+\sum_{j=1}^{i-3}v^j\sum_{\ell=1}^{j-1}a(n-1;i-1,\ell)\\
&=(a(n-2;i-2)+e(n-1;i-2))v^{i-2}+f_{n,i}(v)+\frac{v}{1-v}(b_{n-1,i-1}(v)-v^{i-2}b_{n-1,i-1}(1)).
\end{align*}
Multiplying both sides of the last equation by $u^i$, and summing over $3 \leq i \leq n-1$, yields
\begin{align}
b_n(u,v)&=u^2(a_{n-2}(uv)+e_{n-1}(uv)-C_{n-3}(uv)^{n-2})+f_n(u,v)\notag\\
&\quad+\frac{u}{1-v}(vb_{n-1}(u,v)-b_{n-1}(uv,1)), \qquad n \geq 3. \label{171Be2}
\end{align}
By \eqref{171Bl11e1} and \eqref{171Bl11e3}, we get
\begin{align}
b_{n,n}(v)&=(a(n-2;n-2)+b_{n-1,n-1}(1))v^{n-2}+\frac{1}{1-v}(b_{n-1,n-1}(v)-v^{n-2}b_{n-1,n-1}(1))\notag\\
&=C_{n-3}v^{n-2}+\frac{1}{1-v}(b_{n-1,n-1}(v)-v^{n-1}b_{n-1,n-1}(1)), \qquad n \geq 3. \label{171Be3}
\end{align}
By the definitions, we have
\begin{align}
c_n(u,v)&=a_{n-1}(u)+u(a_{n-1}(u)-a(n-1;n-1)u^{n-1})+b_n(u,v)\notag\\
&=(u+1)a_{n-1}(u)-C_{n-2}u^n+b_n(u,v), \qquad n \geq 2. \label{171Be4}
\end{align}

Multiplying both sides of \eqref{171Bl11e4} by $u^i$, and summing over $1 \leq i \leq n-3$, gives
\begin{align}
e_n(u)&=e(n;n-2)u^{n-2}+e_{n-1}(u)+\sum_{j=1}^{n-3}a(n-1;n-1,j)\left(\frac{u^j-u^{n-2}}{1-u}\right)\notag\\
&=C_{n-2}u^{n-2}+e_{n-1}(u)+\frac{1}{1-u}(b_{n-1,n-1}(u)-u^{n-2}b_{n-1,n-1}(1)), \qquad n \geq 3. \label{171Be5}
\end{align}
By \eqref{171B1l1e5}, we have
$$f_{n,i}(v)=f_{n-1,i}(v)+\frac{1}{1-v}\sum_{\ell=1}^{i-3}a(n-2;n-2,\ell)(v^{\ell+1}-v^{i-2}), \qquad 4 \leq i \leq n-2,$$
with
$$f_{n,n-1}(v)=\sum_{j=1}^{n-4}f(n;n-1,j)v^j=\sum_{j=1}^{n-4}a(n-1;n-1,j)v^j=b_{n-1,n-1}(v)-C_{n-3}v^{n-3}, \qquad n \geq 5.$$
We then get
\begin{align}
f_n(u,v)-f_{n,n-1}(v)u^{n-1}&=f_{n-1}(u,v)+\frac{1}{1-v}\sum_{\ell=1}^{n-4}a(n-2;n-2,\ell)v^{\ell+1}\sum_{i=\ell+3}^{n-2}u^i\notag\\
&\quad-\frac{1}{1-v}\sum_{\ell=1}^{n-4}a(n-2;n-2,\ell)\sum_{i=\ell+3}^{n-2}u^iv^{i-2}\notag\\
&=f_{n-1}(u,v)+\frac{uv}{(1-u)(1-v)}(u^2b_{n-2,n-2}(uv)-u^{n-2}b_{n-2,n-2}(v))\notag\\
&\quad-\frac{u}{v(1-uv)(1-v)}(u^2v^2b_{n-2,n-2}(uv)-(uv)^{n-2}b_{n-2,n-2}(1)),\quad n\geq5. \label{171Be6}
\end{align}

Let $C(x)=\sum_{n\geq0}C_nx^n$.  Define the generating functions $a(x;u)=\sum_{n\geq1}a_n(u)x^n$, $b(x;u,v)=\sum_{n\geq3}b_n(u,v)x^n$, $c(x;u,v)=\sum_{n\geq2}c_n(u,v)x^n$, $d(x;v)=\sum_{n\geq3}b_{n,n}(v)x^n$, $e(x;u)=\sum_{n\geq3}e_n(u)x^n$, and $f(x;u,v)=\sum_{n\geq5}f_n(u,v)x^n$.
By \eqref{171Be4} at $v=1$, we have
\begin{equation}\label{171Be7}
c(x;u,1)=x(u+1)a(x;u)+b(x;u,1)-x^2u^2C(xu).
\end{equation}
Rewriting recurrences \eqref{171Be1}--\eqref{171Be3}, \eqref{171Be5}, and \eqref{171Be6} in terms of generating functions, and applying \eqref{171Be7} to the relation obtained from \eqref{171Be1}, yields the following system of functional equations.

\begin{lemma}\label{171Bl2}
We have
\begin{equation}\label{171feqa}
(1-x(u+1))a(x;u)=xu(1-xu)C(xu)+b(x;u,1)+e(x;u),
\end{equation}
\begin{align}
\left(1-\frac{xuv}{1-v}\right)b(x;u,v)&=f(x;u,v)+xu^2e(x;uv)-x^3u^3vC(xuv)-\frac{xu}{1-v}b(x;uv,1)\notag\\
&\quad+x^2u^2a(x;uv),\label{171feqb}
\end{align}
\begin{equation}\label{171feqd}
\left(1-\frac{x}{1-v}\right)d(x;v)=x^3vC(xv)-\frac{x}{1-v}d(xv;1),
\end{equation}
\begin{equation}\label{171feqe}
(1-x)e(x;u)=x^2(C(xu)-1)+\frac{x}{u(1-u)}\left(ud(x;u)-d(xu;1)\right),
\end{equation}
and
\begin{align}
(1-x)f(x;u,v)&=xd(xu;v)-x^3u^2(C(xuv)-1)+\frac{x^2uv}{(1-u)(1-v)}(u^2d(x;uv)-d(xu;v))\notag\\
&\quad-\frac{x^2u}{v(1-uv)(1-v)}(u^2v^2d(x;uv)-d(xuv;1)).\label{171feqf}
\end{align}
\end{lemma}

Note that the last three equations in the prior lemma are independent of the first two.

\begin{lemma}\label{171Bl3}
We have
\begin{equation}\label{171de2}
d(x;v)=\frac{x^2v(1-(1-x)C(xv))}{1-x-v},
\end{equation}
\begin{equation}\label{171ee2}
e(x;u)=\frac{x^2u(1-(1-x)C(xu))}{(1-x)(1-x-u)},
\end{equation}
and
\begin{equation}\label{171fe2}
f(x;1/(1-x),1-x)=\frac{1-6x+9x^2-2x^3-(1-3x)(1-x)\sqrt{1-4x}}{2x(1-x)^2}.
\end{equation}
\end{lemma}
\begin{proof}
Replacing $x$ with $x/v$ in \eqref{171feqd} gives
\begin{equation}\label{171Bl3e1}
\left(1-\frac{x}{v(1-v)}\right)d(x/v;v)=\frac{x^3}{v^2}C(x)-\frac{x}{v(1-v)}d(x;1).
\end{equation}
Applying the kernel method to \eqref{171Bl3e1}, and taking $v=\frac{1+\sqrt{1-4x}}{2}=\frac{1}{C(x)}$, yields
$$d(x;1)=x^2\left(1-\frac{1}{C(x)}\right)C^2(x)=x(C(x)-1)-x^2C(x)=(x-x^2)C(x)-x$$
and thus
\begin{align*}
d(x;v)&=\frac{x^3v(1-v)C(xv)}{1-x-v}-\frac{x}{1-x-v}(xv(1-xv)C(xv)-xv)\\
&=\frac{x^2v(1-(1-x)C(xv))}{1-x-v}.
\end{align*}
Formula \eqref{171ee2} now follows from \eqref{171de2} and \eqref{171feqe}.
By taking $u\rightarrow1/v$  in \eqref{171feqf}, we obtain
\begin{align*}
&(1-x)f(x;1/v,v)\\
&=xd(x/v;v)-x^3/v^2(C(x)-1)+\frac{x^2}{(1-1/v)(1-v)}(d(x;1)/v^2-d(x/v;v))\\
&\quad-\frac{x^2}{v^2(1-v)}\lim_{u\rightarrow1/v}\frac{u^2v^2d(x;uv)-d(xuv;1)}{1-uv}\\
&=xd(x/v;v)-x^3/v^2(C(x)-1)+\frac{x^2}{(1-1/v)(1-v)}(d(x;1)/v^2-d(x/v;v))\\
&\quad+\frac{x^2(2d(x;1)+\frac{d}{dw}d(x;w)\mid_{w=1}-x\frac{d}{dx}d(x;1))}{v^2(1-v)}.
\end{align*}
Substituting $v=1-x$ in the last expression, and using \eqref{171de2}, yields \eqref{171fe2}.
\end{proof}

We can now determine the generating function $F_T(x)$.

\begin{theorem}\label{th171A2}
Let $T=\{1324,1342,4123\}$.  Then
$$F_T(x)=\frac{1-3x+x^2-x^3-(1-x)^3\sqrt{1-4x}}{2x(1-4x+4x^2-x^3)}.$$
\end{theorem}
\begin{proof}
In the notation above, we seek to determine $1+a(x;1)$.
By \eqref{171feqb} with $u=1/(1-x)$ and $v=1-x$ and by \eqref{171feqa} with
$u=1$, we have
\begin{align*}
f(x;1/(1-x),1-x)&=-\frac{x}{(1-x)^2}e(x;1)+\frac{x^3}{(1-x)^2}C(x)
+\frac{1}{1-x}b(x;1,1)-\frac{x^2}{(1-x)^2}a(x;1),\\
a(x;1)&=\frac{x(1-x)}{1-2x}C(x)+\frac{1}{1-2x}b(x;1,1)+\frac{1}{1-2x}e(x;1).
\end{align*}
Substituting the expressions for $e(x;1)$ and
$f(x;1/(1-x),x)$ from the prior lemma, and then solving the system that results for $a(x;1)$ and
$b(x;1,1)$, we obtain
\begin{align*}
a(x;1)&=\frac{1-5x+9x^2-9x^3+2x^4-(1-x)^3\sqrt{1-4x}}{2x(1-4x+4x^2-x^3)},\\
b(x;1,1)&=\frac{1-8x+22x^2-23x^3+5x^4-(1-6x+12x^2-7x^3+x^4)\sqrt{1-4x}}{2x(1-3x+x^2)}.\\
\end{align*}
Hence,
\begin{align*}
1+a(x;1)&=\frac{1-3x+x^2-x^3-(1-x)^3\sqrt{1-4x}}{2x(1-4x+4x^2-x^3)},
\end{align*}
as desired.
\end{proof}

\subsection{Case 174}
The three representative triples $T$ are:

\{2134,2341,2413\} (Theorem \ref{th174A3})

\{2143,2314,2341\} (Theorem \ref{th174A5})

\{2143,2314,2431\} (Theorem \ref{th174A6})
\begin{theorem}\label{th174A3}
Let $T=\{2134,2341,2413\}$. Then
 $$F_T(x)=\frac{1-6x+10x^2-3x^3+x^4}{(1-3x+x^2)(1-4x+2x^2)}.$$
\end{theorem}
\begin{proof}
Let $G_m(x)$ be the generating function for $T$-avoiders with $m$ left-right maxima.
Clearly, $G_0(x)=1$ and $G_1(x)=xF_T(x)$. Now let us write an equation for $G_m(x)$ with $m\geq2$.

For $m=2$, let $\pi=i\pi'n\pi''\in S_n(T)$ with two left-right maxima, $i$ and $n$, and consider cases on $i$. If $i=n-1$, then $\pi \rightarrow \pi' (n-1)\pi''$ is a bijection
to nonempty $T$-avoiders of length
$n-1$, giving a contribution of  $x(F_T(x)-1)$. If $2\le i \le n-2$ and $\pi''$ has the form $(n-1)
(n-2)\cdots (i+1)\pi'''$, then $\pi \rightarrow \pi' i \pi'''$ is a bijection to $T$-avoiders of length $i$, giving a contribution of  $x^2/(1-x)(F_T(x)-1-x)$.
\begin{lemma}\label{lem174A3}
If $\pi=i\pi'n\pi''\in S_n(T)$ has two left-right maxima, $i\le n-2$ and
$\pi''$ does not have the form
$(n-1)(n-2)\cdots (i+1)\pi'''$, then $\pi'=\emptyset$ and $i=1$.
\end{lemma}
\begin{proof}
Note first that all letters of $[i+1,n-1]$ must occur prior to any letters of $[i-1]$ within $\pi''$ in order to avoid 2413.  By hypothesis, there exist $a,b \in \pi''$ such that $i<a<b$ with $a$ occurring before $b$. If $x \in \pi'$, then $ixab$ is a
2134. Hence, $\pi'=\emptyset$. If $i>1$, then 1 occurs (i) before $a$ or
(ii) after $b$. If (i), $i1ab$ is a 2134; if (ii), $iab1$ is a 2341, both
forbidden. Hence, $i=1$.
\end{proof}
By the lemma, the only remaining case is $\pi=1n\pi''$ with $n\ge 3$ (since $n=2$ falls under
the case  $i=n-1$). Here, $\pi \rightarrow \pi''$ is a bijection to $T$-avoiders of length
$n-2$, giving a contribution of $x^2 (F_T(x)-1)$. Summing all contributions, we find
$$G_2(x)=\frac{x}{1-x}\left(F_T(x)-1\right)+x^2\left(F_T(x)-\frac{1}{1-x}\right).$$

For $m\geq3$, let $\pi=i_1\pi^{(1)}\cdots i_m\pi^{(m)}\in S_n(T)$ with $m$ left-right maxima.
Then $\pi^{(1)}=\pi^{(2)}=\dots = \pi^{(m-2)}=\emptyset$ for else a 2134 is present;
also $\pi^{(m-1)}>i_{m-3}$ (with $i_0=0$) and $\pi^{(m)}>i_{m-2}$ or a 2341 is present.
Consequently, if there is no letter between $i_{m-3}$ and $i_{m-2}$, then $\pi$ has the form
$12\cdots(m-2)\pi'$ where $\pi'$ is a permutation of $m-1,m-2,\dots,n$ with two left-right maxima that
avoids $T$, giving a contribution of $x^{m-2}G_2(x)$.
On the other hand, if there is a letter between $i_{m-3}$ and $i_{m-2}$, the reader may
verify that $\pi$ must have the form
$$\pi=12\cdots(m-3)i_{m-2}i_{m-1}(i_{m-1}-1)\cdots(i_{m-2}+1)\pi'i_m(i_m-1)\cdots(i_{m-1}+1)\, ,$$
where $\pi'$ is a nonempty permutation of $i_{m-3}+1,\ldots,i_{m-2}-1$ that avoids $213$ and $2341$.
Here, the contribution is $\frac{x^m}{(1-x)^2}(K-1)$,
where $K=\sum_{n\geq0}|S_n(213,2341)|x^n=\frac{1-2x}{1-3x+x^2}$ (see \cite[Seq. A001519]{Sl}).
Hence, for $m\ge 3$,
$$G_m(x)=x^{m-2}G_2(x)+\frac{x^m}{(1-x)^2}(K-1)\, .$$
Since $F_T(x)=\sum_{m\geq0}G_m(x)$, we have
$$F_T(x)=1+xF_T(x)-\frac{x}{(1-x)^2}\left((x^2-x-1)F_T(x)+x+1\right)+\frac{(K-1)x^3}{(1-x)^3},$$
with solution the stated $F_T(x)$.
\end{proof}

\begin{theorem}\label{th174A5}
Let $T=\{2143,2314,2341\}$. Then
 $$F_T(x)=\frac{1-6x+10x^2-3x^3+x^4}{(1-3x+x^2)(1-4x+2x^2)}.$$
\end{theorem}
\begin{proof}
Let $G_m(x)$ be the generating function for $T$-avoiders with $m$ left-right maxima.
Clearly, $G_0(x)=1$ and $G_1(x)=xF_T(x)$. Now let us write an equation for $G_m(x)$ with $m\geq3$.
Suppose $\pi=i_1\pi^{(1)}\cdots i_m\pi^{(m)}\in S_n(T)$ with $m\ge 3$ left-right maxima $i_1,i_2,\dots,i_m$. Then, because $\pi$ avoids 2314 and 2341,
$$\pi^{(1)}<i_1<\pi^{(2)}<i_2<\cdots<\pi^{(m-2)}<i_{m-2}<i_{m-1}\pi^{(m-1)}i_m\pi^{(m)}.$$
If $\pi^{(1)}=\cdots=\pi^{(j-1)}=\emptyset$ and $\pi^{(j)}\neq\emptyset$ with $j=1,2,\ldots,m-2$, then we have that $\pi^{(j+1)}=\cdots=\pi^{(m)}=\emptyset$ ($\pi$ avoids $2143$). So the contribution for each $j=1,2,\ldots,m-2$ is $x^m(K-1)$, where $K=\sum_{n\geq0}|S_n(231,2143)|x^n=\frac{1-2x}{1-3x+x^2}$ (see \cite[Seq. A001519]{Sl}). If $\pi^{(1)}=\cdots=\pi^{(m-2)}=\emptyset$, then the contribution is given by $x^{m-2}G_2(x)$. Thus,
$$G_m(x)= (m-2)x^m(K-1) + x^{m-2}G_2(x), \qquad m \geq3.$$

It remains to find a formula for $G_2(x)$. So suppose $\pi=i\pi'n\pi''$ and consider whether $
\pi'$ is empty or not.
If $\pi'=\emptyset$, then $\pi \rightarrow i \pi''$ is a bijection to nonempty $T$ avoiders, giving
a contribution of $x \big(F_T(x)-1\big)$. If $\pi'\ne \emptyset$, say $x\in \pi'$, then $i=n-1$
because $i<n-1$ implies $i xn(n-1)$ is a 2143. So $\pi$ has first letter $n-1$ (contributes $x
(F_T(x)-1)$\,) and second letter $\ne n$ (hence, subtract $x^2 F_T(x)$\,)
for a net contribution of $x\big(F_T(x)-1-xF_T(x)\big)$. Thus,
\[
G_2(x)=x\big(F_T(x)-1\big)+x\big(F_T(x)-1-xF_T(x)\big)\, .
\]
Since $F_T(x)=\sum_{m\geq0}G_m(x)$, we have
$$F_T(x)=1+xF_T(x)+\frac{x}{1-x}\big(2F_T(x)-2-xF_T(x) \big)+\frac{x^3(K-1)}{(1-x)^2},$$
with solution the stated $F_T(x)$.
\end{proof}

\begin{theorem}\label{th174A6}
Let $T=\{2143,2314,2431\}$. Then
 $$F_T(x)=\frac{1-6x+10x^2-3x^3+x^4}{(1-3x+x^2)(1-4x+2x^2)}.$$
\end{theorem}
\begin{proof}
Let $G_m(x)$ be the generating function for $T$-avoiders with $m$ left-right maxima.
Clearly, $G_0(x)=1$ and $G_1(x)=xF_T(x)$. Now let us write an equation for $G_m(x)$ with $m\geq2$.
Suppose $\pi=i_1\pi^{(1)}\cdots i_m\pi^{(m)}\in S_n(T)$ with $m\ge 2$ left-right maxima $i_1,i_2,\dots,i_m$.
Then, because $\pi$ avoids 2314,
$$\pi^{(1)}<i_1<\pi^{(2)}<i_2<\cdots<\pi^{(m-2)}<i_{m-2}<\pi^{(m-1)}<i_{m-1}.$$
If $\pi^{(1)}, \dots,\pi^{(m)}$ are all empty, the contribution is $x^m$.
Now suppose the $\pi$'s are not all empty and $j$ is minimal such that $\pi^{(j)}\ne \emptyset$.

First, suppose $j \in [m-1]$. Then $\pi^{(1)}=\cdots=\pi^{(j-1)}=\emptyset$ by supposition
and $\pi^{(j+1)}=\dots=\pi^{(m-1)}=\emptyset$ (to avoid 2143).
Furthermore, $\pi_m>i_{j-1}$ (to avoid 2431) and $\pi_m <i_{j}$ (to avoid 2143).
Hence, $\pi=12\cdots (j-1)i_j \pi^{(j)} (i_j +1)(i_j+2)\cdots n\, \pi^{(m)}$.
So ``delete $12\cdots (j-1)$ and $(i_j +1)(i_j+2)\cdots (n-1)$ and standardize'' is a bijection to $T$-avoiders with second largest letter in first position and largest letter not in second position, giving
a contribution of $x^{m-1}(F_T(x)-1-xF_T(x))$ as in Theorem \ref{th174A5}, for each $j \in [m-1]$.

Next, suppose $j=m$ so that $\pi=i_1 i_2 \cdots i_m \pi^{(m)}$ with $\pi^{(m)} \ne \emptyset$. Then,
because $\pi$ avoids 2431, $\pi$ has the form
$i_1 i_2 \cdots i_m\,\beta^{(1)}\beta^{(2)}\cdots\beta^{(m)}$
with $\beta^{(1)}<i_1<\beta^{(2)}<i_2<\cdots<\beta^{(m)}<i_m$. Let $\ell$ be the minimal index such
that $\beta^{(\ell)}$ is not empty, and say $x \in \beta^{(\ell)}$.
Then $\beta^{(j)}= \emptyset$ for $j\ge
\ell +2$ because if $y \in \beta^{(j)}$ with $j\ge \ell +2$, then $i_{j-2}i_{j-1}xy$ is a 2314.
Furthermore,  $\beta^{(\ell+1)}$ is increasing, because $z>y$ in
$\beta^{(\ell+1)}$ implies $i_{\ell}xzy$
is a 2143.

If $\beta^{(\ell+1)}=\emptyset$, we get a contribution of $x^m(F_T(x)-1)$ for each $\ell \in
[m]$.

If $\beta^{(\ell+1)}\ne \emptyset$, then $\beta^{(\ell)}$ must avoid 231. So ``delete the initial
$m$ letters (= the left-right maxima) and standardize'' is a bijection to pairs ($\gamma^{(\ell)},
\gamma^{(\ell+1)}$) with $\gamma^{(\ell)}$ a nonempty $\{231,2143\}$-avoider and $\gamma^{(\ell+1)}$ an
initial segment of the positive integers.
Thus we get, for each $\ell \in [m-1]$, a contribution of $x^m \frac{x}{1-x}(K-1)$, where
$K=\sum_{n\geq0}|S_n(231,2143)|x^n=\frac{1-2x}{1-3x+x^2}$ (see \cite[Seq. A001519]{Sl}).
Summing all contributions, we have for $m\ge 2$,
\[
G_m(x)= x^m + (m-1)x^{m-1}\big(F_T(x)-1-xF_T(x)\big) + m x^m(F_T(x)-1) +(m-1)\left(\frac{x^{m+1}}{1-x}(K-1)\right).
\]
Since $F:=F_T(x)=\sum_{m\geq0}G_m(x)$, we find
\[
F=1 +xF+\frac{x}{(1-x)^{2}}(F-1-xF) +\frac{(2-x)x^{2}}{(1-x)^{2}}(F-1)+\frac{x^{2}}{1-x}+
\frac{x^{3}}{(1-x)^{3}}(K-1),
\]
with solution the stated $F_T(x)$.
\end{proof}

\subsection{Case 177}
The two representative triples $T$ are:

\{2143,2341,2413\} (Theorem \ref{th177A1})

\{2143,2341,3241\} (Theorem \ref{th177A2})

\begin{theorem}\label{th177A1}
Let $T=\{2143,2341,2413\}$. Then
 $$F_T(x)=\frac{1-4x+3x^2-x^3}{1-5x+6x^2-3x^3}.$$
\end{theorem}
\begin{proof}
Let $G_m(x)$ be the generating function for $T$-avoiders with $m$ left-right maxima.
Clearly, $G_0(x)=1$ and $G_1(x)=xF_T(x)$. Now let us write an equation for $G_m(x)$ with $m\geq2$.
So suppose $\pi=i_1\pi^{(1)}i_2\pi^{(2)}\cdots i_m\pi^{(m)}\in S_n(T)$ with $m\ge 2$ left-right maxima. We consider the following three cases:
\begin{itemize}
\item $\pi^{(1)}\neq\emptyset$.
Here, the only letters occurring after $i_2$ that are $>i_1$ are $i_3,\dots,i_m$ (to avoid 2143) and no letter occurring after $i_3$ is $<i_1$ (to avoid 2341). So $\pi$ has the form $i_1 \pi^{(1)} (i_1 +1) \gamma^{(2)}(i_1+2)\cdots n$ with $\gamma^{(2)}<i_1$. Thus, $(i_1+2)\cdots n$ contributes $x^{m-2}$ and deleting these letters is a bijection to $T$-avoiders of length $r$ for some $r\geq2$ with first letter $r-1$ and second letter $\ne r$, contributing $x(F_T-1)-x^2 F_T$.

\item $\pi^{(1)}=\emptyset$ and $\pi^{(2)}$ has a letter $a<i_1$. Here, no non-left-right max letter occurring after $i_3$ (if present) is $>i_1$ (2143) and, again, no letter occurring after $i_3$ is $<i_1$ (2341). Also, in $\pi_2$, all letters $>i_1$ occur before all letters $<i_1$ (2413).
So $\pi$ has the form $i_1 i_2\gamma^{(2)}\gamma^{(1)} (i_2+1)\cdots n$ with $\gamma^{(1)}<i_1<\gamma^{(2)}<i_2$ and, furthermore, $\gamma^{(2)}$ is decreasing because $b<c$ in $\gamma^{(2)}$ implies $i_1bca$ is a 2341. So $\pi=i_1 i_2(i_2 -1) \cdots (i_1 +1)\gamma^{(1)} (i_2+1)\cdots n$ with $\gamma^{(1)}$ a $T$-avoider of length $\in [n-m]$,  giving a contribution of $\frac{x^m}{1-x}(F_T(x)-1)$.

\item $\pi^{(1)}=\emptyset$ and $\pi^{(2)}>i_1$. This condition implies $i_1=1$ (obvious if $m=2$ and because $i_1 i_2 i_3 1$ would be a 2341 if $m\ge 3$), giving a contribution of $x G_{m-1}(x)$.
\end{itemize}

Summing the contributions, we have for $m\geq2$,
$$G_m(x)=x^{m-1}(F_T(x)-1-xF_T(x))+\frac{x^m}{1-x}(F_T(x)-1)+xG_{m-1}(x).$$
Since $F_T(x)=\sum_{m\geq0}G_m(x)$, we find that
$$F_T(x)=1+xF_T(x)+\frac{x}{1-x}(F_T(x)-1-xF_T(x))+\frac{x^2}{(1-x)^2}(F_T(x)-1)+x(F_T(x)-1),$$
which, by solving for $F_T(x)$, completes the proof.
\end{proof}

\begin{theorem}\label{th177A2}
Let $T=\{2143,2341,3241\}$. Then
 $$F_T(x)=\frac{1-4x+3x^2-x^3}{1-5x+6x^2-3x^3}.$$
\end{theorem}
\begin{proof}
Let $G_m(x)$ be the generating function for $T$-avoiders with $m$ left-right maxima.
Clearly, $G_0(x)=1$ and $G_1(x)=xF_T(x)$. Now let us write an equation for $G_m(x)$ with $m\geq2$.
So suppose $\pi=i_1\pi^{(1)}i_2\pi^{(2)}\cdots i_m\pi^{(m)}\in S_n(T)$ with $m\ge 2$ left-right maxima. We consider the following three cases:
\begin{itemize}
\item $\pi^{(1)}\neq\emptyset$.
Since $T$ contains 2143 and 2341, $\pi$ has the form $i_1 \pi^{(1)} (i_1 +1) \gamma^{(2)}(i_1+2)\cdots n$ with $\gamma^{(2)}<i_1$, as in Theorem \ref{th177A1}. Furthermore, $\pi^{(1)} < \gamma^{(2)}$ for else $i_1,i_1+1$ are the 3 and 4 of a 3241. Hence, if $\pi^{(1)}$ is increasing, then $\pi^{(1)}=12\cdots i$ for some $i\ge 1$ and $\pi \rightarrow \textrm{St}(\gamma^{(2)})$ is a bijection to $T$-avoiders, giving a contribution of $\frac{x^{m+1}}{1-x}F_T(x)$. On the other hand, if $\pi^{(1)}$ is not increasing, then $\gamma^{(2)}=\emptyset$ because $b>a$ in $\pi^{(1)}$ and $c \in \gamma^{(2)}$ implies $ba i_2 c$ is a 2143. So, $\pi \rightarrow \pi^{(1)}$ is a bijection to non-identity $T$-avoiders, giving a contribution of $x^{m} (F_T(x)-\frac{1}{1-x})$.

\item $\pi^{(1)}=\emptyset$ and $\pi^{(2)}$ has a letter smaller than $i_1$. Here
$\pi$ has the form $i_1 i_2 \pi^{(2)} (i_2+1)\cdots n$ (to avoid 2143,\,2341) with $i_1 \ne 1$, and $\pi \rightarrow i_1 \pi^{(2)}$ is a bijection to $T$-avoiders of length $\ge 2$ with first letter $\ne 1$,
giving a contribution of $x^{m-1}(F_T(x)-1-xF_T(x))$.

\item $\pi^{(1)}=\emptyset$ and $\pi^{(2)}>i_1$. As in Theorem \ref{th177A1}, $i_1=1$ and the contribution is $xG_{m-1}(x)$.
\end{itemize}

Summing the contributions, we have for $m\geq2$,
$$G_m(x)=\frac{x^{m+1}}{1-x}F_T(x)+x^m \left(F_T(x)-\frac{1}{1-x}\right)+x^{m-1}\big(F_T(x)-1-xF_T(x)\big)+xG_{m-1}(x).$$
Since $F_T(x)=\sum_{m\geq0}G_m(x)$, we find that
\begin{align*}
F_T(x)&=1+xF_T(x)+\frac{x^3}{(1-x)^2}F_T(x)+\frac{x^2}{1-x}\left(F_T(x)-\frac{1}{1-x}\right)\\
&\quad+\frac{x}{1-x}\big(F_T(x)-1-xF_T(x)\big)+x(F_T(x)-1),
\end{align*}
which, by solving for $F_T(x)$, completes the proof.
\end{proof}

\subsection{Case 191}
\begin{theorem}\label{th191A1}
Let $T=\{1342,2134,2413\}$. Then
 $$F_T(x)=\frac{(1-x)(1-2x)(1-3x)}{1-7x+16x^2-14x^3+3x^4}.$$
\end{theorem}
\begin{proof}
Let $G_m(x)$ be the generating function for $T$-avoiders with $m$ left-right maxima.
Clearly, $G_0(x)=1$ and $G_1(x)=xF_T(x)$. Now let us write an equation for $G_m(x)$ with $m\geq2$.

For $m=2$, let $\pi=i\pi'n\pi''\in S_n(T)$ with two left-right maxima.
The entries after $n$ and $>i$ all precede entries $<i$ or else $i,n$ are the ``2,4''of a 2413.
Hence, $\pi=i\pi'n \beta'\beta''$ with $\beta'>i>\beta''$.

If $\pi'=\emptyset$ so that $\pi=in\beta'\beta''$, then $\beta'$ avoids 231 or $i$ would start a 1342. So $\beta'$ avoids 2134 and 231 (which subsumes 2413), and $\beta''$ avoids $T$.
The \gf $K(x)$ for $\{231,2134\}$-avoiders is $K(x)=1+\frac{x(1-3x+3x^2)}{(1-x)(1-2x)^2}$ \cite[Seq. A005183]{Sl}.
Thus the contribution is $x^2K(x)F_T(x)$.

If $\pi'\ne \emptyset$, then $\beta'$ is decreasing  (or $i\,\max(\pi')$ would start a 2134) and
St($\pi'n\beta'') $ is a $T$-avoider that does not start with its max. Thus, by deleting $i$, we have a contribution of $\frac{x}{1-x}\big(F_T(x)-1-xF_T(x)\big)$.

Hence,
$$G_2(x)=x^2K(x)F_T(x)+\frac{x}{1-x}\big(F_T(x)-1-xF_T(x)\big)\, .$$

Now, suppose $m\geq3$ and 
$\pi$ avoids $T$ with $m$ left-right maxima $i_1<i_2<\cdots<i_m$. Then $\pi$ has the form shown in the figure below with shaded regions empty to avoid a pattern involving the gray bullet as indicated, and entries in $\beta'$ preceding entries in $\beta''$ to avoid 2413, and similarly for $\gamma',\gamma''$.
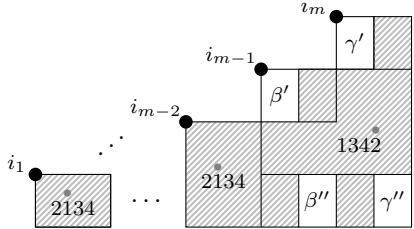
\begin{center}
\begin{pspicture}(-1.3,.1)(4,3)
\psset{xunit =1cm, yunit=.7cm,linewidth=.5\pslinewidth}
\psline(-1,0)(0,0)(0,1)(-1,1)(-1,0)
\psline(2.5,0)(3,0)
\psline(3.5,0)(4,0)(4,1)
\psline(3,3)(3,4)(3.5,4)(3.5,3)(2.5,3)
\pspolygon[fillstyle=hlines,hatchcolor=lightgray,hatchsep=0.8pt](-1,0)(0,0)(0,1)(-1,1)(-1,0)
\psline[fillstyle=hlines,hatchcolor=lightgray,hatchsep=0.8pt](1,0)(2.5,0)(2.5,1)(4,1)(4,4)(3.5,4)(3.5,3)(2.5,3)(2.5,2)(1,2)(1,0)
\pspolygon[fillstyle=hlines,hatchcolor=lightgray,hatchsep=0.8pt](3,0)(3.5,0)(3.5,1)(3,1)(3,0)
\psline(2,0)(2,3)(2.5,3)
\psline(2,1)(2.5,1)
\psline(2.5,2)(3,2)(3,3)
\psline(3.5,3)(4,3)
\pscircle*(-1,1){0.09}\pscircle*(1,2){0.09}\pscircle*(2,3){0.09}\pscircle*(3,4){0.09}
\rput(-1.25,1.2){\textrm{{\footnotesize $i_1$}}}
\rput(0.6,2.25){\textrm{{\footnotesize $i_{m-2}$}}}
\rput(1.6,3.25){\textrm{{\footnotesize $i_{m-1}$}}}
\rput(2.7,4.2){\textrm{{\footnotesize $i_m$}}}
\rput(2.25,2.5){\textrm{{\footnotesize $\beta'$}}}
\rput(2.75,0.5){\textrm{{\footnotesize $\beta''$}}}
\rput(3.25,3.5){\textrm{{\footnotesize $\gamma'$}}}
\rput(3.75,0.5){\textrm{{\footnotesize $\gamma''$}}}
\rput(-0.5,.5){\textrm{{\footnotesize $2\overset{{\gray \bullet}}{1}34$}}}
\rput(1.5,1){\textrm{{\footnotesize $2\overset{{\gray \bullet}}{1}34$}}}
\rput(3.3,1.7){\textrm{{\footnotesize $134\overset{{\gray \bullet}}{2}$}}}
\rput(0.5,0.5){$\dots$}
\rput(0,1.7){$\iddots$}
\end{pspicture}
\end{center}
We consider 4 cases according as $\beta',\beta''$ are empty or not.

If $\beta',\beta''$ are both empty, then $\gamma'$ avoids 231 (else $i_{m-1}$ is the ``1'' of a 1342) and
$\gamma''$ avoids $T$, giving a contribution of $x^m K(x)F_T(x)$.

If $\beta'\ne \emptyset,\,\beta''=\emptyset$, then $\beta'$  avoids 213 due to $i_m$ (2134) and avoids 231 due to $i_{m-2}$ (1342). The \gf $L(x)$ for $\{213,231\}$-avoiders is $L(x) =\frac{1-x}{1-2x}$ \cite{SiS}.
Also, $\gamma'$ is decreasing (2134), and $\gamma''$ avoids $T$. By deleting the left-right maxima and $\gamma'$, the contribution is $\frac{x^m}{1-x}(L(x)-1)F_T(x)$.

If $\beta' = \emptyset,\,\beta''\ne \emptyset$, then $\gamma'$ is decreasing once again and St($\beta''\,i_m\,\gamma'')$ avoids $T$ and does not start with its max. Deleting $i_1,\dots,i_{m-1}$ and $\gamma'$, the contribution is $\frac{x^{m-1}}{1-x}\big(F_T(x)-1-xF_T(x)\big)$.

If $\beta',\beta''$ are both nonempty, then $\beta'$ avoids 213 and 231, $\gamma'$ is decreasing, and St($\beta''\,i_m\,\gamma'')$ avoids $T$ and does not start with its max. Again deleting $i_1,\dots,i_{m-1}$ and $\gamma'$, the contribution is $\frac{x^{m-1}}{1-x}(L(x)-1)\big(F_T(x)-1-xF_T(x)\big)$.

Hence, for $m\ge 3$,
$$G_m(x)=x^mK(x)F_T(x)+\frac{x^m}{1-x}(L(x)-1)F_T(x)+\frac{x^{m-1}}{1-x}L(x)\big(F_T(x)-1-xF_T(x)\big).$$
Since $F_T(x)=\sum_{m\geq0}G_m(x)$, we obtain
\begin{align*}
F_T(x)&=1+xF_T(x)+x^2K(x)F_T(x)+\frac{x}{1-x}\big(F_T(x)-1-xF_T(x)\big)\\
&\quad+\frac{x^3}{1-x}K(x)F_T(x)+\frac{x^3}{(1-x)^2}(L(x)-1)F_T(x)+\frac{x^2}{(1-x)^2}L(x)\big(F_T(x)-1-xF_T(x)\big).
\end{align*}
By solving for $F_T(x)$, we complete the proof.
\end{proof}

\subsection{Case 196}
The two representative triples $T$ are:

\{2143,2431,3241\} (Theorem \ref{th196A3})

\{2413,2431,3214\} (Theorem \ref{th196A4})
\begin{theorem}\label{th196A3}
Let $T=\{2143,2431,3241\}$. Then
 $$F_T(x)=\frac{1-5x+7x^2-4x^3}{1-6x+11x^2-9x^3+2x^4}\, .$$
\end{theorem}
\begin{proof}
Let $G_m(x)$ be the generating function for $T$-avoiders with $m$ left-right maxima.
Clearly, $G_0(x)=1$ and $G_1(x)=xF_T(x)$. For $m\geq2$, we need a simple lemma.
Let $S_{n,m}$ denote the set of all permutations $\pi=i_1\pi^{(1)}i_2\pi^{(2)}\cdots i_m\pi^{(m)}
\in S_n$ with $m$ left-right maxima $i_1,i_2,\dots,i_m$ and $R_{n,m}$ the subset that satisfy the condition
(*) $\pi^{(1)}=\emptyset$ and $\pi^{(j)}<i_{j-1}$ for all $j=2,3,\ldots,m$.  Let $S_{n,m}(T)$ and $R_{n,m}(T)$ have the obvious meaning.
\begin{lemma}\label{lemma196}
For $m\ge 2$, the map $\phi:\:i_1 \pi^{(1)}i_2\pi^{(2)} \cdots i_m\pi^{(m)} \rightarrow i_1 \pi^{(2)}i_2\pi^{(3)} \cdots i_{m-1}\pi^{(m)}$ is a bijection from $R_{n,m}$ to $S_{n-1,m-1}$. Furthermore, the restriction of $\phi$ to $R_{n,m}(T)$ is a bijection to  $S_{n-1,m-1}(T). $\qed
\end{lemma}
Now suppose $\pi \in S_{n,m}(T)$ with $m\ge 2$ and consider cases according as condition (*) holds or not.

If (*) holds, so that $\pi \in R_{n,m}(T)$, then the contribution is $xG_{m-1}(x)$, by the lemma.
If (*) does not hold, then there exists an index $s \in [m]$ such that $\pi^{(s)}$ contains a letter
between $i_{s-1}$ and $i_s$. This imposes restrictions on $\pi$ as illustrated:
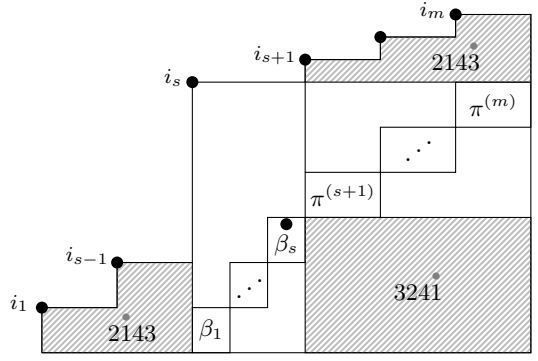
\begin{center}
\begin{pspicture}(-4,-.1)(13,4.2)
\psset{xunit =.5cm, yunit=.3cm,linewidth=.5\pslinewidth}
\psline(0,0)(13,0)(13,15)(11,15)(11,14)(9,14)(9,13)(7,13)(7,12)(4,12)(4,4)(2,4)(2,2)(0,2)(0,0)
\psline(4,0)(4,2)(5,2)(5,4)(6,4)(6,6)(7,6)(9,6)(9,8)(11,8)(11,10)(13,10)(13,12)
\psline(4,0)(5,0)(5,2)(6,2)(6,4)(7,4)(7,6)(7,8)(9,8)(9,10)(11,10)(11,12)(13,12)
\psline(7,8)(7,12)
\pspolygon[fillstyle=hlines,hatchcolor=lightgray,hatchsep=0.8pt](7,0)(13,0)(13,6)(7,6)(7,0)
\pspolygon[fillstyle=hlines,hatchcolor=lightgray,hatchsep=0.8pt](7,12)(13,12)(13,15)(11,15)(11,14)(9,14)(9,13)(7,13)(7,12)
\pspolygon[fillstyle=hlines,hatchcolor=lightgray,hatchsep=0.8pt](0,0)(4,0)(4,4)(2,4)(2,2)(0,2)(0,0)
\rput(10,3){\textrm{{\small $324\overset{{\gray \bullet}}{1}$}}}
\rput(11,13.2){\textrm{{\small $214\overset{{\gray \bullet}}{3}$}}}
\rput(2.4,1.2){\textrm{{\small $2\overset{{\gray \bullet}}{1}43$}}}
\rput(-0.6,2.2){\textrm{{\footnotesize $i_1$}}}
\rput(1.2,4.25){\textrm{{\footnotesize $i_{s-1}$}}}
\rput(3.5,12.2){\textrm{{\footnotesize $i_{s}$}}}
\rput(6.2,13.2){\textrm{{\footnotesize $i_{s+1}$}}}
\rput(10.4,15.2){\textrm{{\footnotesize $i_{m}$}}}
\rput(5.5,3.2){$\iddots$}
\rput(10,9.3){$\iddots$}
\rput(4.5,1){\textrm{{\small $\beta_1$}}}
\rput(6.5,4.8){\textrm{{\small $\beta_s$}}}
\rput(8,7){\textrm{{\small $\pi^{(s+1)}$}}}
\rput(12,11){\textrm{{\small $\pi^{(m)}$}}}
\rput(13.3,0){\textrm{,}}
\pscircle*(0,2){.08}\pscircle*(2,4){.08}\pscircle*(4,12){.08}\pscircle*(7,13){.08}
\pscircle*(9,14){.08}\pscircle*(11,15){.08}\pscircle*(6.5,5.7){.08}
\end{pspicture}
\end{center}
where dark bullets indicate mandatory entries, shaded regions are empty to avoid the pattern involving a light bullet as indicated, blank regions are empty, and the $\beta$'s and $\pi$'s are in the displayed order to avoid 2431.

Thus, the contribution is $x^m\sum_{s=1}^m L_{m,s;s}(x)$, where $L_{m,s;d}(x)$ is the generating function for such avoiders $\pi$ (as illustrated) with $\beta_1=\cdots=\beta_{s-d}=\emptyset,\ 1\le d \le s$, when the left-right maxima are understood to make no contribution, i.e., are weighted with 1 rather than $x$. Note the latter condition on the $\beta$'s is vacuous---no restriction---when $d=s$. We need to introduce $d$ because we can get a recurrence
for $L_{m,s;d}$ in terms of $L_{m,s;d-1}$ that will yield $L_{m,s;s}$.
To do so, let $2\le d \le s$ and consider whether $\beta_{s-d+1}$ is empty or not.
If $\beta_{s-d+1}=\emptyset$, clearly the contribution is $L_{m,s;d-1}(x)$.
If $\beta_{s-d+1}\ne \emptyset$ then, to avoid 2143, $\beta_{i}$ is increasing
(could be empty) for $s-d+2 \le i \le s-1$ while $\beta_{s}$ is increasing and nonempty. Moreover, also to avoid 2143 (but utilizing different letters), $\pi^{(j)}=\emptyset$ for all $j=s+1,s+2,\ldots,m$. There are $d-1$ $\beta$'s required to be increasing and so the contribution is
$(F_T(x)-1)\frac{x}{(1-x)^{d-1}}$.

Adding the two contributions, we have
\begin{equation}\label{eq196}
L_{m,s;d}(x)=L_{m,s;d-1}(x)+(F_T(x)-1)\frac{x}{(1-x)^{d-1}}\, .
\end{equation}
To complete the recurrence, we need an expression for $L_{m,s;1}(x)$. Here, $\beta_1,\dots,\beta_{s-1}$ are all empty. Set $r=m-s,\ M_r=L_{m,s;1}(x)$ and relabel $\beta$'s and $\pi$'s so that the boxes not required to be empty for $M_r$ contain $\beta_{0}\ne \emptyset,\pi^{(1)}, \dots, \pi^{(r)}$. Now consider variable $r$. Clearly, $M_0 = F_T(x)-1$ and we obtain a recurrence for $M_r,\ r\ge 1,$ conditioning on the first nonempty $\pi^{(j)}$. If all $\pi$'s are empty, the contribution is $F_T(x)-1$.
Otherwise, let $j\in [r]$ be minimal with $\pi^{(j)}\ne \emptyset$.
Then $\beta_{0}$ is increasing ($b>a$ in $\beta_{0}$ would make $ba$ the 21 of a 2143) and
the contribution is $\frac{x}{1-x}M_{r-j}$ (because $\pi^{(j)}$ can play the role of $\beta_{0}$).
So $M_r=F_T(x)-1 + \frac{x}{1-x} \sum_{j=1}^{r}M_{r-j}$ for $r\ge 1$, with $M_0 = F_T(x)-1$. This recurrence has solution $M_r= \frac{F_T(x)-1}{(1-x)^r}$.
So $L_{m,s;1}(x)= \frac{F_T(x)-1}{(1-x)^{m-s}}$, the initial condition for recurrence (\ref{eq196}), with solution
\[
L_{m,s;d}(x) = \big(F_T(x)-1\big)\left( \frac{1}{(1-x)^{d-1}} + \frac{1}{(1-x)^{m-s}} - 1\right)\, .
\]
Hence,
\begin{align*}
G_m(x) &=  xG_{m-1}(x)+ x^m\sum_{s=1}^m L_{m,s;s}(x) \\
&=  xG_{m-1}(x)+x^m\sum_{s=1}^m \big(F_T(x)-1\big)\left(\frac{1}{(1-x)^{s-1}}+
\frac{1}{(1-x)^{m-s}}-1\right),
\end{align*}
which implies
$$G_m(x)=xG_{m-1}(x)+x^m\frac{\big(F_T(x)-1\big)\big(2(1-x)^{1-m}-mx-2(1-x)\big)}{x}\, .$$
By summing over all $m\geq2$ and using the initial condition $G_0(x)=1$ and $G_1(x)=xF_T(x)$, we obtain
$$F_T(x)-1-x F_T(x)=x\big(F_T(x)-1\big)+\frac{(2x^2-3x+2)(F_T(x)-1)x^2}{(1-2x)(1-x)^2}\, ,$$
with solution the desired $F_T$.
\end{proof}

\begin{theorem}\label{th196A4}
Let $T=\{2413,2431,3214\}$. Then
 $$F_T(x)=\frac{1-5x+7x^2-4x^3}{1-6x+11x^2-9x^3+2x^4}\, .$$
\end{theorem}
\begin{proof}
Let $a_n=|S_n(T)|$ and let $a_n(i_1,i_2,\ldots,i_s)$ denote the number of permutations $i_1i_2\cdots i_s\pi\in S_n(T)$. We will obtain expressions for $a_n(i,j)$ and $a_n(i)$ and deduce a recurrence for $a_n$. Clearly, $a_n(1)=a_n(n)=a_{n-1}$. For $2\le i \le n-1$, we have the following expressions for $a_n(i,j)$:
\[
a_n(i,j)=
 \begin{cases}
  a_{n-1}(i-1) & \text{if $j=1$}\, , \\
  a_{j-1} & \text{if $2 \le j<i$}\, , \\
  a_{n-1}(i)  & \text{if $j=i+1$}\, , \\
  0  & \text{if $j\ge i+2$}\, .
  \end{cases}
\]
For the first item, ``delete 1 and standardize'' is a bijection from $T$-avoiders that begin $i1$ to
one-size-smaller $T$-avoiders that begin $i-1$. For the second item, $n$ occurs before 1 (3214) and
so $\pi=ij\pi'n\pi''1\pi'''$. Also, $\pi'>j$ (3214),\ $\pi''<j$ (2431),\ $\pi'''<j$ (2413),
and $\pi'$ is increasing (2431). These results imply that
$\pi=ij(j+1)\cdots \widehat{i} \cdots n\beta$ (where $\widehat{i}$ indicates that $i$ is missing)
with $\beta \in S_{j-1}(T)$. The easy proofs of the last two items are left to the reader.

Since $a_n= \sum_{i=1}^{n}a_n(i) = a_n(1) + \sum_{i=2}^{n-1}\sum_{j=1}^{n}a_n(i,j) +a_n(n)$, the preceding results yield
$$a_n=\sum_{i=1}^{n-3}(n-2-i)a_i+2\sum_{i=1}^{n-1}a_{n-1}(i)+2(a_{n-1}-a_{n-2})$$
for $n\ge 3$, which implies
$$a_n=4a_{n-1}-2a_{n-2}+\sum_{i=1}^{n-3}(n-2-i)a_i\, ,$$
with $a_0=a_1=1$ and $a_2=2$. Since $F_T(x)=\sum_{n\geq0}a_n x^n$,
the recurrence for $a_n$ translates to
$$F_T(x)-1-x-2x^2=4x\big(F_T(x)-1-x\big)-2x^2\big(F_T(x)-1\big)+\frac{x^3}{(1-x)^2}\big(F_T(x)-1\big)\, ,$$
with solution the desired $F_T$.
\end{proof}

\subsection{Case 201}
The two representative triples $T$ are:

\{1243,1324,3142\} (Theorem \ref{th201A1})

\{1342,1423,2314\} (Theorem \ref{th201A2})
\subsubsection{$\mathbf{T=\{1243,1324,3142\}}$}
\begin{theorem}\label{th201A1}
Let $T=\{1243,1324,3142\}$. Then
$$F_T(x)=\frac{1-3x+x^2}{1-x}C^3(x).$$
\end{theorem}
\begin{proof}
Let $G_m(x)$ be the generating function for $T$-avoiders with $m$ left-right maxima.
Clearly, $G_0(x)=1$ and $G_1(x)=xF_T(x)$. Now suppose $m\geq2$ and $\pi=i_1\pi^{(1)}\cdots i_m\pi^{(m)}$ avoids $T$. Then $\{i_2,i_3,\dots,i_m=n\}$ are consecutive integers (a gap would give a 1324 or a 1243). We consider two cases.
\begin{itemize}
\item $i_1=n-m+1$, its maximum possible value. Here, $\pi^{(1)}>\pi^{(2)}>\cdots>\pi^{(m)}$ (to avoid 3142).
Also, for $j\in [ m-1]$,
 $\pi^{(j)}$ avoids $132$ (or $i_{j+1}$ is the ``4'' of a 1324), and $\pi^{(m)}$ avoids $T$.
Hence, the contribution is $x^mC(x)^{m-1}F_T(x)$.

\item $i_1<n-m+1$. Here, $i_1$ and $i_2$ are not consecutive and $\pi$ has the form
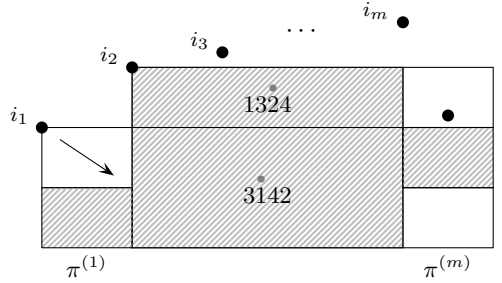
\begin{center}
\begin{pspicture}(-1.5,-.8)(8,3.7)
\psset{xunit =.6cm, yunit=.4cm,linewidth=.5\pslinewidth}
\psline[linecolor=gray](2,4)(8,4)
\psline(0,4)(0,2)(2,2)(2,0)(10,0)(10,2)(8,2)(8,4)(10,4)(10,6)(2,6)(2,4)(0,4)
\psline(8,0)(8,2)(10,2)
\psline(10,4)(8,4)(8,6)
\psline(0,2)(2,2)(2,4)
\rput(1,-0.6){\textrm{{\small $\pi^{(1)}$}}}
\rput(9,-0.6){\textrm{{\small $\pi^{(m)}$}}}
\pspolygon[fillstyle=hlines,hatchcolor=lightgray,hatchsep=0.8pt](2,0)(8,0)(8,6)(2,6)(2,0)
\pspolygon[fillstyle=hlines,hatchcolor=lightgray,hatchsep=0.8pt](8,2)(10,2)(10,4)(8,4)(8,2)
\psline[fillstyle=hlines,hatchcolor=lightgray,hatchsep=0.8pt](0,0)(2,0)(2,2)(0,2)(0,0)
\psline(2,4)(8,4)
\pscircle*(0,4){.08}\pscircle*(2,6){.08}\pscircle*(4,6.5){.08}\pscircle*(8,7.5){.08}
\pscircle*(9,4.4){.08}
\rput(5.8,7.2){$\cdots$}
\rput(5,2){\textrm{{\small $3\overset{{\gray \bullet}}{1}42$}}}
\rput(5,5){\textrm{{\small $13\overset{{\gray \bullet}}{2}4$}}}
\psline[arrows=->,arrowsize=3pt 3](.4,3.6)(1.6,2.4)
\rput(-0.5,4.4){\textrm{{\footnotesize $i_1$}}}
\rput(1.5,6.4){\textrm{{\footnotesize $i_{2}$}}}
\rput(3.5,6.9){\textrm{{\footnotesize $i_{3}$}}}
\rput(7.4,7.9){\textrm{{\footnotesize $i_{m}$}}}
\rput(10.3,0){,}
\end{pspicture}
\end{center}
where the top shaded rectangle is empty (1324) and hence $\pi^{(m)}$ contains $i_1+1$; the rectangle below it is empty (3142);
$\pi^{(1)}$ is decreasing (or $i_m(i_1+1)$ would terminate a 1243) and $i_1\pi^{(1)}$ has no gaps (else there exist $a<b<i_1$ with $a \in \pi^{(1)}$ and $b \in \pi^{(m)}$ and then $i_1ai_mb$ is a 3142). Also, $i_1\pi^{(m)}$ avoids $T$ and does not start with its maximal letter. Hence, the contribution is
$\frac{x^{m-1}}{1-x}(F_T(x)-1-xF_T(x))$.
\end{itemize}
Combining the preceding cases gives $G_m(x)$ for $m\geq2$, and
by summing over all $m\geq0$, we obtain
$$F_T(x)=1+\sum_{m\geq1} x^mC(x)^{m-1}F_T(x)+\sum_{m\geq2}\frac{x^{m-1}}{1-x}\big(F_T(x)-1-xF_T(x)\big),$$
which implies
$$F_T(x)=1+xC(x)F_T(x)+\frac{x}{(1-x)^2}\big(F_T(x)-1-xF_T(x)\big),$$
with solution $F_T(x)=\frac{1-3x+x^2}{(1-x)\big(1 - 2 x + (x^2 - x) C(x)\big)}$, equivalent to the stated expression.
\end{proof}

\subsubsection{$\mathbf{T=\{1342,1423,2314\}}$}
To enumerate the members of $S_n(T)$, we consider the relative positions of the letters $n$ and $n-1$ within a permutation.  More precisely, given $1 \leq i,j \leq n$ with $i \neq j$, let $a(n;i,j)$ denote the number of permutations $\pi=\pi_1\pi_2\cdots \pi_n \in S_n(T)$ such that $\pi_i=n$ and $\pi_j=n-1$.  If $n\geq 2$ and $1 \leq i \leq n$, then let $a(n;i)=\sum_{j=1}^n a(n;i,j)$, with $a(1;1)=1$.  The array $a(n;i,j)$ is determined by the following recurrence relations.

\begin{lemma}\label{201bl1}
If $n \geq 3$, then
\begin{equation}\label{201bl1e1}
a(n;i,i-1)=a(n-1;i-1,i-2)+\sum_{j=2}^{i-2}a(n-1;i-1,j), \qquad 3 \leq i \leq n,
\end{equation}
and
\begin{equation}\label{201bl1e2}
a(n;i,j)=a(n-1;i,j)+a(n-1;i-1,j-1)+\sum_{k=j+1}^{n-1}a(n-1;j-1,k), \qquad 2 \leq i \leq j-2.
\end{equation}
Furthermore, we have $a(n;i,1)=a(n-1;i-1)$ for $2 \leq i \leq n$, $a(n;1,j)=a(n-1;j-1)$ for $2 \leq j \leq n$, $a(n;i,j)=a(n-1;i-1,j)$ for $2 \leq j \leq i-2$, and $a(n;i,i+1)=a(n-1;i)$ for $1 \leq i \leq n-1$.
\end{lemma}
\begin{proof}
Throughout, let $\pi=\pi_1\pi_2\cdots\pi_n \in S_n(T)$ be of the form enumerated in the case under consideration.  The formulas for $a(n;i,1)$ and $a(n;1,j)$ follow from the fact that an initial letter $n-1$ or $n$ within a member of $S_n(T)$ may be safely deleted.  To determine $a(n;i,j)$ where $2 \leq j \leq i-2$, first note that within $\pi \in S_n(T)$ in this case, the letter $n-2$ cannot go to the left of $n-1$ (for if it did, then there would be an occurrence of $2314$ of the form $(n-2)(n-1)xn$ for some $x<n-2$).  Furthermore, the letter $n-2$ cannot go to the right of $n$, for otherwise there would be an occurrence of $1342$ of the form $x(n-1)n(n-2)$ for some $x<n-2$ (since $j \geq2$ implies $n-1$ is not the first letter).  Thus, $n-2$ must go between $n-1$ and $n$ in this case.  Note also that $\min\{\pi_{j+1},\pi_{j+2},\ldots,\pi_{i-1}\}>\max\{\pi_1,\pi_2,\ldots,\pi_{j-1}\}$ so as to avoid an occurrence of $2314$ (of the form $x(n-1)yn$).  Thus, $j\geq 2$ implies the section $\pi_{j+1}\pi_{j+2}\cdots\pi_{i-1}$ is decreasing in order to avoid $1423$, whence $\pi_{j+1}=n-2$.  It follows that the letter $n-2$ may be deleted, which implies $a(n;i,j)=a(n-1;i-1,j)$ if $1<j<i-1$.  Next, observe that $a(n;i,i+1)=a(n-1;i)$ since the letter $n-1$ is extraneous in this case and may be deleted (as none of the patterns in $T$ contain ``4'' directly followed by ``3'').

We now show \eqref{201bl1e1}.  Note that the letter $n-2$ in this case must occur to the left of $n-1$, for otherwise there would be a $1342$.  If $\pi_1=n-2$, there are $a(n-1;i-1,i-2)$ possibilities as the letter $n-2$ may be deleted since it cannot play the role of a ``2''  within a $2314$.  So suppose $\pi_j=n-2$ for some $2 \leq j \leq i-2$.  Then we must have $\min\{\pi_{j+1},\pi_{j+2},\ldots,\pi_{i-2}\}>\max\{\pi_1,\pi_2,\ldots,\pi_{j-1}\}$ in order to avoid $2314$, with $\min\{\pi_1,\pi_2,\ldots,\pi_{j-1}\}>\max\{\pi_{i+1},\pi_{i+2},\ldots,\pi_n\}$ to avoid $1342$.  Since all of the same restrictions on $\pi$ are seen to apply if we delete $n$, it follows that there are $\sum_{j=2}^{i-2}a(n-1;i-1,j)$ possibilities if $n-2$ does not start a permutation.  Combining this case with the previous implies formula \eqref{201bl1e1}.

Finally, to show \eqref{201bl1e2}, it is convenient to write $\pi \in S_n(T)$ enumerated by $a(n;i,j)$ when $1<i<j-1$ as $\pi=w^{(1)}w^{(2)}\cdots w^{(r)}$, where $w^{(i)}$ for $i<r$ denotes the sequence of letters of $\pi$ between the $i$-th and the $(i+1)$-st left-right minimum, including the former but excluding the latter (with $w^{(r)}$ comprising all letters to the right of and including the rightmost left-right minimum).  Observe that $n$ must be the final letter of some $w^{(\ell)}$.  For if not, then $j>i+1$ implies that there would be an occurrence of $1423$ of the form $xny(n-1)$, where $x$ is a left-right minimum and $y$ is not.  Then $n-2$ must be the first letter of $\pi$ or go to the right of $n$, for otherwise, $\pi$ would contain an occurrence of $2314$ of the form $x(n-2)y(n-1)$, where $x$ is the first letter and $y$ is the successor of $n$ (and hence a left-right minimum).  If $n-2$ is the first letter, then it is seen to be extraneous (since $n-1$ occurs to the right of $n$ within $\pi$) and thus may be deleted, yielding $a(n-1;i-1,j-1)$ possibilities.  If $n-2$ occurs to the right of $n$, then it must also occur to the right of $n-1$ in order to avoid $1423$.  If $\pi_{j+1}=n-2$, then it is seen that $n-2$ may be deleted as there can be no possible occurrence of a pattern in $T$ involving both $n-2$ and $n-1$ in this case, whence there are $a(n-1;i,j)$ possibilities.  On the other hand, if $\pi_k=n-2$ for some $k>j+1$, then the letter $n-1$, like $n$, must be the last letter of some $w^{(\ell)}$ in order to avoid $1423$.

We claim that the letter $n$ may be deleted in this case.  First note that $i>1$ implies $n$ belongs to the leftmost $w^{(\ell)}$ such that $w^{(\ell)}$ is not of length one (for otherwise, there would be an occurrence of $2314$, with $n$ playing the role of the ``4'').  If $s$ denotes the index of this $w^{(\ell)}$, then $w^{(s)}$ must be of length two, for if not and $w^{(s)}$ contained a third letter, then $\pi$ would contain $2314$, as witnessed by the subsequence $xyz(n-1)$, where $x$ and $z$ are the first letters of $w^{(s)}$ and $w^{(s+1)}$ and $y$ is the second letter of $w^{(s)}$.  It follows that the letters to the left of $n$ within $\pi$ form a decreasing sequence.  By similar reasoning, the letters between $n$ and $n-1$ are decreasing since $\pi_k=n-2$ for some $k>j+1$.  Since $\min\{\pi_{1},\pi_{2},\ldots,\pi_{i-1}\}>\max\{\pi_{i+1},\pi_{i+2},\ldots,\pi_{j-1}\}$
in order to avoid $1423$, it follows that the letters to the left of $n-1$ excluding $n$ form a decreasing sequence.  From this, it is seen that the letter $n$ may be deleted, which gives $\sum_{k=j+1}^{n-1}a(n-1;j-1,k)$ additional possibilities.  Combining this with the previous cases implies \eqref{201bl1e2} and completes the proof.
\end{proof}

Define the functions $b_{n,i}(v)=\sum_{j=i+2}^na(n;i,j)v^j$ for $1 \leq i \leq n-2$ and $c_{n,i}(v)=\sum_{j=2}^{i-1}a(n;i,j)v^j$ for $3 \leq i \leq n$.  Then recurrences \eqref{201bl1e2} and \eqref{201bl1e1} imply
\begin{equation}\label{201be1}
b_{n,i}(v)=b_{n-1,i}(v)+vb_{n-1,i-1}(v)+\sum_{j=i+2}^nb_{n-1,j-1}(1)v^j, \qquad 2 \leq i \leq n-2,
\end{equation}
and
\begin{equation}\label{201be2}
c_{n,i}(v)=c_{n-1,i-1}(v)+c_{n-1,i-1}(1)v^{i-1}+a(n-1;i-1,i-2)v^{i-1}, \qquad 3 \leq i \leq n.
\end{equation}

Let $a_n(u)=\sum_{i=1}^na(n;i)u^i$ for $n\geq1$, $b_n(u,v)=\sum_{i=2}^{n-2}b_{n,i}(v)u^i$ for $n\geq4$, $c_n(u,v)=\sum_{i=3}^nc_{n,i}(v)u^i$ for $n\geq3$, and $d_n(u)=\sum_{i=2}^na(n;i,i-1)u^i$ for $n\geq 2$.  Let $a(n)=a_n(1)$ for $n \geq 1$, with $a(0)=1$.

By Lemma \ref{201bl1}, we have
$$\sum_{i=2}^{n-1}a(n;i,i+1)u^i=\sum_{i=2}^{n-1}a(n-1;i)u^i=a_{n-1}(u)-a(n-2)u, \qquad n \geq 2,$$
and
$$\sum_{i=2}^na(n;i,1)u^i=\sum_{i=2}^na(n-1;i-1)u^i=ua_{n-1}(u), \qquad n \geq 2.$$
Thus, by the definitions, we have
\begin{equation}\label{201be3}
a_n(u)=u(a(n-1)-a(n-2))+(1+u)a_{n-1}(u)+b_n(u,1)+c_n(u,1), \qquad n \geq2,
\end{equation}
with $a_1(u)=u$, upon considering separately the cases of $a(n;i,j)$ when $i=1$, $j=1$ or both $i,j>1$.

Note that by the definitions,
$$b_{n,1}(v)=\sum_{j=3}^na(n;1,j)v^j=\sum_{j=3}^na(n-1;j-1)v^j=v(a_{n-1}(v)-a(n-2)v), \qquad n \geq 3.$$
Multiplying both sides of \eqref{201be1} by $u^i$, and summing over $2 \leq i \leq n-2$, then yields
\begin{align}
b_n(u,v)&=b_{n-1}(u,v)+uv(b_{n-1}(u,v)+ub_{n-1,1}(v))+\sum_{j=3}^nb_{n-1,j-1}(1)v^j\sum_{i=2}^{j-2}u^i\notag\\
&=(1+uv)b_{n-1}(u,v)+u^2v^2(a_{n-2}(v)-a(n-3)v)\notag\\
&\quad+\frac{v}{1-u}(u^2b_{n-1}(v,1)-b_{n-1}(uv,1)), \qquad n \geq4.\label{201be4}
\end{align}
Multiplying both sides of \eqref{201be2} by $u^i$, and summing over $3 \leq i \leq n$, gives
\begin{equation}\label{201be5}
c_n(u,v)=u(c_{n-1}(u,v)+c_{n-1}(uv,1)+d_{n-1}(uv)), \qquad n \geq3.
\end{equation}
Finally, using recurrence \eqref{201bl1e1} and noting $a(n;2,1)=a(n-2)$, we get
\begin{align}
d_n(u)&=a(n-2)u^2+ud_{n-1}(u)+\sum_{i=3}^nc_{n-1,i-1}(1)u^i\notag\\
&=a(n-2)u^2+ud_{n-1}(u)+uc_{n-1}(u,1), \qquad n \geq 2. \label{201be6}
\end{align}

Define the generating functions $a(x;u)=\sum_{n\geq1}a_n(u)x^n$, $b(x;u,v)=\sum_{n\geq4}b_n(u,v)x^n$, $c(x;u,v)=\sum_{n\geq3}c_n(u,v)x^n$ and $d(x;u)=\sum_{n\geq2}d_n(u)x^n$.  Rewriting recurrence \eqref{201be3}--\eqref{201be6} in terms of generating functions yields the following system of functional equations.

\begin{lemma}\label{201bl2}
We have
\begin{equation}\label{fe1201b}
(1-x-xu)a(x;u)=xu(1-x)(1+a(x;1))+b(x;u,1)+c(x;u,1),
\end{equation}
\begin{equation}\label{fe2201b}
(1-x-xuv)b(x;u,v)=(xuv)^2(a(x;v)-xva(x;1)-xv)+\frac{xv}{1-u}(u^2b(x;v,1)-b(x;uv,1)),
\end{equation}
\begin{equation}\label{fe3201b}
(1-xu)c(x;u,v)=xu(c(x;uv,1)+d(x;uv)),
\end{equation}
and
\begin{equation}\label{fe4201b}
(1-xu)d(x;u)=x^2u^2(1+a(x;1))+xuc(x;u,1).
\end{equation}
\end{lemma}

We can now determine the generating function $F_T(x)$.

\begin{theorem}\label{th201A2}
Let $T=\{1342,1423,2314\}$.  Then
$$F_T(x)=\frac{1-3x+x^2}{1-x}C^3(x).$$
\end{theorem}
\begin{proof}
In our present notation, we seek $1+a(x;1)$.  By taking $u=v=1$ in \eqref{fe1201b}, \eqref{fe3201b} and \eqref{fe4201b}, and then solving the resulting system for $b(x;1,1)$, $c(x;1,1)$ and $d(x;1)$, we obtain
\begin{align*}
b(x;1,1)&=\frac{(1-x)(1-5x+6x^2-x^3)a(x;1)-x(1-4x+5x^2-x^3)}{1-3x+x^2},\\
c(x;1,1)&=\frac{x^3(1+a(x;1))}{1-3x+x^2},\\
d(x;1)&=\frac{x^2(1-2x)(1+a(x;1))}{1-3x+x^2}.
\end{align*}
Hence, equation \eqref{fe2201b} with $v=1$ can be written as
\begin{align*}
&\left(1+\frac{xu^2}{1-u}\right)b(x;u,1)=(xu)^2((1-x)a(x;1)-x)\\
&\qquad\qquad+\frac{xu^2}{1-u}\frac{(1-x)(1-5x+6x^2-x^3)a(x;1)-x(1-4x+5x^2-x^3)}{1-3x+x^2}.
\end{align*}
Applying the kernel method to this last equation, it is seen that taking $u=C(x)$ cancels out the left-hand side. This gives, after several algebraic operations, the formula
$$1+a(x;1)=\frac{2(1-3x+x^2)}{(1-x)(1-3x)+(1-x)^2\sqrt{1-4x}}=\frac{1-3x+x^2}{1-x}C^3(x),$$
as desired.
\end{proof}

\emph{Remark:} From the formula for $a(x;1)$, one can now determine $b(x;1,1)$, as well as $c(x;u,1)$ and $d(x;u,1)$, by \eqref{fe3201b} and \eqref{fe4201b}.  This in turn allows one to find $b(x;u,1)$, by \eqref{fe2201b} at $v=1$.  By \eqref{fe1201b}, one then obtains a formula for $1+a(x;u)$ which generalizes $F_T(x)$ (reducing to it when $u=1$).

\subsection{Case 203}
The two representative triples $T$ are:

\{1324,1432,3142\} (Theorem \ref{th203A1})

\{1234,1342,2314\} (Theorem \ref{th203A2})

\subsubsection{$\mathbf{T=\{1324,1432,3142\}}$}
\begin{theorem}\label{th203A1}
Let $T=\{1324,1432,3142\}$. Then
$$F_T(x)=\frac{1-x}{2-2x-(1-x-x^2)C(x)}.$$
\end{theorem}
\begin{proof}
Let $G_m(x)$ be the generating function for $T$-avoiders with $m$ left-right maxima.
Clearly, $G_0(x)=1$ and $G_1(x)=xF_T(x)$. Now let us write an equation for $G_m(x)$ with $m\geq2$.
Suppose $\pi=i_1\pi^{(1)}i_2\pi^{(2)}\cdots i_m\pi^{(m)}$ is a permutation that avoids $T$ with $m\ge 2$ left-right maxima. Then $\pi^{(j)}$ avoids $132$ for all $j=1,2,\ldots,m-1$ or else $i_m$ is the 4 of a 1324.
All the letters greater than $i_1$ in $\pi^{(m)}$ are increasing (to avoid 1432) and all the letters less than $i_1$ in $\pi^{(m)}$ are $<$ all letters in other $\pi$'s (to avoid 3142),
and $i_1>\pi^{(1)}>\pi^{(2)}>\cdots>\pi^{(m-1)}$ (see figure, where the shaded regions are empty to avoid the indicated pattern with the gray bullets).
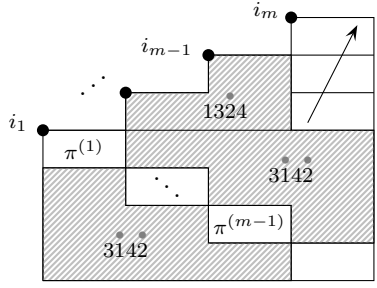
\begin{center}
\begin{pspicture}(0,.1)(4,3.6)
\psset{xunit = 1.1cm, yunit=.5cm,linewidth=.5\pslinewidth}
\psline(3,6)(4,6)
\psline(3,5)(4,5)
\psline(0,0)(0,4)(1,4)(1,5)(2,5)(2,6)(3,6)(3,7)(4,7)(4,0)(0,0)
\psline(0,4)(1,4)(1,3)(2,3)(2,2)(3,2)(3,1)(4,1)
\psline(0,3)(1,3)(1,2)(2,2)(2,1)(3,1)(3,0)
\psline(3,6)(3,4)(4,4)
\psline[arrows=->,arrowsize=3pt 3](3.2,4.2)(3.8,6.8)
\pspolygon[fillstyle=hlines,hatchcolor=lightgray,hatchsep=0.8pt](1,4)(1,3)(2,3)(2,2)(3,2)(3,1)(4,1)(4,4)(3,4)(3,6)(2,6)(2,5)(1,5)(1,4)
\pspolygon[fillstyle=hlines,hatchcolor=lightgray,hatchsep=0.8pt](0,3)(1,3)(1,2)(2,2)(2,1)(3,1)(3,0)(0,0)(0,3)
\psline(1,4)(3,4)
\rput(-0.3,4.2){\textrm{{\footnotesize $i_1$}}}
\rput(1.5,6.2){\textrm{{\footnotesize $i_{m-1}$}}}
\rput(2.7,7.2){\textrm{{\footnotesize $i_{m}$}}}
\rput(1.5,2.7){$\ddots$}
\rput(.6,5.4){$\iddots$}
\rput(2.2,4.7){\textrm{{\footnotesize $13\overset{{\gray \bullet}}{2}4$}}}
\rput(1,1){\textrm{{\footnotesize $3\overset{{\gray \bullet}}{1}4\overset{{\gray \bullet}}{2}$}}}
\rput(3,3){\textrm{{\footnotesize $3\overset{{\gray \bullet}}{1}4\overset{{\gray \bullet}}{2}$}}}

\rput(0.5,3.5){\textrm{{\footnotesize $\pi^{(1)}$}}}
\rput(2.5,1.5){\textrm{{\footnotesize $\pi^{(m-1)}$}}}
\pscircle*(0,4){0.08}\pscircle*(1,5){0.08}\pscircle*(2,6){0.08}\pscircle*(3,7){0.08}
\end{pspicture}
\end{center}

Also, at most one of the $m-1$ rectangles covered by the arrow can be occupied: $ab$ in $\pi^{(m)}$ with $b$ in a higher such rectangle than $a$ makes $ab$ the 24 of a 1324, and $b$ in a lower rectangle than $a$ makes $ab$ the 32 of a 1432. So we distinguish two cases:
\begin{itemize}
\item all of these rectangles except possibly the top one are empty, i.e., there is no letter in $\pi^{(m)}$ between $i_1$ and $i_{m-1}$. In this case $\pi^{(m)}$ can be decomposed as $\beta^{(1)}(i_{m-1}+1)\beta^{(2)}(i_{m-1}+2)\cdots\beta^{(i_m-i_{m-1}-1)}(i_m-1)\beta^{(i_m-i_{m-1})}$ such that $\pi^{(m-1)}>\beta^{(1)}>\cdots>\beta^{(i_m-i_{m-1})}$, $\beta^{(j)}$ avoids $132$ for $j=1,2,\ldots,i_m-i_{m-1}-1$ and $\beta^{(i_m-i_{m-1})}$ avoids $T$.
Since $\beta^{(j)}$ avoids $132$, each $\beta^{(j)}(i_{m-1}+j)$ contributes $xC(x)$ and since there are zero or more of them, their contribution is $\frac{1}{1-xC(x)}$.
So, this case contributes $\frac{x^mC(x)^{m-1}F_T(x)}{1-xC(x)}$.
\item There is a  letter in $\pi^{(m)}$ between $i_p$ and $i_{p+1}$ for some $p\in[m-2]$. Then $\pi^{(p+1)}=\cdots=\pi^{(m-1)}=\emptyset$ (3142) and $\pi^{(m)}$ can be decomposed as $$\beta^{(1)}(i_p+1)\beta^{(2)}(i_p+2)\cdots\beta^{(i_{p+1}-i_p-1)}(i_{p+1}-1)\beta^{(i_{p+1}-i_p)}$$ such that $\pi^{(p)}>\beta^{(1)}>\cdots>\beta^{(i_{p+1}-i_p)}$ where all except the last $\beta^{(j)}$ avoid $132$  and $\beta^{(i_{p+1}-i_p)}$ avoids $T$. This time there is at least one $\beta^{(j)}(i_p+j)$ and so we have an overall contribution of $\frac{x^{m+1}C(x)^{p+1}F_T(x)}{1-xC(x)}$.
\end{itemize}
Since $C(x)=\frac{1}{1-xC(x)}$, we find that
$$G_m(x)=x^mC(x)^mF_T(x)+\sum_{p=1}^{m-2}x^{m+1}C(x)^{p+2}F_T(x),$$
for $m\ge 2$, with $G_1(x)=xF_T(x)$ and $G_0(x)=1$.

From $F_T(x)=\sum_{m\geq0}G_m(x)$, we deduce
$$F_T(x)=1+xF_T(x)+x^2C(x)^3F_T(x)-\frac{x^2C(x)F_T(x)}{1-x}+x^2C(x)^2F_T(x),$$
with solution
$$F_T(x)=\frac{1-x}{2-2x-(1-x-x^2)C(x)}\, .$$
\end{proof}

\subsubsection{$\mathbf{T=\{1234,1342,2314\}}$}
A permutation $\pi=\pi_1\pi_2\cdots\pi_n$ is said to have an \emph{ascent} at index $i$  if $\pi_i<\pi_{i+1}$, where $1 \leq i \leq n-1$.  The letter $\pi_{i+1}$ is called an \emph{ascent top}.  In order to count the members of $S_n(T)$, we categorize them by the nature of their leftmost ascent (i.e., smallest $i$ such that $\pi_i<\pi_{i+1}$).  If $n\geq 2$ and $1 \leq i \leq n-1$, let $a(n;i)$ denote the number of $T$-avoiding permutations of length $n$ whose leftmost ascent occurs at index $i$, with $a(n;n)=1$ for $n\geq 1$ (this accounts for the permutation $n(n-1)\cdots1$, which is understood to have an ascent at index $n$).  Let $a(n)=\sum_{i=1}^n a(n;i)$ for $n\geq1$, with $a(0)=1$.

We now consider various restrictions on the ascent top corresponding to the leftmost ascent which will prove helpful in determining a recurrence for $a(n;i)$.  Let $A_{n,i}$ denote the subset of permutations of $S_n(T)$ enumerated by $a(n;i)$.  If $1 \leq i \leq n-1$, let $b(n;i)$ be the number of members of $A_{n,i}$ in which the leftmost ascent top equals $n$.  If $1 \leq i \leq n-2$, let $c(n;i)$ be the number of members of $A_{n,i}$ not starting with $n$ in which the leftmost ascent top equals $n-1$.  Finally, for $1 \leq i \leq n-2$, let $d(n;i)$ be the number of members of $A_{n,i}$ not starting with $n$ in which the leftmost ascent top is less than $n-1$.  For example, we have $b(4;2)=3$, the enumerated permutations being $2143$, $3142$ and $3241$, $c(4;1)=2$, the permutations being $1324$ and $2341$ (note that $1342$ and $2314$ are excluded), and $d(5;3)=2$, the permutations being $42135$ and $43125$.  Note that by the definitions, we have
\begin{equation}\label{203be1}
a(n;i)=a(n-1;i-1)+b(n;i)+c(n;i)+d(n;i), \qquad 1 \leq i \leq n-1,
\end{equation}
upon considering whether or not a member of $A_{n,i}$ starts with $n$.  The arrays $b(n;i)$, $c(n;i)$ and $d(n;i)$ are determined recursively as follows.

\begin{lemma}\label{203bl1}
We have
\begin{equation}\label{203bl1e1}
b(n;i)=\sum_{j=i}^{n-1}a(n-1;j), \qquad 1 \leq i \leq n-1,
\end{equation}
\begin{equation}\label{203bl1e2}
c(n;i)=\sum_{j=1}^{n-i-1}a(j-1), \qquad 1 \leq i \leq n-2,
\end{equation}
and
\begin{equation}\label{203bl1e3}
d(n;i)=c(n-1;i)+c(n-1;i-1)+d(n-1;i)+d(n-1;i-1), \qquad 1 \leq i \leq n-2.
\end{equation}
\end{lemma}
\begin{proof}
Let $B_{n,i}$, $C_{n,i}$ and $D_{n,i}$ denote the subsets of $S_n(T)$ enumerated by $b(n;i)$, $c(n;i)$ and $d(n;i)$, respectively.  For \eqref{203bl1e1}, observe that members of $B_{n,i}$ can be obtained by inserting $n$ directly after the $i$-th letter of a member of $\cup_{j=i}^{n-1}A_{n-1,j}$, with such an insertion seen not to introduce an occurrence of any of the patterns in $T$ (since the ``4'' does not correspond to the first ascent within these patterns).  This insertion operation is seen to be a bijection and hence \eqref{203bl1e1} follows.  To show \eqref{203bl1e2}, note that members $\pi \in C_{n,i}$ must be of the form
$$\pi=\alpha j(n-1) \beta n \gamma,$$
where $\alpha=j+i-1,j+i-2,\ldots,j+1$ for some $j \in [n-i-1]$, $\beta=n-2,n-3,\ldots,j+i$, and $\gamma$ is a $T$-avoider (on the letters in $[j-1]$).  The section $\alpha$ if nonempty consists of a decreasing string of consecutive numbers ending in $j+1$ in order to avoid $2314$, with all letters in $[j+i,n-2]$ required to be to the left of $n$ and all letters in $[j-1]$ required to be to the right, in order to avoid $1342$ or $2314$, respectively.  That $\beta$ is decreasing is required in order to avoid $1234$.  Furthermore, one may verify that all permutations $\pi$ of the stated form above avoid the patterns in $T$.  Considering all possible $j$, we get $\sum_{j=1}^{n-i-1}a(j-1)$ possibilities for $\pi$, which gives \eqref{203bl1e2}.

Finally, to show \eqref{203bl1e3}, first note that one can express $\sigma \in D_{n,i}$ as
$$\sigma=\sigma^{(1)}jk\sigma^{(2)}\sigma^{(3)}\sigma^{(4)},$$
where $\sigma^{(1)}$ is a decreasing sequence of length $i-1$ in $[j+1,n-1]$, $1 \leq j<k<n-1$, $\sigma^{(2)}$ is contained within $[j+1,k-1]$, $\sigma^{(3)}$ is a sequence in $[k+1,n]$ that contains $n$, and $\sigma^{(4)}$ is a permutation of $[j-1]$.  Observe that $\sigma^{(3)}$ must decrease in order to avoid $1234$ and hence starts with $n$. If $n-1$ belongs to $\sigma^{(3)}$, then removing $n$ is seen to define a bijection with $C_{n-1,i}\cup D_{n-1,i}$.  If $n-1$ belongs to $\sigma^{(1)}$, then removing $n-1$, and replacing $n$ with $n-1$, defines a bijection with $C_{n-1,i-1}\cup D_{n-1,i-1}$.  Combining the two previous cases implies \eqref{203bl1e3} and completes the proof.
\end{proof}

Let $a_n(u)=\sum_{i=1}^na(n;i)u^i$ for $n\geq1$, $b_n(u)=\sum_{i=1}^{n-1}b(n;i)u^i$ for $n\geq2$, $c_n(u)=\sum_{i=1}^{n-2}c(n;i)u^i$ for $n \geq3$, and $d_n(u)=\sum_{i=1}^{n-2}d(n;i)u^i$ for $n\geq3$.  For convenience, we take $a_0(u)=1$.

Then recurrences \eqref{203be1} and \eqref{203bl1e1} imply
\begin{equation}\label{203be2}
a_n(u)=ua_{n-1}(u)+b_n(u)+c_n(u)+d_n(u), \qquad n \geq 1,
\end{equation}
and
\begin{align}
b_n(u)&=\sum_{i=1}^{n-1}u^i\sum_{j=i}^{n-1}a(n-1;j)=\sum_{j=1}^{n-1}a(n-1;j)\sum_{i=1}^j u^i\notag\\
&=\frac{u}{1-u}(a_{n-1}(1)-a_{n-1}(u)), \qquad n \geq2. \label{203be3}
\end{align}
Multiplying both sides of \eqref{203bl1e2} by $u^i$, and summing over $1 \leq i \leq n-2$, yields
\begin{align}
c_n(u)&=\sum_{j=1}^{n-2}a(j-1)\sum_{i=1}^{n-j-1}u^i\notag\\
&=\frac{u}{1-u}\sum_{j=1}^{n-2}a(j-1)-\frac{1}{1-u}\sum_{j=1}^{n-2}a(j-1)u^{n-j}, \label{203be4} \qquad n \geq 3.
\end{align}
Finally, recurrence  \eqref{203bl1e3} gives
\begin{equation}\label{203be5}
d_n(u)=(1+u)(c_{n-1}(u)+d_{n-1}(u)), \qquad n \geq 3.
\end{equation}

Let $a(x;u)=\sum_{n\geq0}a_n(u)x^n$.  It is determined by the following functional equation.

\begin{lemma}\label{203bl2}
We have
\begin{equation}\label{203bl2e1}
\left(1+\frac{xu^2}{1-u}\right)a(x;u)=1+xu\left(\frac{1}{1-u}+\frac{x^2}{(1-x)(1-xu)(1-x-xu)}\right)a(x;1).
\end{equation}
\end{lemma}
\begin{proof}
Let $b(x;u)=\sum_{n\geq2}b_n(u)x^n$, $c(x;u)=\sum_{n\geq3}c_n(u)x^n$, and $d(x;u)=\sum_{n\geq 3}d_n(u)x^n$.  Rewriting recurrences \eqref{203be2}--\eqref{203be5} in terms of generating functions yields the following:
\begin{align*}
a(x;u)&=1+xua(x;u)+b(x;u)+c(x;u)+d(x;u),\\
b(x;u)&=\frac{xu}{1-u}(a(x;1)-a(x;u)),\\
c(x;u)&=\frac{x^3u}{(1-x)(1-xu)}a(x;1),\\
d(x;u)&=x(1+u)(c(x;u)+d(x;u)).
\end{align*}
Noting
$$c(x;u)+d(x;u)=c(x;u)+\frac{x(1+u)}{1-x(1+u)}c(x;u)=\frac{c(x;u)}{1-x(1+u)},$$
and using the expressions for $b(x;u)$ and $c(x;u)$ in the equation for $a(x;u)$, gives \eqref{203bl2e1}.
\end{proof}

We can now determine the generating function for the sequence $a(n)$.

\begin{theorem}\label{th203A2}
Let $T=\{1234,1342,2314\}$.  Then
$$F_T(x)=\frac{1-x}{2-2x-(1-x-x^2)C(x)}.$$
\end{theorem}
\begin{proof}
In the present notation, we must find $a(x;1)$.  Applying the kernel method to \eqref{203bl2e1}, and setting $u=C(x)$, gives
\begin{align*}
a(x;1)&=-\frac{(1-x)(1-u)(1-xu)(1-x-xu)}{xu(1-x)(1-xu)(1-x-xu)+x^3u(1-u)}\\
&=\frac{xu(1-x-xu)(1-x)}{x(1-x)^2-x^2(1-x)u+x^3u(1-u)}\\
&=\frac{(1-x)(1-xu)}{x+(1-x)^2-2x(1-x)u},
\end{align*}
where we have used the fact $xu^2=u-1$ several times.  Multiplying the numerator and denominator of the last expression by $u$ gives
$$a(x;1)=\frac{(1-x)(u-xu^2)}{(1-x+x^2)u-2(1-x)(u-1)}=\frac{1-x}{2-2x-(1-x-x^2)u},$$
as desired.
\end{proof}

\subsection{Case 218}
The three representative triples $T$ are:

\{1342,2314,2413\} (Theorem \ref{th218A1})

\{1324,1423,3142\} (Theorem \ref{th218A2})

\{1243,1342,2314\} (Theorem \ref{th218A3})
\subsubsection{$\mathbf{T=\{1342,2314,2413\}}$}
\begin{theorem}\label{th218A1}
Let $T=\{1342,2314,2413\}$. Then
$$F_T(x)=\frac{(1-2x)(1+\sqrt{1-4x})}{x^2+(2-4x+x^2)\sqrt{1-4x}}.$$
\end{theorem}
\begin{proof}
Let $G_m(x)$ be the generating function for $T$-avoiders with $m$ left-right maxima.
Clearly, $G_0(x)=1$ and $G_1(x)=xF_T(x)$. Now let us write an equation for $G_m(x)$ with $m\geq2$.

For $m=2$, suppose $\pi=i\pi'n\pi''\in S_n(T)$ has two left-right maxima.
In $\pi''$ all letters $>i$ occur before all letters $<i$ for else $\pi''$ contains letters $a,b$ with $a<i<b$ and $inab$ is a 2413. Thus, $\pi= i\pi'n\beta'\beta''$ with $\beta'>i>\beta''$:
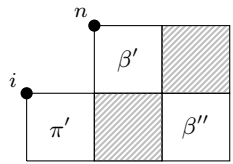
\begin{center}
\begin{pspicture}(0.3,0)(2,2)
\psset{unit =.9cm, linewidth=.5\pslinewidth}
\psline[linecolor=gray](1,0)(2,0)
\psline(0,0)(0,1)(1,1)(1,0)(0,0)
\psline(1,1)(1,2)(2,2)(2,0)(3,0)(3,1)(1,1)
\rput(.5,.5){\textrm{{\small $\pi'$}}}
\rput(1.5,1.5){\textrm{{\small $\beta'$}}}
\rput(2.5,.5){\textrm{{\small $\beta''$}}}
\pscircle*(0,1){.09}\pscircle*(1,2){.09}
\rput(-0.2,1.2){\textrm{{\footnotesize $i$}}}
\rput(0.8,2.2){\textrm{{\footnotesize $n$}}}
\pspolygon[fillstyle=hlines,hatchcolor=lightgray,hatchsep=0.8pt](1,0)(1,1)(2,1)(2,0)(1,0)
\pspolygon[fillstyle=hlines,hatchcolor=lightgray,hatchsep=0.8pt](2,1)(2,2)(3,2)(3,1)(2,1)
\end{pspicture}
\end{center}
If $\beta'$ is decreasing, then  $\pi=i\pi'n(n-1)\cdots(i+1)\beta''$ and $\pi'i\beta'' \in S_i(T)$, giving a
contribution of $\frac{x}{1-x}(F_T(x)-1)$.

If $\beta'$ is not decreasing, then $\pi'>\beta''$ (or an ascent $ab$ in $\beta'$ would be the 34 of a 1342); $\pi'$ avoids 231 (or $n$ is the 4 of a 2314); $\beta'$ avoids 231 (or $i$ is the 1 of a 1342), and $\beta''$ avoids $T$. Since $\beta'$ is not decreasing, its contribution is $C(x)-\frac{1}{1-x}$,
and the overall contribution of this case is $x^2C(x)\big(C(x)-\frac{1}{1-x}\big)F_T(x)$. Thus,
$$G_2(x)=\frac{x}{1-x}(F_T(x)-1)+x^2C(x)\left(C(x)-\frac{1}{1-x}\right)F_T(x)\:  .$$

Now, let $m\geq3$ and suppose $\pi=i_1\pi^{(1)}i_2\pi^{(2)}\cdots i_m\pi^{(m)}$ is a permutation that avoids $T$ with $m$ left-right maxima. Let $\alpha$ (resp. $\beta$) denote the list of letters in $\pi^{(m)}$ that are greater than (resp. less than) $i_1$. All letters of $\alpha$ occur before all letters of $\beta$ in $\pi^{(m)}$ (or $i_{1}i_{m-1}$ are the 23 of a 2314) and so $\pi^{(m)}=\alpha\beta$; $\pi^{(1)}>\beta$ (or $a \in \pi^{(1)},\, b\in \beta$ with $a<b$ makes $ai_{2}i_{m}b$ a 1342); $\pi^{(j)}>i_{j-1}$ for $j=2,\dots,m-1$ (or $i_{j-1}i_{j}i_m$ are the 234 of a 2314); $\alpha>i_{m-1}$ (or $i_1 i_{m-1}i_m$ are the 134 of a 1342). Thus, $\pi$ has the form pictured.
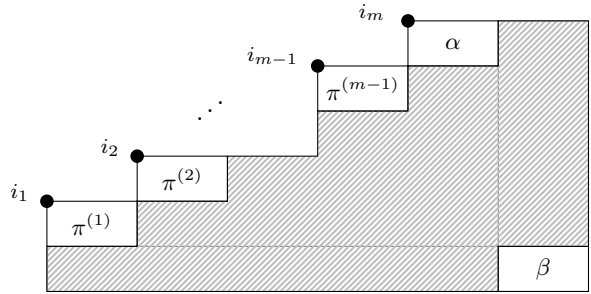
\begin{center}
\begin{pspicture}(0,0)(7,3.8)
\psset{xunit =1.2cm, yunit=.6cm, linewidth=.5\pslinewidth}
\psline[linecolor=gray](1,1)(5,1)(5,5)
\psline[linecolor=gray](0,1)(0,0)(6,0)(6,6)(5,6)
\rput(.5,1.5){\textrm{{\small $\pi^{(1)}$}}}
\rput(1.5,2.5){\textrm{{\small $\pi^{(2)}$}}}
\rput(3.5,4.5){\textrm{{\small $\pi^{(m-1)}$}}}
\rput(1.8,4.2){$\iddots$}
\psline(5,1)(6,1)(6,0)(5,0)(5,1)
\psline(0,1)(1,1)(1,3)(2,3)(2,2)(0,2)(0,1)
\psline(3,4)(4,4)(4,6)(5,6)(5,5)(3,5)(3,4)
\rput(4.5,5.5){\textrm{{\small $\alpha$}}}
\rput(5.5,0.5){\textrm{{\small $\beta$}}}
\pscircle*(0,2){.09}\pscircle*(1,3){.09}\pscircle*(3,5){.09}\pscircle*(4,6){.09}
\rput(-0.3,2.2){\textrm{{\footnotesize $i_1$}}}
\rput(.7,3.2){\textrm{{\footnotesize $i_{2}$}}}
\rput(2.5,5.2){\textrm{{\footnotesize $i_{m-1}$}}}
\rput(3.6,6.2){\textrm{{\footnotesize $i_{m}$}}}
\pspolygon[fillstyle=hlines,hatchcolor=lightgray,hatchsep=0.8pt](0,0)(0,1)(1,1)(1,2)(2,2)(2,3)(3,3)(3,4)(4,4)(4,5)(5,5)(5,6)(6,6)(6,1)(5,1)(5,0)(0,0)
\end{pspicture}
\end{center}
Also, $\pi_j$ avoids 231, $j=1,2,\dots, m-1$ (or $i_m$ is the 4 of a 2314);
$\alpha$ avoids 231 (or $i_{m-1}$ is the 1 of a 1342); $\beta$ avoids $T$.
Hence,
$$G_m(x)=x^mC^m(x)F_T(x)\, .$$

From $F_T(x)=\sum_{m\geq0}G_m(x)$, we obtain
$$F_T(x)=1+xF_T(x)+\frac{x}{1-x}(F_T(x)-1)+x^2C(x)\left(C(x)-\frac{1}{1-x}\right)F_T(x)+\sum_{m\geq3}x^mC^m(x)F_T(x)\,.$$
Solving for $F_T(x)$ yields
$$F_T(x)=\frac{(1-2x)(1-xC(x))}{(1-2x)(1-x)-x(1-2x)(1-x)C(x)-x^2(1-2x)C^2(x)-x^4C^3(x)}\, ,$$
which is equivalent to the desired expression.
\end{proof}

\subsubsection{$\mathbf{T=\{1324,1423,3142\}}$}
\begin{theorem}\label{th218A2}
Let $T=\{1324,1423,3142\}$. Then
$$F_T(x)=\frac{(1-2x)(1+\sqrt{1-4x})}{x^2+(2-4x+x^2)\sqrt{1-4x}}\, .$$
\end{theorem}
\begin{proof}
Let $G_m(x)$ be the generating function for $T$-avoiders with $m$
left-right maxima. Clearly, $G_0(x)=1$ and $G_1(x)=xF_T(x)$. Now
suppose $\pi=i_1\pi^{(1)}i_2\pi^{(2)}\cdots i_m\pi^{(m)}\in S_n(T)$ has $m\geq2$ left-right maxima. Then $i_1>\pi^{(j)}$
for all $j=1,2,\ldots,m-1$ to avoid 1324, and the letters $>i_1$ in $\pi^{(m)}$ are decreasing to avoid 1423.
We consider two cases for $\pi^{(m)}$:
\begin{itemize}
\item Each letter of $\pi^{(m)}$ is either greater than $i_{m-1}$
or smaller than $i_1$. In this case,
$\pi^{(m)}=\beta^{(1)}(n-1)\cdots\beta^{(n-1-i_{m-1})}(i_{m-1}+1)\beta^{(n-i_{m-1})}$,
where $\pi^{(1)}>\cdots>\pi^{(m-1)}>\beta^{(1)}>\cdots>\beta^{(n-i_{m-1})}$
and $\pi^{(j)}$ avoids $132$ for $j=1,2,\ldots,m-1$, $\beta^{(j)}$ avoids
$132$ for $j=1,2,\ldots,n-1-i_{m-1}$ and $\beta^{(n-i_{m-1})}$ avoids $T$. There are zero or more
factors of the form $\beta_j(n-j)$, each contributing $xC(x)$.
Hence, the contribution is
\[
\hspace*{15mm}\frac{x^mC(x)^{m-1}F_T(x)}{1-xC(x)}=x^mC(x)^mF_T(x)\, .
\]
\item $\pi^{(m)}$ has a letter between $i_1$ and $i_{m-1}$ (this case only arises for $m\ge 3$).
Let $s\in [m-2\,]$ be the smallest index such that $\pi^{(m)}$ has a letter between $i_s$ and $i_{s+1}$.
Then $\pi^{(s+1)}=\cdots=\pi^{(m-1)}=\emptyset$ to avoid 3142,
and $\pi$ has the form
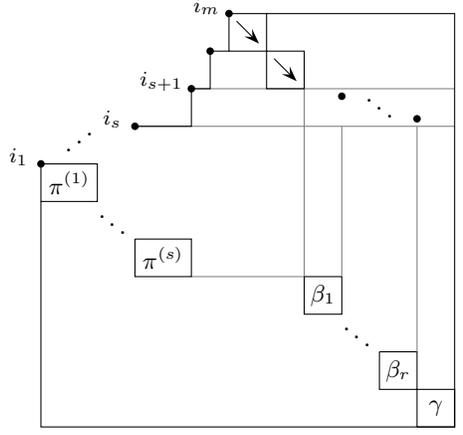
\begin{center}
\begin{pspicture}(-4,-.5)(13,6)
\psset{unit =.5cm, linewidth=.5\pslinewidth}
\psline[linecolor=gray](4,9)(11,9)
\psline[linecolor=gray](2.5,8)(11,8)
\psline[linecolor=gray](4,4)(7,4)
\psline[linecolor=gray](7,9)(7,4)
\psline(2.5,8)(4,8)(4,9)(4.5,9)(4.5,10)(5,10)
\psline(0,7)(1.5,7)(1.5,6)(0,6)(0,7)
\psline(2.5,5)(4,5)(4,4)(2.5,4)(2.5,5)
\psline(10,0)(10,1)(11,1)(11,0)(10,0)
\psline(5,11)(6,11)(6,10)(5,10)(5,11)
\psline(6,10)(7,10)(7,9)(6,9)(6,10)
\psline(7,4)(8,4)(8,3)(7,3)(7,4)
\psline(9,2)(10,2)(10,1)(9,1)(9,2)
\rput(.75,6.5){\textrm{{\small $\pi^{(1)}$}}}
\rput(3.25,4.5){\textrm{{\small $\pi^{(s)}$}}}
\rput(10.5,.5){\textrm{{\small $\gamma$}}}
\rput(7.5,3.5){\textrm{{\small $\beta_1$}}}
\rput(9.5,1.5){\textrm{{\small $\beta_r$}}}
\psline(6,11)(11,11)(11,0)(0,0)(0,6)
\pscircle*(0,7){.1}\pscircle*(2.5,8){.1}\pscircle*(4,9){.1}\pscircle*(4.5,10){.1}
\pscircle*(5,11){.1}\pscircle*(8,8.8){.1}\pscircle*(10,8.2){.1}
\rput(1,7.8){$\iddots$}
\rput(1.9,5.6){$\ddots$}
\rput(8.4,2.6){$\ddots$}
\rput(9,8.7){$\ddots$}
\psline[linecolor=gray](8,4)(8,8)
\psline[linecolor=gray](10,2)(10,8)
\psline[arrows=->,arrowsize=3pt 3](5.2,10.8)(5.8,10.2)
\psline[arrows=->,arrowsize=3pt 3](6.2,9.8)(6.8,9.2)
\rput(-0.6,7.2){\textrm{{\footnotesize $i_1$}}}
\rput(1.9,8.2){\textrm{{\footnotesize $i_{s}$}}}
\rput(3.2,9.2){\textrm{{\footnotesize $i_{s+1}$}}}
\rput(4.4,11.2){\textrm{{\footnotesize $i_{m}$}}}
\rput(11.3,0){,}
\end{pspicture}
\end{center}
where blank regions are empty and there is one $\beta$ for each of the $r:=i_{s+1}-i_{s}-1$ letters in $[i_{s}+1,i_{s+1}-1]$,
the $\pi$'s and $\beta$'s all avoid 132 (due to 1324), $\gamma$ avoids $T$, and the arrows indicate decreasing entries.
The $\pi$'s contribute $C(x)^s$; each $\beta$ and its associated letter between
$i_s$ and $i_{s+1}$ contributes $xC(x)$ and
there are one or more $\beta$'s, so they contribute $\frac{xC(x)}{1-xC(x)}$; each of the $m-1-s$ arrows contributes $\frac{1}{1-x}$; $\gamma$ contributes $F_T(x)$.
Thus, for given $s\in [m-2]$, the contribution is
$$\hspace*{10mm}\frac{x^mC(x)^sF_T(x)}{(1-x)^{m-1-s}}\frac{xC(x)}{1-xC(x)}=
\frac{x^{m+1}C(x)^{s+2}F_T(x)}{(1-x)^{m-1-s}}.$$
\end{itemize}
Hence, from $F_T(x)=\sum_{m\geq0}G_m(x)$, we have
$$F_T(x)=1+xF_T(x)+\sum_{m\geq2}\left(x^mC(x)^m F_T(x)+x^{m+1}C(x)^2 F_T(x)\sum_{s=1}^{m-2}\frac{C(x)^s}{(1-x)^{m-1-s}}\right),$$
with solution
$$F_T(x)=\frac{(1-2x)(1-xC(x))}{(1-2x)(1-x)-x(1-2x)(1-x)C(x)-x^2(1-2x)C(x)^2-x^4C(x)^3},$$
which simplifies to the desired expression.
\end{proof}

\subsubsection{$\mathbf{T=\{1243,1342,2314\}}$}
We will employ an approach similar to
that used for the second case in class 203 above and make use of
the same notation.  As before, we have
\begin{equation}\label{218ce1}
a(n;i)=a(n-1;i-1)+b(n;i)+c(n;i)+d(n;i), \qquad 1 \leq i \leq n-1,
\end{equation}
with $a(n;n)=1$ for $n\geq 1$.  The arrays $b(n;i)$, $c(n;i)$ and
$d(n;i)$ are determined recursively as follows and a similar proof
applies.

\begin{lemma}\label{218cl1}
We have
\begin{align}
b(n;i)&=\sum_{j=i}^{n-1}a(n-1;j),\qquad1\leq i\leq n-1,\label{218cl1e1}\\
c(n;i)&=a(n-i-2)+\sum_{j=1}^{n-i-2}2^{n-i-j-2}a(j-1), \qquad 1 \leq
i \leq n-2,\label{218cl1e2} \\
d(n;i)&=c(n-1;i)+c(n-1;i-1)+d(n-1;i)+d(n-1;i-1), \qquad 1 \leq i
\leq n-2. \label{218cl1e3}
\end{align}
\end{lemma}

Note that the recurrences in Lemma \ref{218cl1} are the same as
those in Lemma \ref{203bl1} except for a factor of $2^{n-i-j-2}$
appearing in the formula for $c(n;i)$.  This accounts for the fact
that within the decomposition of a $T$-avoiding permutation
$\pi=\alpha j(n-1) \beta n \gamma$ enumerated by $c(n;i)$, where
$\alpha=j+i-1,j+i-2,\ldots,j+1$ for some $i$, the section $\beta$
is now a permutation of $[j+i,n-2]$ that avoids the patterns $132$
and $231$ (instead of just being a decreasing sequence as it was
previously).  Thus, there are $2^{n-i-j-2}$ possibilities for
$\beta$ whenever it is nonempty.  Note that a comparison of the
recurrences shows that there are strictly more permutations of
length $n$ that avoid $\{1243,1342,2314\}$ than there are that
avoid $\{1234,1342,2314\}$ for $n\geq 5$.

If $a(x;u)=\sum_{n\geq0}a_n(u)x^n$ as before, then one gets the
following functional equation whose proof we omit.

\begin{lemma}\label{218cl2}
We have
\begin{equation}\label{218cl2e1}
\left(1+\frac{xu^2}{1-u}\right)a(x;u)=1+xu\left(\frac{1}{1-u}+\frac{x^2(1-x)}{(1-2x)(1-xu)(1-x-xu)}\right)a(x;1).
\end{equation}
\end{lemma}

We can now determine the generating function $F_T(x)$.

\begin{theorem}\label{th218A3}
Let $T=\{1243,1342,2314\}$.  Then
$$F_T(x)=\frac{(1-2x)(1+\sqrt{1-4x})}{x^2+(2-4x+x^2)\sqrt{1-4x}}.$$
\end{theorem}
\begin{proof}
Setting $u=C(x)$ in \eqref{218cl2e1}, and using the fact
$xu^2=u-1$, gives
\begin{align*}
a(x;1)&=-\frac{(1-2x)(1-u)(1-xu)(1-x-xu)}{xu(1-2x)(1-xu)(1-x-xu)+x^3u(1-x)(1-u)}\\
&=\frac{x(1-2x)(1-xu)}{x(1-2x)(1-x-xu)+x^2(1-x)(1-u+xu)}\\
&=\frac{(1-2x)(1+\sqrt{1-4x})}{(1-2x)(1-2x+\sqrt{1-4x})+(1-x)(3x-1+(1-x)\sqrt{1-4x})}\\
&=\frac{(1-2x)(1+\sqrt{1-4x})}{x^2+(2-4x+x^2)\sqrt{1-4x}},
\end{align*}
as desired.
\end{proof}

\subsection{Case 229}
The three representative triples $T$ are:

\{2341,2413,3142\} (Theorem \ref{th229A1})

\{1342,1423,2143\} (Theorem \ref{th229A2})

\{1342,1423,2134\} (Theorem \ref{th229A3})

\subsubsection{$\mathbf{T=\{2341,2413,3142\}}$}

\begin{theorem}\label{th229A1}
Let $T=\{2341,2413,3142\}$. Then
$$F_T(x)=\frac{1-2x+2x^2-\sqrt{1-8x+20x^2-24x^3+16x^4-4x^5}}{2x(1-x+x^2)}.$$
\end{theorem}
\begin{proof}
Let $G_m(x)$ be the generating function for $T$-avoiders with $m$ left-right maxima.
Clearly, $G_0(x)=1$ and $G_1(x)=xF_T(x)$.
Now suppose $m\geq2$ and $\pi=i_1\pi^{(1)}\cdots i_m\pi^{(m)}$ avoids $T$.
Clearly, there is no letter smaller than $i_1$ in $\pi^{(3)}\cdots\pi^{(m)}$ (such a letter would be the ``1'' of a 2341). Moreover, to avoid 2413 and 3142, $\pi^{(1)}i_2\pi^{(2)}$ has the form
$\beta'i_2\beta''\beta'''$ with $\beta''>i_1>\beta'>\beta'''$:
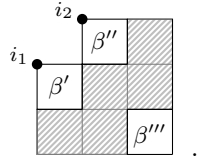
\begin{center}
\begin{pspicture}(-1,.1)(3,2.1)
\psset{xunit =.6cm, yunit=.6cm,linewidth=.5\pslinewidth}
\pspolygon[fillstyle=hlines,hatchcolor=lightgray,hatchsep=0.8pt](0,1)(0,0)(2,0)(2,1)(3,1)(3,3)(2,3)(2,2)(1,2)(1,1)(0,1)
\psline[linecolor=gray](0,1)(0,0)(1,0)(1,1)(1,1)(2,1)(2,2)(3,2)
\psline[linecolor=gray](1,0)(2,0)
\psline[linecolor=gray](2,3)(3,3)
\psline(0,2)(0,1)(1,1)(1,3)(2,3)(2,2)(0,2)
\psline(3,0)(3,3)
\psline(3,0)(2,0)(2,1)(3,1)
\rput(.5,1.5){\textrm{{\small $\beta'$}}}
\rput(1.5,2.5){\textrm{{\small $\beta''$}}}
\rput(2.5,0.4){\textrm{{\small $\beta'''$}}}
\pscircle*(0,2){.07}\pscircle*(1,3){.07}
\rput(-0.4,2.2){\textrm{{\footnotesize $i_1$}}}
\rput(0.6,3.2){\textrm{{\footnotesize $i_{2}$}}}
\rput(3.5,0){.}
\end{pspicture}
\end{center}
If $\beta'''=\emptyset$, then we have a contribution of $xF_T(x)G_{m-1}(x)$.
Otherwise, $\pi$ has the form
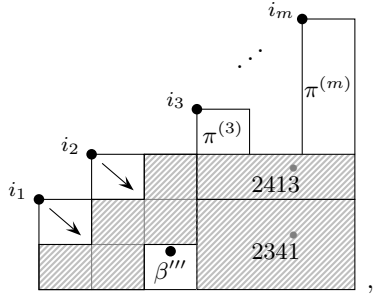
\begin{center}
\begin{pspicture}(-2,-.1)(6,4)
\psset{xunit =.7cm, yunit=.6cm,linewidth=.5\pslinewidth}
\pspolygon[fillstyle=hlines,hatchcolor=lightgray,hatchsep=0.8pt](0,1)(0,0)(2,0)(2,1)(3,1)(3,3)(2,3)(2,2)(1,2)(1,1)(0,1)
\psline[linecolor=gray](0,1)(0,0)(1,0)(1,1)(1,1)(2,1)(2,2)(3,2)
\psline[linecolor=gray](1,0)(2,0)
\psline(2,3)(3,3)
\psline(0,2)(0,1)(1,1)(1,3)(2,3)(2,2)(0,2)
\psline(3,3)(3,4)(4,4)(4,3)
\psline(5,3)(5,6)(6,6)(6,3)
\psline(3,0)(2,0)(2,1)(3,1)
\psline(0,1)(0,0)(2,0)
\psline[arrows=->,arrowsize=3pt 3](.2,1.8)(.8,1.2)
\psline[arrows=->,arrowsize=3pt 3](1.2,2.8)(1.8,2.2)
\rput(2.5,0.4){\textrm{{\small $\beta'''$}}}
\rput(3.5,3.5){\textrm{{\small $\pi^{(3)}$}}}
\rput(5.5,4.5){\textrm{{\small $\pi^{(m)}$}}}
\psline[fillstyle=hlines,hatchcolor=lightgray,hatchsep=0.8pt](3,0)(3,2)(6,2)(6,0)(3,0)
\pspolygon[fillstyle=hlines,hatchcolor=lightgray,hatchsep=0.8pt](3,2)(3,3)(6,3)(6,2)(3,2)
\pscircle*(0,2){.07}\pscircle*(1,3){.07}\pscircle*(3,4){.07}\pscircle*(5,6){.07}\pscircle*(2.5,.85){.07}
\rput(4,5.3){$\iddots$}
\rput(4.5,1){\textrm{{\small $234\overset{{\gray \bullet}}{1}$}}}
\rput(4.5,2.5){\textrm{{\small $241\overset{{\gray \bullet}}{3}$}}}
\rput(-0.4,2.2){\textrm{{\footnotesize $i_1$}}}
\rput(0.6,3.2){\textrm{{\footnotesize $i_{2}$}}}
\rput(2.6,4.2){\textrm{{\footnotesize $i_{3}$}}}
\rput(4.6,6.2){\textrm{{\footnotesize $i_{m}$}}}
\rput(6.3,0){,}
\end{pspicture}
\end{center}
where dark bullets indicate mandatory entries, shaded regions are empty (gray bullets would form part of a forbidden pattern as indicated),
$\beta'$ is decreasing ($b<c$ in $\beta'$ implies $bci_2a$  is a 2341 for $a$ in $\beta'''$), and $\beta''$ is decreasing ($b<c$ in $\beta''$ implies $i_1bca$ is a 2341).

Thus, we have a contribution of $\frac{x^2}{(1-x)^2}(F_T(x)-1)G_{m-2}(x)$.
Hence, for $m\geq2$,
$$G_m(x)=xF_T(x)G_{m-1}(x)+\frac{x^2}{(1-x)^2}(F_T(x)-1)G_{m-2}(x)\, .$$
By summing over $m\geq2$, we obtain
$$F_T(x)-1-xF_T(x)=xF_T(x)(F_T(x)-1)+\frac{x^2}{(1-x)^2}(F_T(x)-1)F_T(x)\, .$$
Solving this quadratic for $F_T(x)$ completes the proof.
\end{proof}

\subsubsection{$\mathbf{T=\{1342,1423,2143\}}$}

Here, and in the subsequent subsection, let $a(n;i_1,i_2,\ldots,i_k)$ denote the number of $T$-avoiding permutations of length $n$ starting with $i_1,i_2,\ldots,i_k$. Let $a(n)=\sum_{i=1}^na(n;i)$ for $n\geq1$ be the total number of $T$-avoiders, with $a(0)=1$, and $\mathcal{T}_{i,j}$ be the set of permutations enumerated by $a(n;i,j)$. Clearly, $a(n;n)=a(n;n-1)=a(n-1)$ for all $n\geq 2$. We have the following recurrence for the array $a(n;i,j)$.

\begin{lemma}\label{l1229b}
If $n\geq3$, then
\begin{equation}\label{l1229be1}
a(n;i,j)=a(n-j+i+1;i+1,i)+\sum_{\ell=1}^{i-1}a(n-j+i+1;i,\ell), \qquad i+2 \leq j \leq n,
\end{equation}
\begin{equation}\label{l1229be2}
a(n;i,i-1)=a(n-1;i;i-1)+\sum_{\ell=1}^{i-2}a(n-1;i-1,\ell), \qquad 2 \leq i \leq n-1,
\end{equation}
and
\begin{equation}\label{l1229be3}
a(n;i,j)=a(n-1;i-1,j)+\sum_{r=2}^{i-j}a(n-r;j+1,j)+\sum_{r=1}^{i-j}\sum_{\ell=1}^{j-1}a(n-r;j,\ell)
\end{equation}
for $3 \leq i \leq n-1$ and $1 \leq j \leq i-2$, with $a(n;i,i+1)=a(n-1;i)$ for $1 \leq i \leq n-1$.
\end{lemma}
\begin{proof}
Let $x$ denote the third letter of a member of $\mathcal{T}_{i,j}$.  Clearly, we have $|\mathcal{T}_{i,i+1}|=a(n-1;i)$, as the letter $i+1$ may be deleted.  To show \eqref{l1229be2}, first note members of $\mathcal{T}_{i,i-1}$ must have $x=i+1$ or $x<i-1$. In the first case, the letter $i+1$ may be deleted, implying $a(n-1;i,i-1)$ possibilities, while in the latter, the letter $i$ may be, which gives $\sum_{\ell=1}^{i-2}a(n-1;i-1,\ell)$ possibilities.  We now show \eqref{l1229be1}.  Note first that one cannot have $x>j$ or $x<i$ within members of $\mathcal{T}_{i,j}$ if $j\geq i+3$, lest there be an occurrence of $1342$ or $1423$ (as witnessed by $ijx(j-1)$ or $ij(j-2)(j-1)$, respectively).  So we must have $x \in [i+1,j-1]$ and thus $x=j-1$ in order to avoid $1423$.  By similar reasoning, the fourth letter must be $x-1$ if $x \geq i+3$.  Repeating this argument shows that the block of letters $j,j-1,\ldots,i+2$ must occur.  The next letter $z$ must be $i+1$ or less than $i$ (so as to avoid $1342$).  If $z=i+1$, then all members of $[i+3,j]$, along with $i$, are seen to be irrelevant concerning avoidance of $T$ and hence may be deleted, while if $z<i$, then all members of $[i+2,j]$ may be deleted (note that $i,z$ imposes the same requirement on subsequent letters as does $i,i+2$ and $i+2,z$, together).  It follows that there are $a(n-j+i+1;i+1,i)+\sum_{\ell=1}^{i-1}a(n-j+i+1;i,\ell)$ members of $\mathcal{T}_{i,j}$ when $j \geq i+2$.

For \eqref{l1229be3}, we consider the following cases for $x$:  (i) $x=j+1$, (ii) $x<j$, (iii) $j+1<x<i$, and (iv) $x=i+1$.  There are clearly $a(n-1;i-1,j)$ possibilities in (i) and $\sum_{\ell=1}^{j-1}a(n-1;j,\ell)$ possibilities in (ii).  Reasoning as in the previous paragraph shows in case (iii) that the block of letters $x,x-1,\ldots,j+2$ must occur directly following $j$.  The next letter $z$ may either equal $j+1$ or be less than $j$.  Thus, all members of $[j+3,x]$, along with $i$, may be deleted in either case.  Furthermore, the letter $j$ may also be deleted if $z=j+1$ (since $j+2,j+1$ is more restrictive than $i,j$), while the letter $j+2$ may be deleted if $z<j$ (since $j+2$ is redundant in light of $j,z$).  Considering all possible $x$, and letting $r=x-j$, one gets $\sum_{r=2}^{i-j-1}a(n-r;j+1,j)$ possibilities if $z=j+1$, and $\sum_{r=2}^{i-j-1}\sum_{\ell=1}^{j-1}a(n-r;j,\ell)$ possibilities if $z<j$.  If $x=i+1$, then the block $x,x-2,x-3,\ldots,j+2$ must occur with the next letter $z$ as in case (iii) above.  This implies that there are $a(n-i+j;j+1,j)+\sum_{\ell=1}^{j-1}a(n-i+j;j,\ell)$ possibilities in (iv).  Combining all of the previous cases gives \eqref{l1229be3} and completes the proof.
\end{proof}

In order to solve the recurrence in Lemma \ref{l1229b}, we introduce the following auxiliary functions: $b_{n,i}(v)=\sum_{j=1}^{i-1}a(n;i,j)v^j$ for $2 \leq i \leq n-1$, $c_{n,i}(v)=\sum_{j=i+1}^{n}a(n;i,j)v^j$ for $1 \leq i \leq n-1$, $b_n(u,v)=\sum_{i=2}^{n-2}b_{n,i}(v)u^i$ for $n \geq 4$, $c_n(u,v)=\sum_{i=1}^{n-2}c_{n,i}(v)u^i$ for $n \geq3$, and $d_n(u)=\sum_{i=2}^{n-1}a(n;i,i-1)u^i$ for $n\geq3$.  Let $a_n(u,v)=\sum_{i=1}^n\sum_{j=1,j\neq i}^na(n;i,j)u^iv^j$ for $n\geq2$, with $a_1(u,v)=u$.  Note that by the definitions, we have
\begin{equation}\label{229be1}
a_n(u,v)=u^{n-1}(1+u)a_{n-1}(v,1)-(uv)^{n-1}(1-v)a_{n-2}(1,1)+b_n(u,v)+c_n(u,v), \qquad n \geq 2.
\end{equation}

By \eqref{l1229be2} and \eqref{l1229be3}, we have for $2 \leq i \leq n-2$,
\begin{align*}
b_{n,i}(v)&=b_{n-1,i-1}(v)+a(n-1;i,i-1)v^{i-1}+\sum_{j=1}^{i-1}b_{n-1,j}(1)v^j+\sum_{j=1}^{i-2}v^j\sum_{r=2}^{i-j}a(n-r;j+1,j)\notag\\
&\quad+\sum_{j=1}^{i-2}v^j\sum_{r=2}^{i-j}b_{n-r,j}(1).
\end{align*}
Multiplying both sides of the last recurrence by $u^i$, and summing over $2 \leq i\leq n-2$, yields
\begin{align}
b_n(u,v)&=ub_{n-1}(u,v)+\frac{1}{v}d_{n-1}(uv)+\sum_{j=1}^{n-3}b_{n-1,j}(1)\left(\frac{u^{j+1}-u^{n-1}}{1-u}\right)v^j\notag\\
&\quad+\sum_{j=1}^{n-4}v^j\sum_{r=2}^{n-j-1}a(n-r;j+1,j)\left(\frac{u^{j+r}-u^{n-1}}{1-u}\right)+\sum_{j=1}^{n-4}v^j\sum_{r=2}^{n-j-2}b_{n-r,j}(1)\left(\frac{u^{j+r}-u^{n-1}}{1-u}\right)\notag\\
&=ub_{n-1}(u,v)+\frac{1}{v}d_{n-1}(uv)+\frac{u}{1-u}b_{n-1}(uv,1)-\frac{u^{n-1}}{1-u}b_{n-1}(v,1)\notag\\
&\quad+\frac{1}{uv(1-u)}\sum_{r=2}^{n-2}(d_{n-r}(uv)u^r+a(n-r-2)u^nv^{n-r})\notag\\
&\quad-\frac{u^{n-1}}{v(1-u)}\sum_{r=2}^{n-2}(d_{n-r}(v)+a(n-r-2)v^{n-r})+\frac{1}{1-u}\sum_{r=2}^{n-3}b_{n-r}(uv,1)u^r\notag\\
&\quad-\frac{u^{n-1}}{1-u}\sum_{r=2}^{n-3}b_{n-r}(v,1)\notag\\
&=ub_{n-1}(u,v)+\frac{1}{v}d_{n-1}(uv)+\frac{1}{1-u}\sum_{r=3}^{n-1}b_r(u,v)u^{n-r}-\frac{u^{n-1}}{1-u}\sum_{r=3}^{n-1}b_r(v,1)\notag\\
&\quad+\frac{1}{uv(1-u)}\sum_{r=2}^{n-2}d_r(uv)u^{n-r}-\frac{u^{n-1}}{v(1-u)}\sum_{r=2}^{n-2}d_r(v), \qquad n \geq 4. \label{229be2}
\end{align}

By \eqref{l1229be1}, we have
$$c_{n,i}(v)=a(n-1;i)v^{i+1}+\sum_{j=2}^{n-i}a(n-j+1;i+1,i)v^{i+j}+\sum_{j=2}^{n-i}b_{n-j+1,i}(1)v^{i+j}, \qquad 1 \leq i \leq n-2,$$
and thus
\begin{align}
c_n(u,v)&=v\sum_{i=1}^{n-2}a(n-1;i)(uv)^i+\sum_{j=2}^{n-1}v^j\sum_{i=1}^{n-j}a(n-j+1;i+1,i)(uv)^i+\sum_{j=2}^{n-1}v^j\sum_{i=1}^{n-j}b_{n-j+1,i}(1)(uv)^i\notag\\
&=v(a_{n-1}(uv,1)-a(n-2)(uv)^{n-1})+\frac{1}{u}\sum_{j=2}^{n-1}v^{j-1}(d_{n-j+1}(uv)+a(n-j-1)(uv)^{n-j+1})\notag\\
&\quad+\sum_{j=2}^{n-1}v^j(b_{n-j+1}(uv,1)+(a(n-j)-a(n-j-1))(uv)^{n-j})\notag\\
&=v(a_{n-1}(uv,1)-a(n-2)(uv)^{n-1})+\frac{1}{u}\sum_{j=2}^{n-1}d_j(uv)v^{n-j}+\sum_{j=2}^{n-1}b_j(uv,1)v^{n-j+1}\notag\\
&\quad+v^n\sum_{j=1}^{n-2}a(j)u^j, \qquad n \geq 3.  \label{229be3}
\end{align}
Multiplying both sides of \eqref{l1229be2} by $u^i$, and summing over $2 \leq i \leq n-1$, gives
\begin{equation}\label{229be4}
d_n(u)=u^{n-1}a(n-2)+ub_{n-1}(u)+d_{n-1}(u), \qquad n \geq3.
\end{equation}

Define generating functions $a(x;u,v)=\sum_{n\geq1}a_n(u,v)x^n$, $b(x;u,v)=\sum_{n\geq4}b_n(u,v)x^n$, $c(x;u,v)=\sum_{n\geq3}c_n(u,v)x^n$, and $d(x;u)=\sum_{n\geq3}d_n(u)x^n$.  Rewriting recurrences \eqref{229be1}--\eqref{229be4} in terms of generating functions yields the following system of functional equations.

\begin{lemma}\label{l2229b}
We have
\begin{align}
a(x;u,v)&=xu(1-xv+xv^2)+b(x;u,v)+c(x;u,v)+x(1+u)a(xu;v,1)\notag\\
&\quad-x^2uv(1-v)a(xuv;1,1),\label{b229e1}
\end{align}
\begin{align}
(1-xu)b(x;u,v)&=\frac{x}{(1-u)(1-xu)}(ub(x;uv,1)-b(xu;v,1))+\frac{x(1-u+xu^2)}{v(1-u)(1-xu)}d(x;uv)\notag\\
&\quad-\frac{x^2u}{v(1-u)(1-xu)}d(xu;v), \label{b229e2}
\end{align}
\begin{align}
c(x;u,v)&=-x^2uv^2+xva(x;uv,1)+\left(\frac{x^2v^2}{1-xv}-x^2uv^2\right)a(xuv;1,1)+\frac{xv^2}{1-xv}b(x;uv,1)\notag\\
&\quad+\frac{xv}{u(1-xv)}d(x;uv),\label{b229e3}
\end{align}
\begin{equation}\label{b229e4}
(1-x)d(x;u)=x^2ua(xu;1,1)+xub(x;u,1).
\end{equation}
\end{lemma}

We can now determine the generating function $F_T(x)$.

\begin{theorem}\label{th229A2}
Let $T=\{1342,1423,2143\}$.  Then
$$F_T(x)=\frac{1-2x+2x^2-\sqrt{1-8x+20x^2-24x^3+16x^4-4x^5}}{2x(1-x+x^2)}.$$
\end{theorem}
\begin{proof}
In the notation above, we seek to determine $1+a(x;1,1)$.  Setting $u=v=1$ in \eqref{b229e1}, \eqref{b229e3} and \eqref{b229e4}, and solving the resulting system for $b(x;1,1)$, $c(x;1,1)$ and $d(x;1)$, yields
\begin{align*}
b(x;1,1)&=\frac{(1-5x+7x^2-5x^3+x^4)a(x;1,1)-x(1-x)^3}{1-x+x^2},\\
c(x;1,1)&=\frac{x(2-2x+x^2)((1-x)a(x;1,1)-x)}{1-x+x^2},\\
d(x;1)&=\frac{x(1-x)^2((1-x)a(x;1,1)-x)}{1-x+x^2}.
\end{align*}
Substituting the expression for $d(x;u)$ from \eqref{b229e4} into \eqref{b229e2} at $v=1$, we find
\begin{align*}
&\left(1-x-\frac{x}{(1-u)(1-x)}-\frac{x^2(1-u+xu)}{(1-u)(1-x)(u-x)}\right)b(x/u;u,1)\\
&\quad=\frac{x^3(1-u+xu)}{u(1-u)(1-x)(u-x)}a(x;1,1)-\frac{x}{u(1-u)(1-x)}b(x;1,1)-\frac{x^2}{u(1-u)(1-x)}d(x;1).
\end{align*}
Applying the kernel method to the preceding equation, and setting
$$u=u_0=\frac{1-2x+\sqrt{1-8x+20x^2-24x^3+16x^4-4x^5}}{2(1-x)^2},$$
we obtain
\begin{align*}
b(x;1,1)=\frac{x^2(1-u_0+xu_0)}{u_0-x}a(x;1,1)-xd(x;1).
\end{align*}
Substituting out the expressions above for $b(x;1,1)$ and $d(x;1)$, and then solving the equation that results for $a(x;1,1)$, yields
$$a(x;1,1)=\frac{x(1-x)^2(u_0-x)}{(1-4x+4x^2-2x^3)u_0-x(1-x)^3}.$$
Substituting the expression for $u_0$ into the last equation gives the desired formula for $1+a(x;1,1)$ and completes the proof.
\end{proof}

\emph{Remark:}  Once $a(x;1,1)$ is known, it is possible to find $b(x;u,1)$, and thus $d(x;u)$, $a(x;u,1)$ and $c(x;u,1)$.  This in turn allows one to solve the system \eqref{b229e1}--\eqref{b229e4} for all $u$ and $v$, and thus obtain a generating function formula for the joint distribution of the statistics recording the first two letters.

\subsubsection{$\mathbf{T=\{1342,1423,2134\}}$}

Clearly, $a(n;n)=a(n-1)$ for all $n\geq 1$. We have the following recurrence for the array $a(n;i,j)$ where $i<n$.

\begin{lemma}\label{l1229c}
If $n \geq3$, then
\begin{equation}\label{l1229ce1}
a(n;i,n)=a(n-1;i,n-1)+\sum_{j=1}^{i-1}a(n-1;i,j), \qquad 1 \leq i \leq n-2,
\end{equation}
\begin{equation}\label{l1229ce2}
a(n;i,i-1)=a(n-1;i-1,n-1)+\sum_{j=1}^{i-2}a(n-1;i-1,j), \qquad 2 \leq i \leq n-1,
\end{equation}
and
\begin{align}
a(n;i,j)&=a(n-1,i-1,j)+a(n-1;j,n-1)+\sum_{\ell=1}^{j-1}a(n-1;j,\ell)\notag\\
&\quad+\sum_{\ell=2}^{i-j-1}a(n-\ell;j+1,j), \qquad 1 \leq j \leq i-2 \quad and \quad  3\leq i \leq n-1. \label{l1229ce3}
\end{align}
Furthermore, we have $a(n;i,i+1)=a(n-1;i)$ for $1 \leq i \leq n-1$ and $a(n;i,j)=a(n-j+i+1;i+1,i)$ for $i+2\leq j \leq n-1$.
\end{lemma}
\begin{proof}
Throughout, let $x$ denote the third letter of a member of $\mathcal{T}_{i,j}$.  To show \eqref{l1229ce1}, first note that for members of $\mathcal{T}_{i,n}$, we must have $x=n-1$ or $x<i$.   There are $a(n-1;i,n-1)$ possibilities in the first case as the letter $n$ is extraneous concerning avoidance of $T$, whence it may be deleted, and $\sum_{j=1}^{i-1}a(n-1;i,j)$ possibilities in the latter case as again $n$ may be deleted (note that the presence of $i,j$ imposes a stronger restriction on the order of subsequent letters than does $i,n$).  To show \eqref{l1229ce2}, first note that members of $\mathcal{T}_{i,i-1}$ for $2\leq i \leq n-1$ must have $x=n$ or $x<i-1$.  There are $a(n-1;i-1,n-1)$ possibilities in the former case and $\sum_{j=1}^{i-2}a(n-1;i-1,j)$ possibilities in the latter since the letter $i$ may be deleted in either case as the restriction it imposes is redundant.

Next, we show \eqref{l1229ce3}.  For this, we consider the following cases: (i) $x=j+1$, (ii) $x=n$, (iii) $x<j$, and (iv) $j+1<x<i$. The first three cases are readily seen to be enumerated by the first three terms, respectively, on the right-hand side of \eqref{l1229ce3}. For case (iv), let $y$ denote the fourth letter of $\pi \in \mathcal{T}_{i,j}$.  First note that one cannot have $y>x$, for otherwise $\pi$ would contain $1342$ as witnessed by the subsequence $jxy(x-1)$.  It is also not possible to have $y<j$, for otherwise $\pi$ would again contain $1342$, this time with the subsequence $jxn(j+1)$, since all letters to the right of $y$ and larger than $j$ would have to occur in decreasing order (due to the presence of $j,y$).  So we must have $j<y<x$ and thus $y=x-1$ in order to avoid $1423$.  By similar reasoning, the next letter must be $x-2$ if $x>j+2$.  Repeating this argument shows that the block $x,x-1,\ldots,j+1$ must occur directly following $j$, with each of these letters, except the last two, seen to be extraneous concerning the avoidance or containment of patterns in $T$.  Note further that the presence of $j+2,j+1$ imposes a stricter requirement on subsequent letters than does $i,j$ when $i \geq j+3$, whence the $i$ and $j$ are also extraneous.  Deleting all members of $[j+3,x]$ from $\pi$, along with $i$ and $j$, implies that there are $a(n-\ell;j+1,j)$ possibilities where $\ell=x-j$.  Summing over all possible values of $\ell$ gives the last term on the right-hand side of \eqref{l1229ce3}.

There are clearly $a(n-1;i)$ members of $\mathcal{T}_{i,j}$ if $j=i+1$, as the letter $i+1$ may be deleted.  If $j \geq i+2$, then similar reasoning as before shows that the block $j,j-1,\ldots,i+1$ must occur when $j <n$, and thus all members of $[i+3,j]$, along with $i$, may be deleted.  This implies that there are $a(n-j+i+1;i+1,i)$ members of $\mathcal{T}_{i,j}$ in this case, which completes the proof.
\end{proof}

In order to solve the recurrence in Lemma \ref{l1229c}, we introduce the following functions: $b_{n,i}(v)=\sum_{j=1}^{i-1}a(n;i,j)v^j$ for $2 \leq i \leq n-1$, $c_{n,i}(v)=\sum_{j=i+1}^{n-1}a(n;i,j)v^j$ for $1 \leq i \leq n-2$, $b_n(u,v)=\sum_{i=2}^{n-1}b_{n,i}(v)u^i$ for $n \geq 3$, $c_n(u,v)=\sum_{i=1}^{n-2}c_{n,i}(v)u^i$ for $n \geq3$, and $d_n(u)=\sum_{i=1}^{n-1}a(n;i,n)u^i$ for $n\geq2$.  Let $a_n(u,v)=\sum_{i=1}^n\sum_{j=1,j\neq i}^na(n;i,j)u^iv^j$ for $n\geq2$, with $a_1(u,v)=u$.  Note that by the definitions, we have
\begin{equation}\label{229ce1}
a_n(u,v)=u^na_{n-1}(v,1)+b_n(u,v)+c_n(u,v)+v^nd_n(u), \qquad n \geq 2.
\end{equation}

In order to determine a formula for $b_n(u,v)$, first note that \eqref{l1229ce2} and \eqref{l1229ce3} imply
\begin{align*}
b_{n,i}(v)&=b_{n-1,i-1}(v)+\sum_{j=1}^{i-1}a(n-1;j,n-1)v^j+\sum_{j=1}^{i-1}b_{n-1,j}(1)v^j\notag\\
&\quad+\sum_{j=1}^{i-3}v^j\sum_{\ell=2}^{i-j-1}a(n-\ell;j,n-\ell), \qquad 2 \leq i \leq n-1,
\end{align*}
where we have used the fact $a(m;j+1,j)=a(m;j,m)$ in the last sum.  Multiplying both sides of the last recurrence by $u^i$, and summing over $2 \leq i \leq n-1$, gives
\begin{align}
b_n(u,v)&=ub_{n-1}(u,v)+\sum_{j=1}^{n-2}a(n-1;j,n-1)\left(\frac{u^{j+1}-u^n}{1-u}\right)v^j+\sum_{j=1}^{n-2}b_{n-1,j}(1)\left(\frac{u^{j+1}-u^n}{1-u}\right)v^j\notag\\
&\quad+\sum_{j=1}^{n-4}v^j\sum_{\ell=2}^{n-j-1}a(n-\ell;j,n-\ell)\left(\frac{u^{j+\ell+1}-u^n}{1-u}\right)\notag\\
&=ub_{n-1}(u,v)+\frac{u}{1-u}(d_{n-1}(uv)-u^{n-1}d_{n-1}(v))+\frac{u}{1-u}(b_{n-1}(uv,1)-u^{n-1}b_{n-1}(v,1))\notag\\
&\quad+\frac{u}{1-u}\sum_{\ell=2}^{n-2}d_\ell(uv)u^{n-\ell}-\frac{u^n}{1-u}\sum_{\ell=2}^{n-2}d_\ell(v), \qquad n\geq 3, \label{229ce2}
\end{align}
where we have replaced the index $\ell$ by $n-\ell$ in the last sum.

By Lemma \ref{l1229c}, we have
$$c_{n,i}(v)=a(n-1;i)v^{i+1}+\sum_{j=i+2}^{n-1}a(n-j+i+1;i+1,i)v^j, \qquad 1 \leq i \leq n-2,$$
and thus
\begin{align}
c_n(u,v)&=\sum_{i=1}^{n-2}a(n-1;i)u^iv^{i+1}+\sum_{i=1}^{n-3}u^i\sum_{j=2}^{n-i-1}a(n-j+1;i+1,i)v^{i+j}\notag\\
&=v(a_{n-1}(uv,1)-(uv)^{n-1}a(n-2))+\sum_{j=2}^{n-2}v^j\sum_{i=1}^{n-j-1}a(n-j+1;i,n-j+1)(uv)^i\notag\\
&=v(a_{n-1}(uv,1)-(uv)^{n-1}a(n-2))+\sum_{j=2}^{n-2}v^j(d_{n-j+1}(uv)-(uv)^{n-j}a(n-j-1))\notag\\
&=v(a_{n-1}(uv,1)-(uv)^{n-1}a(n-2))+\sum_{j=3}^{n-1}v^{n-j+1}(d_j(uv)-(uv)^{j-1}a(j-2)), \qquad n \geq 3.\label{229ce3}
\end{align}
Multiplying both sides of \eqref{l1229ce1} by $u^i$, and summing over $1 \leq i \leq n-2$ implies
\begin{equation}\label{229ce4}
d_n(u)=u^{n-1}a(n-2)+b_{n-1}(u,1)+d_{n-1}(u), \qquad n \geq2.
\end{equation}

Define generating functions $a(x;u,v)=\sum_{n\geq1}a_n(u,v)x^n$, $b(x;u,v)=\sum_{n\geq3}b_n(u,v)x^n$, $c(x;u,v)=\sum_{n\geq3}c_n(u,v)x^n$, and $d(x;u)=\sum_{n\geq2}d_n(u)x^n$.  Rewriting recurrences \eqref{229ce1}--\eqref{229ce4} in terms of generating functions yields the following system of functional equations.

\begin{lemma}\label{l2229c}
We have
\begin{equation}\label{l2c229e1}
a(x;u,v)=xu+xua(xu;v,1)+b(x;u,v)+c(x;u,v)+d(xv;u),
\end{equation}
\begin{align}
(1-xu)b(x;u,v)&=\frac{xu}{1-u}(b(x;uv,1)-b(xu;v,1))+\frac{xu(1-xu+xu^2)}{(1-u)(1-xu)}d(x;uv)\notag\\
&\quad-\frac{xu}{(1-u)(1-xu)}d(xu;v), \label{l2c229e2}
\end{align}
\begin{equation}\label{l2c229e3}
c(x;u,v)=xva(x;uv,1)-\frac{x^2uv^2}{1-xv}(a(xuv;1,1)+1)+\frac{xv^2}{1-xv}d(x;uv),
\end{equation}
\begin{equation}\label{l2c229e4}
(1-x)d(x;u)=x^2u(a(xu;1,1)+1)+xb(x;u,1).
\end{equation}
\end{lemma}

We can now determine the generating function $F_T(x)$.

\begin{theorem}\label{th229A3}
Let $T=\{1342,1423,2134\}$.  Then
$$F_T(x)=\frac{1-2x+2x^2-\sqrt{1-8x+20x^2-24x^3+16x^4-4x^5}}{2x(1-x+x^2)}.$$
\end{theorem}
\begin{proof}
In the notation above, we seek to determine $1+a(x;1,1)$.
By \eqref{l2c229e4}, we have $d(x;u)=\frac{x^2u}{1-x}(a(xu;1,1)+1)+\frac{x}{1-x}b(x;u,1)$. Thus, equation \eqref{l2c229e2} with $v=1$ gives
\begin{align*}
&\left(1-x-\frac{x}{1-u}-\frac{x^2(1-x+xu)}{(1-u)(1-x)(u-x)}\right)b(x/u;u,1)\\
&\qquad=-\left(\frac{x}{1-u}+\frac{x^2}{(1-u)(1-x)^2}\right)b(x;1,1)\\
&\qquad\quad+\left(\frac{x^3(1-x+xu)}{(1-u)(1-x)(u-x)}-\frac{x^3}{(1-u)(1-x)^2}\right)(a(x;1,1)+1).
\end{align*}
Applying the kernel method to this last equation, and setting
$$u=u_0=\frac{1-2x+\sqrt{1-8x+20x^2-24x^3+16x^4-4x^5}}{2(1-x)^2},$$
we obtain
\begin{align*}
b(x;1,1)=\frac{x^2(1-u_0)(a(x;1,1)+1)}{u_0-x}.
\end{align*}
Note that $c(x;1,1)=\frac{x((1-2x)a(x;1,1)-x+d(x;1))}{1-x}$ by \eqref{l2c229e3}, and
$$a(x;1,1)=x+xa(x;1,1)+b(x;1,1)+c(x;1,1)+d(x;1)$$
by \eqref{l2c229e1}.
Substituting out $c(x;1,1)$, and then $d(x;1)$ and $b(x;1,1)$, in the preceding equation and solving the equation that results for $a(x;1,1)$ yields
$$a(x;1,1)=\frac{x^3+x(1-x)^2u_0}{x(2x^2-2x+1)-(1-x)^3u_0}.$$
Substituting the expression for $u_0$ into the last equation gives the desired formula for $1+a(x;1,1)$ and completes the proof.
\end{proof}

\subsection{Case 234}
The two representative triples $T$ are:

\{2143,2314,2413\} (Theorem \ref{th234A1})

\{1243,1342,3142\} (Theorem \ref{th234A2})

\begin{theorem}\label{th234A1}
Let $T=\{2143,2314,2413\}$. Then
$$F_T(x)=\frac{(1-x)^2-\sqrt{(1-x)^4-4x(1-2x)(1-x)}}{2x(1-x)}\, .$$
\end{theorem}
\begin{proof}
Let $G_m(x)$ be the generating function for $T$-avoiders with $m$ left-right maxima.
Clearly, $G_0(x)=1$ and $G_1(x)=xF_T(x)$.
Now suppose $\pi=i_1\pi^{(1)}\cdots i_m\pi^{(m)}$
is a permutation that avoids $T$ with $m\geq2$ left-right maxima.  Then $\pi^{(m)}$ has the form $\beta_m  \beta_{m-1} \cdots \beta_1$ with $\beta_1<i_1<\beta_2<i_2< \cdots < \beta_m<i_m$ because $c<d$ in $\pi^{(m)}$ with $c<i_j<d$ implies $i_j i_m cd $ is a $2413$.

If $\pi^{(1)}=\cdots =\pi^{(m-1)}=\emptyset$, the contribution is $(xF_T(x))^m$.
Otherwise, let $k$ be minimal such that $\pi^{(k)}\ne \emptyset$. Then $\pi$ has the form
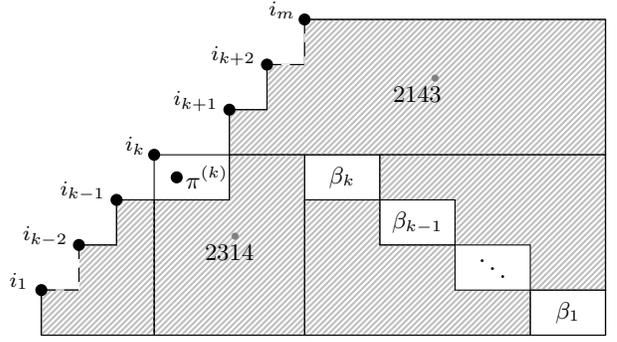
\begin{center}
\begin{pspicture}(-1.5,0)(9.5,4.5)
\psset{xunit =.5cm, yunit=.3cm,linewidth=.5\pslinewidth}
\psline[linestyle=dashed](0,2)(1,2)(1,4)
\psline(3,0)(3,6)(5,6)(5,8)
\psline(0,0)(0,2)
\psline(1,4)(2,4)(2,6)(3,6)(3,8)(5,8)(5,10)(6,10)(6,12)
\psline(7,14)(15,14)(15,0)(0,0)
\psline(7,8)(9,8)(9,6)(11,6)(11,4)(9,4)(9,6)(7,6)(7,8)
\psline(13,2)(15,2)(15,0)(13,0)(13,2)
\pspolygon[fillstyle=hlines,hatchcolor=lightgray,hatchsep=0.8pt](5,8)(5,10)(6,10)(6,12)(7,12)(7,14)(15,14)(15,8)(5,8)
\pspolygon[fillstyle=hlines,hatchcolor=lightgray,hatchsep=0.8pt](5,8)(7,8)(7,0)(3,0)(3,6)(5,6)(5,8)
\pspolygon[fillstyle=hlines,hatchcolor=lightgray,hatchsep=0.8pt](7,0)(7,6)(9,6)(9,4)(11,4)(11,2)(13,2)(13,0)(7,0)
\pspolygon[fillstyle=hlines,hatchcolor=lightgray,hatchsep=0.8pt](9,6)(11,6)(11,4)(13,4)(13,2)(15,2)(15,8)(9,8)(9,6)
\pspolygon[fillstyle=hlines,hatchcolor=lightgray,hatchsep=0.8pt](0,0)(0,2)(1,2)(1,4)(2,4)(2,6)(3,6)(3,0)(0,0)

\psline[linecolor=white](6,12)(7,12)(7,14)
\psline[linestyle=dashed](6,12)(7,12)(7,14)
\psline[linecolor=white](0,2)(1,2)(1,4)
\psline[linestyle=dashed](0,2)(1,2)(1,4)
\rput(4.4,7){\textrm{{\small $\pi^{(k)}$}}}
\rput(8,7){\textrm{{\small $\beta_k$}}}
\rput(10,5){\textrm{{\small $\beta_{k-1}$}}}
\rput(14,1){\textrm{{\small $\beta_1$}}}
\rput(12,3.3){$\ddots$}
\rput(15.4,0){,}
\pscircle*(0,2){.08}\pscircle*(1,4){.08}\pscircle*(2,6){.08}\pscircle*(3,8){.08}\pscircle*(5,10){.08}
\pscircle*(6,12){.08}\pscircle*(7,14){.08}
\pscircle*(3.6,7){.08}
\rput(5,4){\textrm{{\small $23\overset{{\gray \bullet}}{1}4$}}}
\rput(10,11){\textrm{{\small $214\overset{{\gray \bullet}}{3}$}}}
\rput(-0.6,2.4){\textrm{{\footnotesize $i_1$}}}
\rput(.1,4.4){\textrm{{\footnotesize $i_{k-2}$}}}
\rput(1.1,6.4){\textrm{{\footnotesize $i_{k-1}$}}}
\rput(2.5,8.4){\textrm{{\footnotesize $i_k$}}}
\rput(4.1,10.4){\textrm{{\footnotesize $i_{k+1}$}}}
\rput(5.1,12.4){\textrm{{\footnotesize $i_{k+2}$}}}
\rput(6.4,14.4){\textrm{{\footnotesize $i_{m}$}}}
\end{pspicture}
\end{center}
where dark bullets indicate mandatory entries and some shaded regions are empty because the gray bullet would form part of the indicated pattern; $\pi^{(k)}i_m \beta_k$ avoids $T$ and does not start with its largest entry, and $\beta_{k-1},\dots, \beta_1$ all avoid $T$. Thus, the contribution for fixed $k\in [m]$ is given by $x^{m-1}(F_T(x)-1-xF_T(x))F_T(x)^{k-1}$.

Hence, for $m\ge 2$,
$$G_m(x)=(xF_T(x))^m+x^{m-1}(F_T(x)-1-xF_T(x))\sum_{k=0}^{m-1}F_T(x)^k.$$
Summing over $m\ge 0$, we obtain
\begin{align*}
F_T(x)&=1+\frac{xF_T(x)}{1-xF_T(x)}+\frac{\big(F_T(x)-1-xF_T(x)\big)\left(\frac{x}{1-x}-\frac{xF_T(x)}{1-xF_T(x)}\right)}{1-F_T(x)},
\end{align*}
which has the desired solution.
\end{proof}

\begin{theorem}\label{th234A2}
Let $T=\{1243,1342,3142\}$. Then
$$F_T(x)=\frac{(1-x)^2-\sqrt{(1-x)^4-4x(1-2x)(1-x)}}{2x(1-x)}.$$
\end{theorem}
\begin{proof}
Let $G_m(x)$ be the generating function for $T$-avoiders with $m$ left-right maxima.
Clearly, $G_0(x)=1$ and $G_1(x)=xF_T(x)$.
For $m=2$, suppose $\pi=i\pi' n\pi''$ is a permutation in $S_n(T)$ with two left-right maxima.
Let $\beta$ denote the subsequence of letters less than $i$ in $\pi''$. Then $\beta<\pi'\ (a\in \pi'$ and $b \in \beta$ with $a<b$ implies $ianb$ is a 3142) and so $\pi$ is as in the figure.
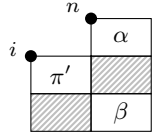
\begin{center}
\begin{pspicture}(-1.5,0)(3,1.8)
\psset{xunit =.8cm, yunit=.5cm,linewidth=.5\pslinewidth}
\pspolygon[fillstyle=hlines,hatchcolor=lightgray,hatchsep=0.8pt](1,1)(1,2)(2,2)(2,1)(1,1)
\pspolygon[fillstyle=hlines,hatchcolor=lightgray,hatchsep=0.8pt](0,0)(0,1)(1,1)(1,0)(0,0)
\psline(0,2)(0,1)(2,1)(2,0)(1,0)(1,3)(2,3)(2,2)(0,2)
\rput(1.5,2.5){\textrm{{\small $\alpha$}}}
\rput(1.5,0.5){\textrm{{\small $\beta$}}}
\rput(0.5,1.5){\textrm{{\small $\pi'$}}}
\pscircle*(0,2){.08}\pscircle*(1,3){.08}
\rput(-0.3,2.2){\textrm{{\footnotesize $i$}}}
\rput(0.7,3.3){\textrm{{\footnotesize $n$}}}
\end{pspicture}
\end{center}
If $\alpha=\emptyset$, then $\pi'$ and $\beta$ avoid $T$ and the contribution is $x^2 F_T(x)^2$.
If $\alpha\ne \emptyset$ so that $i+1 \in \alpha$, then $\pi'$ is decreasing (or $n(i+1)$ would be the 43 of a 1243), and St($i\pi''$) is a $T$-avoider that does not start with its maximal element. Hence, the contribution is
$\frac{x}{1-x}\big(F_T(x)-1-xF_T(x)\big)$.
Thus,
\[
G_2(x)=x^2F_T(x)^2+\frac{x}{1-x}\big(F_T(x)-1-xF_T(x)\big).
\]
For $m\ge 3$, $\pi$ has the form
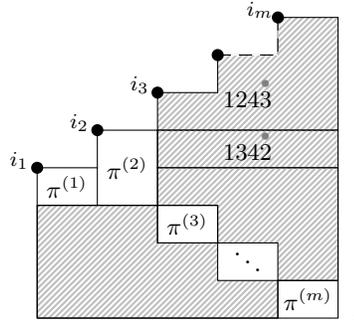
\begin{center}
\begin{pspicture}(-1.5,0)(6,4.5)
\psset{xunit =.8cm, yunit=.5cm,linewidth=.5\pslinewidth}
\pspolygon[fillstyle=hlines,hatchcolor=lightgray,hatchsep=0.8pt](2,5)(2,6)(3,6)(3,7)(4,7)(4,8)(5,8)(5,5)(2,5)
\pspolygon[fillstyle=hlines,hatchcolor=lightgray,hatchsep=0.8pt](5,5)(2,5)(2,4)(5,4)(5,5)
\pspolygon[fillstyle=hlines,hatchcolor=lightgray,hatchsep=0.8pt](0,0)(0,3)(2,3)(2,2)(3,2)(3,1)(4,1)(4,0)(0,0)
\pspolygon[fillstyle=hlines,hatchcolor=lightgray,hatchsep=0.8pt](2,3)(2,4)(5,4)(5,1)(4,1)(4,2)(3,2)(3,3)(2,3)
\psline[linecolor=white](3,7)(4,7)(4,8)
\psline[linestyle=dashed](3,7)(4,7)(4,8)
\psline(0,0)(0,4)(1,4)(1,3)(0,3)
\psline(1,4)(1,5)(2,5)(2,2)(3,2)(3,3)(1,3)
\psline(0,0)(5,0)(5,4)
\psline(4,0)(4,1)(5,1)
\rput(3.5,1.7){$\ddots$}
\rput(.5,3.5){\textrm{{\small $\pi^{(1)}$}}}
\rput(1.5,4){\textrm{{\small $\pi^{(2)}$}}}
\rput(2.5,2.5){\textrm{{\small $\pi^{(3)}$}}}
\rput(4.5,.5){\textrm{{\small $\pi^{(m)}$}}}
\pscircle*(0,4){.08}\pscircle*(1,5){.08}\pscircle*(2,6){.08}\pscircle*(3,7){.08}\pscircle*(4,8){.08}
\rput(-0.3,4.2){\textrm{{\footnotesize $i_1$}}}
\rput(.7,5.2){\textrm{{\footnotesize $i_2$}}}
\rput(1.7,6.2){\textrm{{\footnotesize $i_3$}}}
\rput(3.7,8.2){\textrm{{\footnotesize $i_m$}}}
\rput(3.5,4.6){\textrm{{\small $134\overset{{\gray \bullet}}{2}$}}}
\rput(3.5,6){\textrm{{\small $124\overset{{\gray \bullet}}{3}$}}}
\rput(5.3,0){,}
\end{pspicture}
\end{center}
where some shaded regions are empty to avoid the indicated pattern and the $\pi$'s are in their relative positions to avoid 3142. Hence, $G_m(x)=G_2(x)(xF_T(x))^{m-2}$.

Summing over $m\ge 0$, we obtain
\begin{align*}
F_T(x)&=1+xF_T(x)+\frac{x^2F_T(x)+\frac{x}{1-x}(F_T(x)-1-xF_T(x))}{1-xF_T(x)},
\end{align*}
which has the desired solution.
\end{proof}
\subsection{Case 235}
The three representative triples $T$ are:

\{1423,1432,2143\} (Theorem \ref{th235A1})

\{1423,1432,3142\} (Theorem \ref{th235A2})

\{1234,1243,2314\} (Theorem \ref{th235A3})

\subsubsection{$\mathbf{T=\{1423,1432,2143\}}$}
Let $a(n;i_1,i_2,\ldots,i_k)$, $a(n)$ and $\mathcal{T}_{i,j}$ be as in the second case of class 171 above.  Note here that $a(n;n)=a(n;n-1)=a(n-1)$ for $n\geq2$.  It is convenient to consider separately the case of a permutation starting $i,j,j+2$, where $j \leq i-3$.  Define $f(n;i,j)=a(n;i,j,j+2)$ for $4 \leq i \leq n$ and $1 \leq j \leq i-3$.  The arrays $a(n;i,j)$ and $f(n;i,j)$ are determined recursively as follows.

\begin{lemma}\label{235al1}
We have
\begin{equation}\label{235al1e1}
a(n;i,i+2)=a(n-1;i,i+2)+a(n-1;i+1,i)+\sum_{j=1}^{i-1}a(n-1;i,j), \qquad 1 \leq i \leq n-2,
\end{equation}
\begin{equation}\label{235al1e2}
a(n;i,i-1)=a(n-1;i,i-1)+\sum_{j=1}^{i-2}a(n-1;i-1,j), \qquad 2 \leq i \leq n-1,
\end{equation}
\begin{equation}\label{235al1e3}
a(n;i,i-2)=a(n-1;i,i-2)+a(n-1;i-1,i-2)+\sum_{j=1}^{i-3}a(n-1;i-2,j),\qquad 3\leq i \leq n-1,
\end{equation}
\begin{equation}\label{235al1e4}
a(n;i,j)=a(n-1;i-1,j)+f(n;i,j)+\sum_{\ell=1}^{j-1}a(n-1;j,\ell), \qquad 1 \leq j \leq i-3,
\end{equation}
and
\begin{equation}\label{235al1e5}
f(n;i,j)=f(n-1;i-1,j)+a(n-2;j+1,j)+\sum_{\ell=1}^{j-1}a(n-2;j,\ell), \qquad 1 \leq j \leq i-4,
\end{equation}
with $f(n;i,i-3)=a(n-1;i-1,i-3)$ for $4 \leq i \leq n$, $a(n;i,i+1)=a(n-1;i)$ for $1 \leq i \leq n-1$, and $a(n;i,j)=0$ for $1 \leq i \leq j-3 \leq n-3$.
\end{lemma}
\begin{proof}
The formulas for $f(n;i,i-3)$ and $a(n;i,i+1)$, and for $a(n;i,j)$ when $i \leq j-3$, follow from the definitions.  In the cases that remain, let $x$ denote the third letter of a $T$-avoiding permutation.  For \eqref{235al1e1}, first note that members of $\mathcal{T}_{i,i+2}$ where $i<n-2$ must have $x=i+3$, $x=i+1$ or $x<i$, lest there be an occurrence of $1423$ or $1432$.  The letter $i+2$ can be deleted in the first case, while the letter $i$ can in the second, giving $a(n-1;i,i+2)$ and $a(n-1;i+1,i)$ possibilities, respectively.  If $x<i$, then $i,x$ imposes a stricter requirement on subsequent letters than does $i+2,x$, whence $i+2$ may be deleted in this case.  This gives $\sum_{j=1}^{i-1}a(n-1;i,j)$ possibilities, which implies \eqref{235al1e1} when $i<n-2$.  Equation \eqref{235al1e1} is also seen to hold when $i=n-2$ since there is no $x=i+3$ case with $a(n-1;i,i+2)=0$ accordingly.  For \eqref{235al1e2}, note that members of $\mathcal{T}_{i,i-1}$ where $i<n$ must have $x=i+1$ or $x<i-2$ so as to avoid $2143$.  This yields $a(n-1;i,i-1)$ and $\sum_{j=1}^{i-2}a(n-1;i-1,j)$ possibilities, respectively, which implies \eqref{235al1e2}.  For \eqref{235al1e3}, note that members of $\mathcal{T}_{i,i-2}$ where $i<n$ must have $x=i+1$, $x=i-1$ or $x<i-2$, yielding $a(n-1;i,i-2)$, $a(n-1;i-1,i-2)$ and $\sum_{j=1}^{i-3}a(n-1;i-2,j)$ possibilities, respectively.

To show \eqref{235al1e4}, first observe that members of $\mathcal{T}_{i,j}$ where $j \leq i-3$ must have $x=j+1$, $x=j+2$ or $x<j$, lest there be an occurrence of $1423$ or $1432$.  If $x=j+1$, then there are $a(n-1;i-1,j)$ possibilities since the letter $j+1$ is extraneous and may be deleted.  If $x=j+2$, then there are $f(n;i,j)$ possibilities, by definition.  If $x<j$, then the letter $i$ may be deleted, which gives the last term on the right-hand side of \eqref{235al1e4}.  Finally, to show \eqref{235al1e5}, let $y$ denote the fourth letter of a permutation enumerated by $f(n;i,j)$ where $j<i-3$.  Then we must have $y=j+3$, $y=j+1$ or $y<j$.  If $y=j+3$, then $y$ may be deleted, yielding $f(n-1;i-1,j)$ possibilities, by definition.  If $y=j+1$, then $j+2,j+1$ is seen to impose a stricter requirement on subsequent letters than does $i,j$ with regard to $2143$, with $j+1$ also making $j$ redundant concerning $1423$ or $1432$.  Thus, both $i$ and $j$ may be deleted in this case, giving $a(n-2;j+1,j)$ possibilities.  Finally, if $y<j$, then both the $i$ and $j+2$ may be deleted and thus there are $\sum_{\ell=1}^{j-1}a(n-2;j,\ell)$ possibilities, which implies \eqref{235al1e5} and completes the proof.
\end{proof}

To aid in solving the recurrences of the prior lemma, we define the following auxiliary functions:  $b(n;i)=\sum_{j=1}^{i-1}a(n;i,j)$, $c(n;i)=a(n;i,i-2)$, $d(n;i)=a(n;i,i-1)$ and $e(n;i)=a(n;i,i+2)$.  Assume functions are defined on the natural range for $i$, given $n$, and are zero otherwise.  For example, $c(n;i)$ is defined for $3 \leq i \leq n$, with $c(n;1)=c(n;2)=0$.  Let $f(n;i)=\sum_{j=1}^{i-3}f(n;i,j)$ for $4 \leq i \leq n$.

The recurrences in the previous lemma may be recast as follows.

\begin{lemma}\label{235al2}
We have
\begin{equation}\label{235al2e1}
a(n;i)=a(n-1;i)+b(n;i)+e(n;i), \qquad 1 \leq i \leq n-1,
\end{equation}
\begin{align}
b(n;i)&=c(n;i)+d(n;i)+b(n-1;i-1)-d(n-1;i-1)+f(n;i)\notag\\
&\quad+\sum_{j=1}^{i-3}b(n-1;j), \qquad 2 \leq i \leq n-1,\label{235al2e2}
\end{align}
\begin{equation}\label{235al2e3}
c(n;i)=b(n-1;i-2)+c(n-1;i)+d(n-1;i-1), \qquad 3 \leq i \leq n-1,
\end{equation}
\begin{equation}\label{235al2e4}
d(n;i)=b(n-1;i-1)+d(n-1;i), \qquad 2 \leq i \leq n-1,
\end{equation}
\begin{equation}\label{235al2e5}
e(n;i)=b(n-1;i)+d(n-1;i+1)+e(n-1;i), \qquad 1 \leq i \leq n-2,
\end{equation}
and
\begin{equation}\label{235al2e6}
f(n;i)=c(n-1;i-1)+f(n-1;i-1)+\sum_{j=1}^{i-4}b(n-2;j)+\sum_{j=1}^{i-4}d(n-2;j+1), \qquad 4 \leq i \leq n.
\end{equation}
\end{lemma}
\begin{proof}
For \eqref{235al2e1}, note that by the definitions, we have
\begin{align*}
a(n;i)&=\sum_{i=1, i\neq j}^na(n;i,j)=\sum_{j=1}^{i-1}a(n;i,j)+a(n;i,i+1)+a(n;i,i+2)\\
&=b(n;i)+a(n-1;i)+e(n;i).
\end{align*}
For \eqref{235al2e2}, note that by summing \eqref{235al1e4} over $j$ and the definitions, we have
\begin{align*}
b(n;i)&=a(n;i,i-2)+a(n;i,i-1)+\sum_{j=1}^{i-3}a(n;i,j)\\
&=a(n;i,i-2)+a(n;i,i-1)+\sum_{j=1}^{i-3}a(n-1;i-1,j)+\sum_{j=1}^{i-3}f(n;i,j)+\sum_{j=1}^{i-3}b(n-1;j)\\
&=c(n;i)+d(n;i)+(b(n-1;i-1)-d(n-1;i-1))+f(n;i)+\sum_{j=1}^{i-3}b(n-1;j).
\end{align*}
Next, observe that formulas \eqref{235al2e3}, \eqref{235al2e4} and \eqref{235al2e5} follow directly from the definitions and recurrences \eqref{235al1e3}, \eqref{235al1e2} and \eqref{235al1e1}, respectively.  Finally, formula \eqref{235al2e6} follows from summing \eqref{235al1e5} over $1 \leq j \leq i-4$ and noting $f(n;i,i-3)=c(n-1;i-1)$.
\end{proof}

Define $a_n(u)=\sum_{i=1}^na(n;i)u^i$ for $n\geq1$, $b_n(u)=\sum_{i=2}^{n-1}b(n;i)u^i$ for $n\geq3$, $c_n(u)=\sum_{i=3}^{n-1}c(n;i)u^i$ for $n\geq4$, $d_n(u)=\sum_{i=2}^{n-1}d(n;i)u^i$ for $n\geq3$, $e_n(u)=\sum_{i=1}^{n-2}e(n;i)u^i$ for $n \geq3$, and $f_n(u)=\sum_{i=4}^n f(n;i)u^i$ for $n\geq4$.  Assume all functions take the value zero if $n$ is such that the sum in question is empty.  Note that $a_1(u)=u$, with $b_3(u)=d_3(u)=u^2$.

Multiplying both sides of \eqref{235al2e1} by $u^i$, and summing over $1 \leq i \leq n-1$, yields
\begin{equation}\label{235ae1}
a_n(u)=a(n-1)u^n+a_{n-1}(u)+b_n(u)+e_n(u), \qquad n \geq2.
\end{equation}
Note that, by the definitions,
\begin{align*}
f(n;n)&=\sum_{j=1}^{n-3}f(n;n,j)=\sum_{j=1}^{n-3}a(n-1;j,j+2)=\sum_{j=1}^{n-3}e(n-1;j)=e_{n-1}(1), \qquad n \geq 4,
\end{align*}
and
$$b(n;n-1)=a(n;n-1)-a(n;n-1,n)=a(n-1)-a(n-2), \qquad n \geq 2.$$
By recurrence \eqref{235al2e2}, we then have
\begin{align}
b_n(u)&=c_n(u)+d_n(u)+u(b_{n-1}(u)-d_{n-1}(u))+f_n(u)-f(n;n)u^n+\sum_{j=1}^{n-3}b(n-1;j)\sum_{i=j+3}^{n-1}u^i\notag\\
&=c_n(u)+d_n(u)+u(b_{n-1}(u)-d_{n-1}(u))+f_n(u)-e_{n-1}(1)u^n\notag\\
&\quad+\frac{u^3}{1-u}(b_{n-1}(u)-(a(n-2)-a(n-3))u^{n-2})-\frac{u^n}{1-u}(b_{n-1}(1)-(a(n-2)-a(n-3)))\notag\\
&=c_n(u)+d_n(u)+u(b_{n-1}(u)-d_{n-1}(u))+f_n(u)-e_{n-1}(1)u^n\notag\\
&\quad+\frac{u}{1-u}(u^2b_{n-1}(u)-u^{n-1}b_{n-1}(1))+(a(n-2)-a(n-3))u^n, \qquad n \geq 3. \label{235ae2}
\end{align}

From recurrence \eqref{235al2e3}, we get
\begin{align}
c_n(u)&=u^2(b_{n-1}(u)-b(n-1;n-2)u^{n-2})+c_{n-1}(u)+c(n-1;n-1)u^{n-1}+ud_{n-1}(u)\notag\\ &=u^2b_{n-1}(u)-u^na(n-2)+u^{n-1}(1+u)a(n-3)+c_{n-1}(u)+ud_{n-1}(u), \qquad n \geq 4. \label{235ae3}
\end{align}
By \eqref{235al2e4} and \eqref{235al2e5}, we have
\begin{equation}\label{235ae4}
d_n(u)=a(n-3)u^{n-1}+ub_{n-1}(u)+d_{n-1}(u), \qquad n \geq3,
\end{equation}
and
\begin{equation}\label{235ae5}
e_n(u)=a(n-3)u^{n-2}+b_{n-1}(u)+\frac{1}{u}d_{n-1}(u)+e_{n-1}(u), \qquad n \geq 3.
\end{equation}
Finally, multiplying both sides of \eqref{235al2e6} by $u^i$, and summing over $4 \leq i \leq n$, yields
\begin{align}
f_n(u)&=u(c(n-1;u)+a(n-3)u^{n-1})+uf_{n-1}(u)\notag\\
&\quad+\sum_{j=1}^{n-3}b(n-2;j)\sum_{i=j+4}^nu^i+\sum_{j=1}^{n-4}d(n-2;j+1)\sum_{i=j+4}^nu^i\notag\\
&=a(n-3)u^n+uc_{n-1}(u)+uf_{n-1}(u)\notag\\
&\quad+\frac{u^3}{1-u}(ub_{n-2}(u)+d_{n-2}(u)-u^{n-2}(b_{n-2}(1)+d_{n-2}(1))), \qquad n \geq 4. \label{235ae6}
\end{align}

Define the generating functions $a(x;u)=\sum_{n\geq1}a_n(u)x^n$, $b(x;u)=\sum_{n\geq3}b_n(u)x^n$, $c(x;u)=\sum_{n\geq4}c_n(u)x^n$, $d(x;u)=\sum_{n\geq3}d_n(u)x^n$, $e(x;u)=\sum_{n\geq3}e_n(u)x^n$ and $f(x;u)=\sum_{n\geq 4}f_n(u)x^n$.  Recall that $a(n)=a_n(1)$ for $n\geq 1$, with $a(0)=1$.  Rewriting recurrences \eqref{235ae1}--\eqref{235ae6} in terms of generating functions yields the following system of functional equations.

\begin{lemma}\label{235al3}
We have
\begin{equation}\label{235a1f}
(1-x)a(x;u)=xu(1+a(xu;1))+b(x;u)+e(x;u),
\end{equation}
\begin{align}
(1-xu)b(x;u)&=-x^3u^3+c(x;u)+(1-xu)d(x;u)-xue(xu;1)+f(x;u)\notag\\
&\quad+x^2u^2(1-xu)a(xu;1)+\frac{xu}{1-u}(u^2b(x;u)-b(xu;1)),\label{235a2f}
\end{align}
\begin{equation}\label{235a3f}
(1-x)c(x;u)=x^3u^3-x^2u^2(1-x-xu)a(xu;1)+xu^2b(x;u)+xud(x;u),
\end{equation}
\begin{equation}\label{235a4f}
(1-x)d(x;u)=x^3u^2(1+a(xu;1))+xub(x;u),
\end{equation}
\begin{equation}\label{235a5f}
(1-x)e(x;u)=x^3u(1+a(xu;1))+xb(x;u)+\frac{x}{u}d(x;u),
\end{equation}
and
\begin{align}
(1-xu)f(x;u)&=x^3u^3a(xu;1)+xuc(x;u)\notag\\
&\quad+\frac{x^2u^3}{1-u}(ub(x;u)+d(x;u)-b(xu;1)-d(xu;1)).\label{235a6f}
\end{align}
\end{lemma}

We now determine the generating function $F_T(x)$.

\begin{theorem}\label{th235A1}
Let $T=\{1423,1432,2143\}$.  Then $y=F_T(x)$ satisfies the equation
$$y=1-x+xy+x(1-2x)y^2+x^2y^3.$$
\end{theorem}
\begin{proof}
By solving \eqref{235a1f}, \eqref{235a4f} and \eqref{235a5f} with $u=1$ for $b(x;1)$, $d(x;1)$ and $e(x;1)$, we obtain
\begin{align*}
b(x;1)&=\frac{1-4x+5x^2-3x^3}{1-x+x^2}a(x;1)-\frac{x(1-2x+2x^2)}{1-x+x^2},\\
d(x;1)&=\frac{x(1-x)^3}{1-x+x^2}a(x;1)-\frac{x^2(1-x)^2}{1-x+x^2},\\
e(x;1)&=\frac{x(1-x)^2}{1-x+x^2}a(x;1)-\frac{x^2(1-x)}{1-x+x^2}.
\end{align*}
Define $K(x;u)=u^2(1-u)-xu(2-u^2)+x^2(1+2u-2u^2)-x^3$. Substituting the expressions for $b(x;1)$, $d(x;1)$ and $e(x;1)$ into \eqref{235a2f}--\eqref{235a6f}, and then solving for $b(x/u;u)$, $c(x/u;u)$, $d(x/u;u)$, $e(x/u;u)$ and $f(x/u;u)$, yields
\begin{align*}
K(x;u)b(x/u;u)&=x(-u^2+2xu(u+1)-x^2(u^2+3u+1)+x^3(2u+1))a(x;1)\\
&\quad+x^2u(2x^2-xu-x+u),\\
K(x;u)e(x/u;u)&=x^2(1-x)(x-u)a(x;1)+x^3(1-x).
\end{align*}
Multiplying both sides of \eqref{235a1f} by $K(x;u)$, and then substituting the expressions of $K(x;u)b(x/u;u)$ and $K(x;u)e(x/u;u)$, gives
\begin{align*}
(1-x/u)K(x;u)a(x/u;u)&=x(x-u)(u^2+x(1-u-u^2)+x^2(2u-1))a(x;1)\\
&\quad+x(1-u)(u^2-xu(2+u)+x^2(2+3u)-2x^3).
\end{align*}
To solve this last equation, we let $u=u_0=u_0(x)$ such that $K(x;u_0(x))=0$. Then
\begin{align*}
F_T(x)=1+a(x;1)&=\frac{(1-x)(x^2-xu_0+u_0^2)}{(u_0-x)(x(1-x)-x(1-2x)u_0+(1-x)u_0^2)}.
\end{align*}
Using the fact that $u_0^3=u_0^2(1-u_0)-xu_0(2-u_0^2)+x^2(1+2u_0-2u_0^2)$, we obtain
\begin{align*}
&1-x-(1-x)F_T(x)+x(1-2x)F_T^2(x)+x^2F_T^3(x)\\
&\quad=\frac{(1-x)^2K(x;u_0)V(x;u_0)}{(x-u_0)^3(x(1-x)-x(1-2x)u_0+(1-x)u_0^2)^3}=0,
\end{align*}
where
\begin{align*}
V(x;u)&=-x^5(2x^4+7x^2(1-x)-5x+2)+x^2(x+1)(4x^4-7x^3+8x^2-5x+1)(u-x)\\
&\quad-(7x^6+2x^5(1-x^2)-x^4-37x^3(1-x)+24x^2(1-x^2)-8x+1)(u-x)^2.
\end{align*}
Hence, the generating function $F_T(x)$ satisfies
$$F_T(x)=1-x+xF_T(x)+x(1-2x)F_T^2(x)+x^2F_T^3(x),$$
as desired.
\end{proof}

\subsubsection{$\mathbf{T=\{1423,1432,3142\}}$}
\begin{theorem}\label{th235A2}
Let $T=\{1423,1432,3142\}$. Then $y=F_T(x)$ satisfies the equation
$$y=1-x+xy+x(1-2x)y^2+x^2y^3.$$
\end{theorem}
\begin{proof}
Let $G_m(x)$ be the generating function for $T$-avoiders with $m$ left-right maxima.
Clearly, $G_0(x)=1$ and $G_1(x)=xF_T(x)$. Now suppose  $\pi=i_1\pi^{(1)}i_2\pi^{(2)}\cdots i_m\pi^{(m)}\in S_n(T)$ with $m\ge 2$ left-right maxima. Since $\pi$ avoids $1423$ and $1432$, we have that either $i_2=i_1+1$ or $i_2=i_1+2$.
\begin{itemize}
\item The case $i_2=i_1+1$. Since $\pi$ avoids $3142$, we see that there is no element between the minimal element of $\pi^{(1)}$ and $i_1$ in $\pi^{(2)}\pi^{(3)}\cdots\pi^{(m)}$. Thus, the contribution in this case is $xF_T(x)G_{m-1}(x)$, where $xF_T(x)$ accounts for the section $i_1\pi^{(1)}$ and $G_{m-1}(x)$ for $i_2\pi^{(2)}\cdots i_m\pi^{(m)}$.

\item The case $i_2=i_1+2$. Let $j$ be the index with  $i_1+1 \in \pi^{(j)}$.
Then $\pi$ has the form
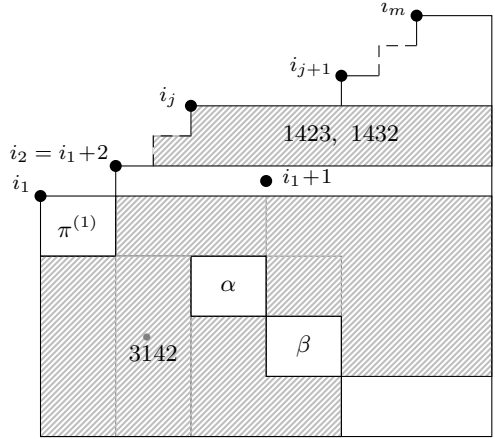
\begin{center}
\begin{pspicture}(-1.5,-0.5)(9.5,6.3)
\psset{xunit =.5cm, yunit=.4cm,linewidth=.5\pslinewidth}
\psline[linecolor=gray](6,6)(6,8)
\psline[linecolor=gray](4,4,)(4,0)
\psline[linecolor=gray](6,6)(8,6)(8,4)
\psline[linecolor=gray](4,0)(4,4)
\psline[linecolor=gray](2,0)(2,6)
\psline(0,0)(12,0)(12,14)(10,14)(10,13)
\psline(4,6)(4,4)(8,4)(8,0)
\psline(4,6)(6,6)(6,2)(8,2)
\psline(0,6)(0,8)(2,8)(2,6)(0,6)
\psline(8,11)(8,12)(9,12)
\psline(2,8)(2,9)(3,9)
\pspolygon[fillstyle=hlines,hatchcolor=lightgray,hatchsep=0.8pt](3,9)(3,10)(4,10)(4,11)(12,11)(12,9)(3,9)
\pspolygon[fillstyle=hlines,hatchcolor=lightgray,hatchsep=0.8pt](0,0)(0,6)(4,6)(4,4)(6,4)(6,2)(8,2)(8,0)(0,0)
\pspolygon[fillstyle=hlines,hatchcolor=lightgray,hatchsep=0.8pt](2,8)(2,6)(6,6)(6,4)(8,4)(8,2)(12,2)(12,8)(2,8)
\psline[linecolor=white](2,6)(4,6)
\psline[linecolor=white](3,9)(3,10)(4,10)
\psline[linestyle=dashed](3,9)(3,10)(4,10)
\psline[linecolor=white](9,12)(9,13)(10,13)
\psline[linestyle=dashed](9,12)(9,13)(10,13)
\rput(1,7){\textrm{{\small $\pi^{(1)}$}}}
\rput(5,5){\textrm{{\small $\alpha$}}}
\rput(7,3){\textrm{{\small $\beta$}}}
\pscircle*(0,8){.08}\pscircle*(2,9){.08}\pscircle*(4,11){.08}\pscircle*(8,12){.08}\pscircle*(10,14){.08}
\pscircle*(6,8.5){.08}
\rput(8,10){\textrm{{\small $1423,\ 1432$}}}
\rput(3,3){\textrm{{\small $3\overset{{\gray \bullet}}{1}42$}}}
\rput(-0.5,8.4){\textrm{{\footnotesize $i_1$}}}
\rput(.5,9.4){\textrm{{\footnotesize $i_2=i_{1}\!+\!2$}}}
\rput(3.4,11.3){\textrm{{\footnotesize $i_{j}$}}}
\rput(7.2,12.3){\textrm{{\footnotesize $i_{j+1}$}}}
\rput(7.1,8.6){\textrm{{\footnotesize $i_{1}\!+\!1$}}}
\rput(9.4,14.3){\textrm{{\footnotesize $i_{m}$}}}
\end{pspicture}
\end{center}
with $\pi^{(1)}>\alpha>\beta$ to avoid 3142, where $i_1 \pi^{(1)}\alpha\beta$ spans an interval of integers, also to avoid 3142,
and the other shaded regions are empty to avoid the indicated patterns.

Thus, for given $j$, we have a bijection between such permutations and triples \linebreak
$(\pi^{(1)},\alpha,i_j\beta i_{j+1}\pi^{(j+1)}\cdots i_m\pi^{(m)}),$
where $\pi^{(1)}$ and $\alpha$ avoid $T$, and $i_j\beta i_{j+1}\pi^{(j+1)}\cdots i_m\pi^{(m)}$ avoids $T$ with exactly $m-(j-1)$ left-right maxima.
Hence, the contribution in this case is given by $x^{j}F_T^2(x)G_{m-j+1}(x)$, where $j=2,3,\ldots,m$.
\end{itemize}
By adding all the contributions, we get
$$G_m(x)=xF_T(x)G_{m-1}+\sum_{j=2}^{m}x^jF_T^2(x)G_{m-j+1}(x),\quad m\geq2,$$
which implies
$$G_m(x)-xG_{m-1}(x)=xF_T(x)G_{m-1}-x^2F_T(x)G_{m-2}(x)+x^2F_T^2(x)G_{m-1}(x)$$
with $G_0(x)=1$ and $G_1(x)=xF_T(x)$. By summing this recurrence over all $m\geq2$, we have
$$F_T(x)-1-xF_T(x)-x(F_T(x)-1)=xF_T(x)(F_T(x)-1)-x^2F_T^2(x)+x^2 F_T^2(x)(F_T(x)-1),$$
which leads to
$$F_T(x)=1-x+xF_T(x)+x(1-2x)F_T^2(x)+x^2F_T^3(x),$$
as required.
\end{proof}

\subsubsection{$\mathbf{T=\{1234,1243,2314\}}$}
To enumerate the members of $S_n(T)$, we categorize them by their first letter and the position of the leftmost ascent.  More precisely, given $1 \leq j \leq i \leq n$, let $a(n;i,j)$ be the number of $T$-avoiding permutations of length $n$ starting with the letter $i$ whose leftmost ascent is at index $j$.  For example, we have $a(4;3,2)=3$, the enumerated permutations being $3124$, $3142$ and $3241$. If $1 \leq i \leq n$, then let $a(n;i)=\sum_{j=1}^ia(n;i,j)$ and let $a(n)=\sum_{i=1}^na(n;i)$ for $n\geq1$, with $a(0)=1$.  The array $a(n;i,j)$ satisfies the following recurrence relations.

\begin{lemma}\label{235cl1}
If $n \geq3$, then
\begin{equation}\label{235cl1e1}
a(n;i,j)=\sum_{\ell=1}^{n-i}\sum_{k=j}^{i}a(n-\ell;i,k), \qquad 1 \leq j \leq i \leq n-2.
\end{equation}
If $2 \leq j \leq n-1$, then $a(n;n-1,j)=\sum_{i=j-1}^{n-2}a(n-1;i,j-1)$, with $a(n;n-1,1)=a(n-2)$ for $n\geq2$.  If $2 \leq j \leq n$, then $a(n;n,j)=\sum_{i=j-1}^{n-1}a(n-1;i,j-1)$, with $a(n;n,1)=\delta_{n,1}$ for $n\geq1$.
\end{lemma}
\begin{proof}
Let $A_{n,i,j}$ denote the subset of $S_n(T)$ enumerated by $a(n;i,j)$.  First note that removing the initial letter $n$ from members of $A_{n,n,j}$ for $2 \leq j \leq n$ defines a bijection with $\cup_{i=j-1}^{n-1}A_{n-1,i,j-1}$ (where $A_{n,n,n}$ is understood to be the singleton set consisting of the decreasing permutation $n(n-1)\cdots1$).  This implies the formula for $a(n;n,j)$ for $j>1$, with the condition $a(n;n,1)=\delta_{n,1}$ following from the definitions.  Similarly, removing $n-1$ from members of $A_{n,n-1,j}$ when $j>1$ implies the formula for $a(n;n-1,j)$ in this case. That $a(n;n-1,1)=a(n-2)$ follows from the fact that one may safely delete both $n-1$ and $n$ from members of $S_n(T)$ starting with these letters.

To show \eqref{235cl1e1}, we first consider the possible values of $\pi_{j+1}$ within $\pi=\pi_1\pi_2\cdots\pi_n \in A_{n,i,j}$ where $i<n-1$.  Note that if $\pi_{j+1}<n-1$, then $\pi$ would contain either $1234$ or $1243$, as witnessed by the subsequences $\pi_j\pi_{j+1}(n-1)n$ or $\pi_j\pi_{j+1}n(n-1)$, which is impossible.  Thus, we must have $\pi_{j+1}=n-1$ or $n$.  If $\pi_{j+1}=n-1$, consider further the sequence of letters $\pi_{j+1}\pi_{j+2}\cdots \pi_r$, where $r$ is such that $\pi_r=n$.  If $r>j+2$, then each letter $\pi_s$ for $j+2\leq s \leq r-1$ must satisfy $\pi_s>i$, for otherwise $\pi$ would contain $2314$ (with the subsequence $i(n-1)xn$ for some $x<i$).  Furthermore, if $r>j+2$ and $\pi_{j+2}<n-2$, then $i\pi_{j+2}$ would be the first two letters in an occurrence of $1234$ or $1243$, which is impossible.  Thus, we must have $\pi_{j+2}=n-2$.  Similarly, by an inductive argument, we get $\pi_{j+1}\pi_{j+2}\cdots\pi_{r-1}\pi_r=(n-1)(n-2)\cdots(n-r+j+1)n$.  Note that each of these $\ell$ letters, where $\ell=r-j$, is seen to be extraneous concerning avoidance of $T$ and thus may be deleted.  The remaining letters comprise a member of $A_{n-\ell,i,k}$ for some $k \in [j,i]$ and hence there are $\sum_{k=j}^ia(n-\ell;i,k)$ possibilities for these letters.  Since each letter of the section $\pi_{j+1}\cdots\pi_r$ belongs to $[i+1,n]$, its length $\ell$ can range from $1$ to $n-i$, with the contents of the section determined by its length.  Allowing $\ell$ to vary implies formula \eqref{235cl1e1} and completes the proof.
\end{proof}

Let $a_{n,i}(v)=\sum_{j=1}^ia(n;i,j)v^j$ for $1 \leq i \leq n$ and $a_n(u,v)=\sum_{i=1}^na_{n,i}(v)u^i$ for $n\geq1$.  Multiplying both sides of \eqref{235cl1e1} by $v^j$, and summing over $1 \leq j \leq i$, gives
\begin{align*}
a_{n,i}(v)&=\sum_{j=1}^{i}v^j\sum_{\ell=1}^{n-i}\sum_{k=j}^ia(n-\ell;i,k)=\sum_{\ell=1}^{n-i}\sum_{k=1}^ia(n-\ell;i,k)\left(\frac{v-v^{k+1}}{1-v}\right)\\
&=\frac{v}{1-v}\sum_{\ell=1}^{n-i}(a_{n-\ell,i}(1)-a_{n-\ell,i}(v)), \qquad 1 \leq i \leq n-2,
\end{align*}
with
\begin{align*}
a_{n,n-1}(v)-a(n-2)v&=\sum_{j=2}^{n-1}v^j\sum_{i=j-1}^{n-2}a(n-1;i,j-1)=\sum_{i=1}^{n-2}\sum_{j=2}^{i+1}a(n-1;i,j-1)v^j\\
&=v\sum_{i=1}^{n-2}a_{n-1,i}(v)=v(a_{n-1}(1,v)-a_{n-1,n-1}(v))\\
&=v(a_{n-1}(1,v)-va_{n-2}(1,v)),\qquad n \geq2,
\end{align*}
and
\begin{align*}
a_{n,n}(v)=\sum_{i=1}^{n-1}\sum_{j=2}^{i+1}a(n-1;i,j-1)v^j=v\sum_{i=1}^{n-1}a_{n-1,i}(v)=va_{n-1}(1,v), \qquad n \geq1.
\end{align*}

The preceding equations then imply
\begin{align}
a_n(u,v)&=\frac{v}{1-v}\sum_{i=1}^{n-2}u^i\sum_{\ell=1}^{n-i}(a_{n-\ell,i}(1)-a_{n-\ell,i}(v))+u^{n-1}a_{n,n-1}(v)+u^na_{n,n}(v)\notag\\
&=\frac{v}{1-v}\sum_{\ell=1}^{n-1}\sum_{i=1}^{n-\ell}(a_{n-\ell,i}(1)-a_{n-\ell,i}(v))u^i-\frac{u^{n-1}v}{1-v}(a_{n-1,n-1}(1)-a_{n-1,n-1}(v))\notag\\
&\quad+u^{n-1}a_{n,n-1}(v)+u^na_{n,n}(v)\notag\\
&=\frac{v}{1-v}\sum_{\ell=1}^{n-1}(a_\ell(u,1)-a_\ell(u,v))-\frac{u^{n-1}v}{1-v}(a(n-2)-va_{n-2}(1,v))\notag\\
&\quad+u^{n-1}v(a_{n-1}(1,v)-va_{n-2}(1,v))+u^{n-1}va(n-2)+u^nva_{n-1}(1,v)\notag\\
&=\frac{v}{1-v}\sum_{\ell=1}^{n-1}(a_\ell(u,1)-a_\ell(u,v))+u^{n-1}v(1+u)a_{n-1}(1,v)\notag\\
&\quad-\frac{u^{n-1}v^2}{1-v}(a(n-2)-va_{n-2}(1,v)), \qquad n \geq 2, \label{235ce1}
\end{align}
with $a_0(u,v)=1$ and $a_1(u,v)=uv$.

Let $a(x;u,v)=\sum_{n\geq1}a_n(u,v)x^n$.  Multiplying both sides of \eqref{235ce1} by $x^n$, and summing over $n\geq2$, yields the following functional equation.

\begin{lemma}\label{235cl2}
We have
\begin{align}
a(x;u,v)&=xuv(1-xv)+\frac{xv}{(1-x)(1-v)}(a(x;u,1)-a(x;u,v))+xv(1+u)a(xu;1,v)\notag\\
&\quad-\frac{x^2uv^2}{1-v}(a(xu;1,1)-va(xu;1,v)). \label{235cl2e1}
\end{align}
\end{lemma}

We now determine the generating function $F_T(x)$.

\begin{theorem}\label{th235A3}
Let $T=\{1234,1243,2314\}$.  Then $y=F_T(x)$ satisfies the equation
$$y=1-x+xy+x(1-2x)y^2+x^2y^3.$$
\end{theorem}
\begin{proof}
In the current notation, we need to determine $1+a(x;1,1)$.  Letting $u=1$ in \eqref{235cl2e1}, and rearranging, gives
\begin{equation}\label{t235ce1}
\left((1-v)(1-2xv)-x^2v^3+\frac{xv}{1-x}\right)a(x;1,v)=xv(1-v)(1-xv)+\left(\frac{xv}{1-x}-x^2v^2\right)a(x;1,1).
\end{equation}
Setting $v=v_0$ in \eqref{t235ce1} such that
$$1-x-(1-2x^2)v_0+2x(1-x)v_0^2=x^2(1-x)v_0^3,$$
and solving for $a(x;1,1)$, implies
$$1+a(x;1,1)=\frac{x+(1-x)v_0-x(1-x)v_0^2}{1-x(1-x)v_0}.$$
Let $f(v)=1-x-(1-2x^2)v+2x(1-x)v^2-x^2(1-x)v^3$ and $h(v)=f(v)g(v)$, where $$g(v)=(1-x)(1+x^3-x(2-3x+2x^2)v-2x^3(1-x)v^2+x^3(1-x)v^3).$$  Then $y=1+a(x;1,1)$ is a solution of the equation $$1-x-(1-x)y+x(1-2x)y^2+x^2y^3=0$$
if and only if $h(v)=0$ at $v=v_0$, which is the case since $f(v_0)=0$, by definition.  This implies $F_T(x)$ is a solution of the equation stated above, as desired.
\end{proof}

\subsection{Case 238}

The five representative triples $T$ are:

\{1423,2413,3142\} (Theorem \ref{th238A1})

\{2134,2143,2413\} (Theorem \ref{th238A2})

\{1234,1342,1423\} (Theorem \ref{th238A3})

\{1324,1342,1423\} (Theorem \ref{th238A4})

\{1243,1342,1423\} (Theorem \ref{th238A5})


\subsubsection{$\mathbf{T=\{1423,2413,3142\}}$}

\begin{theorem}\label{th238A1}
Let $T=\{1423,2413,3142\}$. Then
$$F_T(x)=\frac{3-2x-\sqrt{1-4x}-\sqrt{2-16x+4x^2+(2+4x)\sqrt{1-4x}}}{2(1-\sqrt{1-4x})}\, .$$
\end{theorem}
\begin{proof}
We say that a permutation $\pi$ has {\em $(m,k)$ left-right maxima}, $1\leq k\leq m$, if it has $m$ left-right maxima $i_1,i_2,\dots,i_m$ of which the last $k$ are consecutive, that is,
$$i_1<i_2<\cdots<i_{m-k}<i_{m-k+1}=n-k+1<i_{m-k+2}=n-k+2<\cdots<i_{m-1}=n-k+1 < i_m=n,$$
where $n$ is maximal letter of $\pi$. Let $G_{m,k}(x)$ be the generating function for $T$-avoiders with $(m,k)$ left-right maxima. Define $G_{0,0}(x)=1$. To find an equation for $G_{m,k}(x), \ 1\le k \le m$,
let $\pi=i_1\pi^{(1)}\cdots i_m\pi^{(m)}$ be a permutation that avoids $T$ with $(m,k)$ left-right maxima. If $k=m$, then it is easy to see that $\pi^{(1)}>\pi^{(2)}>\cdots>\pi^{(m)}$, where each $\pi^{(j)}$ avoids $T$. Thus, $G_{m,m}(x)=(xF_T(x))^m$.

So suppose $1\leq k\leq m-1$.
Since $\pi$ avoids $1423$, all the letters in $I=\{i_{m-k}+1,\ldots,n-k\}$ appear in decreasing order in $\pi$ .
Since $\pi$ avoids $2413$, only left-right maxima can appear between letters that belong to $I$.
If $I=\emptyset$, then the contribution is given by $G_{m,k+1}(x)$.
Otherwise, there exists a largest $s\in[n-k+1,n]$
such that $\pi^{(s)}$ contains at least one letter from $I$.
By the preceding observations,
$$\pi^{(n-k+1)}\cdots\pi^{(s)}=(n-k)(n-k-1)\cdots(i_{m-k}+1){\pi'}^{(s)},$$
where $i_{m-k}>{\pi'}^{(s)}$.
We can now safely delete the
left-right maxima $n-k+2,n-k+3,\dots,s$ and all elements of $I$. The deleted left-right maxima contribute $x^{s-(n-k)-1}$, the deleted $i_{m-k}+1\in I$ (necessarily present) contributes $x$, and the other elements of $I$, which amount to distributing an arbitrary number of balls (possibly none) among the $s-(n-k)$
boxes $\pi^{(n-k+1)},\dots, \pi^{(s)}$, contribute $1/(1-x)^{s-(n-k)}$. After the deletion, we have a $T$-avoider with
$m-(s-n+k-1)$ left-right maxima of which the last $n-s+2$ are guaranteed consecutive, and so it contributes $G_{m+1-s+n-k,n+2-s}(x)$.
Hence, the contribution for given $s$ equals
$$\frac{x^{s-(n-k)}}{(1-x)^{s-(n-k)}}G_{m+1-s+n-k,n+2-s}(x)\, .$$
By summing over all $s=n-k+1,\ldots,n$, we see that the contribution for the case $I\neq\emptyset$ is given by
$$\sum_{j=1}^k\frac{x^j}{(1-x)^j}G_{m+1-j,k+2-j}(x)\, .$$
Combining all the contributions, we obtain for $1\le k < m$,
$$G_{m,k}(x)=G_{m,k+1}(x)+\frac{x}{1-x}\sum_{j=0}^{k-1}\frac{x^j}{(1-x)^j}G_{m-j,k+1-j}(x)\, ,$$
with $G_{m,m}(x)=(xF_T(x))^m$.

In order to determine an equation for $F_T(x)$, we define $G(t,u)=1+\sum_{m\geq1}\sum_{k=1}^m G_{m,k}(x)u^{k-1}t^m$.
By multiplying the above recurrence by $t^mu^{k-1}$ and summing over $k=1,2,\ldots,m-1$ and $m\geq1$, we find
$$G(t,u)=1+\frac{xF_T(x)}{1-tuxF_T(x)}+\frac{G(t,u)-G(t,0)}{u}+\frac{x(G(t,u)-G(t,0))}{u(1-x-xut)}\, .$$
Note that $G(1,0)=1+\sum_{m\geq0}G_{m,1}(x)=F_T(x)$. Hence,
$$G(1,u)=1+\frac{xF_T(x)}{1-uxF_T(x)}+\frac{G(1,u)-F_T(x)}{u}+\frac{x(G(1,u)-F_T(x))}{u(1-x-xu)}\, .$$
To solve this functional equation, we apply the kernel method and take $u=C(x)$, which is seen to cancel out $G(1,u)$. Thus,
$$0=1+\frac{xF_T(x)}{1-xC(x)F_T(x)}-\frac{F_T(x)}{C(x)}-\frac{xF_T(x))}{C(x)(1-x-xC(x))}\, ,$$
which, using the identity $C(x)=1+xC^2(x)$, is equivalent to
$$F_T(x)=1+\frac{xF_T(x)}{1-xC(x)F_T(x)}\, .$$
Solving this last equation completes the proof.
\end{proof}

\subsubsection{$\mathbf{T=\{2134,2143,2413\}}$}
\begin{theorem}\label{th238A2}
Let $T=\{2134,2143,2413\}$. Then
$$F_T(x)=\frac{3-2x-\sqrt{1-4x}-\sqrt{2-16x+4x^2+(2+4x)\sqrt{1-4x}}}{2(1-\sqrt{1-4x})}\, .$$
\end{theorem}
\begin{proof}
Let $G_m(x)$ be the generating function for $T$-avoiders with $m$ left-right maxima.
Clearly, $G_0(x)=1$ and $G_1(x)=xF_T(x)$. Now let us write an equation for $G_m(x)$.
If $\pi$ is a permutation that avoids $T$ with $m$ left-right maxima, then, to avoid 2134, $\pi$ has the form
 $$\pi=i_1i_2\cdots i_{m-1}\pi'  i_m\pi''$$
with $i_1<i_2<\cdots<i_m=n$ ($n$ is the maximal letter of $\pi$), $i_{m-1}>\pi'$, and $i_{m}>\pi''$.

If $\pi'$ is empty, then since $\pi$ avoids $2413$ we see that $\pi''$ can be decomposed as
$\pi_{m}'' \pi_{m-1}'' \cdots \pi_{1}''$, where $\pi_{j}''>i_{j-1}>\pi_{j-1}'',\ j=2,\dots,m,$ and $\pi_{j}''$ avoids $T$.

If $\pi'$ is not empty, then with $i_0=0$, there is a maximal integer $s$ such that $i_{s-1}<\pi'$.
Since $\pi$ avoids $2413$, we see that $\pi'=\pi_{m-1}' \cdots \pi_{s+1}'\pi_{s}'$ and
$\pi''=\pi_{s}'' \cdots \pi_{1}''$, where
\[
\pi_{m-1}'>i_{m-2}>\pi_{m-2}' > \cdots >i_{s+1}>\pi_{s+1}'>i_s >
\pi_{s}'\pi_{s}''>i_{s-1} > \pi_{s-1}''>\cdots >i_1>\pi_{1}''\, .
\]

This means that $\pi$ has the following diagrammatic shape.
\begin{center}
\begin{pspicture}(-2,0)(12,5.5)
\psset{xunit=.9cm,yunit=.5cm}
\pspolygon[fillstyle=solid,fillcolor=white](3.5,9)(4.5,9)(4.5,8)(3.5,8)(3.5,9)
\pspolygon[fillstyle=solid,fillcolor=white](6.5,6)(7.5,6)(7.5,5)(6.5,5)(6.5,6)
\pspolygon[fillstyle=solid,fillcolor=white](7.5,4)(7.5,5)(8.5,5)(8.5,4)(7.5,4)
\pspolygon[fillstyle=solid,fillcolor=white](8.5,4)(8.5,5)(9.5,5)(9.5,4)(8.5,4)
\pspolygon[fillstyle=solid,fillcolor=white](9.5,3)(9.5,4)(10.5,4)(10.5,3)(9.5,3)
\pspolygon[fillstyle=solid,fillcolor=white](12.5,0)(12.5,1)(13.5,1)(13.5,0)(12.5,0)
\psdots(0,1)(.5,3)(1,4)(1.5,5)(2,6)(3,8)(3.5,9)(8.5,10)
\psline[linecolor=gray](8.5,10)(8.5,5)
\psline[linecolor=gray](3,8)(3.5,8)
\psline[linecolor=gray](2,6)(6.5,6)
\psline[linecolor=gray](1.5,5)(6.5,5)
\psline[linecolor=gray](1,4)(7.5,4)
\psline[linecolor=gray](.5,3)(9.5,3)
\psline[linecolor=gray](0,1)(12.5,1)
\rput(4,8.5){\textrm{{\footnotesize $\pi_{m-1}'$}}}
\rput(7,5.5){\textrm{{\footnotesize $\pi_{s+1}'$}}}
\rput(8,4.5){\textrm{{\footnotesize $\pi_{s}'$}}}
\rput(9,4.5){\textrm{{\footnotesize $\pi_{s}''$}}}
\rput(10,3.5){\textrm{{\footnotesize $\pi_{s-1}''$}}}
\rput(13,0.5){\textrm{{\footnotesize $\pi_{1}''$}}}
\rput(-0.3,1.2){\textrm{{\scriptsize $i_1$}}}
\rput(0.0,3.2){\textrm{{\scriptsize $i_{s-2}$}}}
\rput(0.5,4.2){\textrm{{\scriptsize $i_{s-1}$}}}
\rput(1.2,5.2){\textrm{{\scriptsize $i_s$}}}
\rput(1.5,6.2){\textrm{{\scriptsize $i_{s+1}$}}}
\rput(2.3,8.2){\textrm{{\scriptsize $i_{m-2}$}}}
\rput(2.8,9.2){\textrm{{\scriptsize $i_{m-1}$}}}
\rput(8.7,10.5){\textrm{{\scriptsize $i_{m}=n$}}}
\rput(0.0,2){\textrm{{\footnotesize $\iddots$}}}
\rput(2.5,7){\textrm{{\footnotesize $\iddots$}}}
\rput(5.5,7){\textrm{{\footnotesize $\ddots$}}}
\rput(11.5,2){\textrm{{\footnotesize $\ddots$}}}
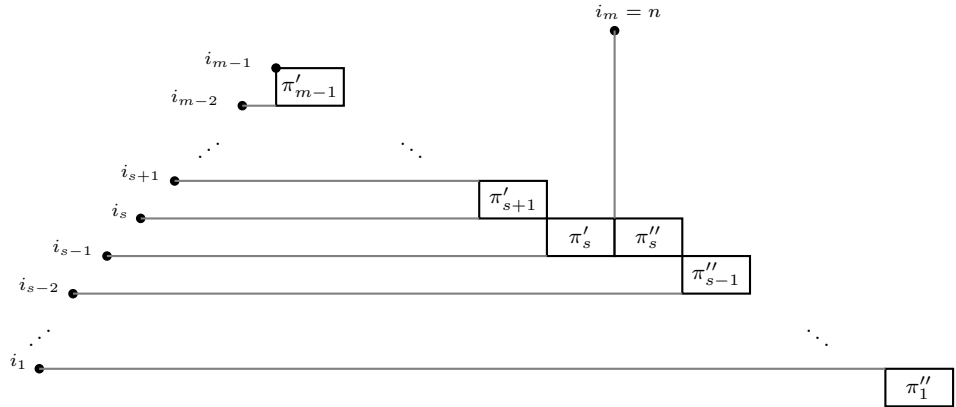
\rput(6,-0.5){\small{Decomposition of $T$-avoider, case $\pi'\ne \emptyset$}}
\end{pspicture}
\end{center}
Furthermore, $\pi_{j}'$ avoids $213$ for $j=m-1,m-2,\ldots,s+1$ for else $n$ is the 4 of a 3124; $\pi_{s}' n \pi_{s}''$ avoids $T$
and, since $\pi_{s}'$ is not empty, it does not start with its largest letter;  $\pi_{j}''$ avoids $T$ for $j=s-1,\ldots,1$.

Hence, the contribution in the case $\pi'$ is empty is $x^mF_t^m(x)$; otherwise, the contribution for given $s, \ 1\leq s\leq m$, is
$$x^{m-1}C^{m-1-s}(x)(F_T(x)-1-xF_T(x))F_T^{s-1}(x)\, .$$
Combining all the contributions, we obtain
\begin{align*}
F_T(x)&=1+\sum_{j\geq1}(x^jF_T^j(x))+(F_T(x)-1-xF_T(x))\sum_{m\geq2}\sum_{s=1}^{m-1}x^{m-1}C^{m-1-s}(x)F_T^{s-1}(x)\\
&=1+\sum_{j\geq1}(x^jF_T^j(x))+(F_T(x)-1-xF_T(x))\sum_{m\geq2}x^{m-1}\frac{C^{m-1}(x)-F_T^{m-1}(x)}{C(x)-F_T(x)},
\end{align*}
and, using $C(x)=1+xC^2(x)$, we find that
$$F_T(x)=1-x^2C^2(x)F_T(x)+xC(x)F_T^2(x),$$
which yields the stated \gf.
\end{proof}

For the remaining three cases, we consider (right-left) cell decompositions. So suppose
$$\pi=\pi^{(m)}i_m\pi^{(m-1)}i_{m-1}\cdots\pi^{(1)}i_1\in S_n$$
has $m\geq2$ right-left maxima $n=i_m>i_{m-1}>\cdots >i_1\geq1$. The right-left maxima determine a
{\em cell decomposition} of the matrix diagram of $\pi$ as illustrated in the figure below for $m=4$.
There are $\binom{m+1}{2}$ cells $C_{ij},\ i,j \ge 1,\,i+j \le m+1,$ indexed by $(x,y)$ coordinates, for example, $C_{21}$ and $C_{32}$ are shown.
\begin{center}
\begin{pspicture}(-2,-.8)(4,3)
\psset{unit=.7cm}
\psline(0,4)(0,0)(4,0)
\psline(0,4)(1,4)(1,0)
\psline(0,3)(2,3)(2,0)
\psline(0,2)(3,2)(3,0)
\psline(0,1)(4,1)(4,0)
\rput(1.5,0.5){\textrm{{\footnotesize $C_{21}$}}}
\rput(2.5,1.5){\textrm{{\footnotesize $C_{32}$}}}
\rput(1.3,4.2){\textrm{\footnotesize $i_4$}}
\rput(2.3,3.2){\textrm{\footnotesize $i_3$}}
\rput(3.3,2.2){\textrm{\footnotesize $i_2$}}
\rput(4.3,1.2){\textrm{\footnotesize $i_1$}}
\rput(2,-.8){\textrm{Cell decomposition}}
\end{pspicture}
\end{center}
Cells with $i=1$ or $j=1$ are \emph{boundary} cells, the others are \emph{interior}. A cell is \emph{occupied} if it contains at least one letter of $\pi$, otherwise it is \emph{empty}. Let $\al_{ij}$ denote the subpermutation of entries in $C_{ij}$.

We now consider $R=\{1342,1423\}$, a subset of the pattern set in the remaining three cases. The reader may check the following
characterization of $R$-avoiders in terms of the cell decomposition.
A permutation $\pi$ is an $R$-avoider if and only if
\vspace*{-1mm}
\begin{enumerate}
\item For each occupied cell $C$, all cells that lie both strictly east and strictly north of $C$ are empty.
\item For each pair of occupied cells $C,D$ with $D$ directly north of $C$ (same column), all entries in $C$ lie to the right of all entries in $D$.
\item For each pair of occupied cells $C,D$ with $D$ directly east of $C$ (same row), all entries in $C$ are larger than all entries in $D$.
\item $\al_{ij}$ avoids $R$ for all $i,j$.
\end{enumerate}
\vspace*{-1mm}
Condition (1)  imposes restrictions on occupied cells as follows. A \emph{major} cell for $\pi$ is an interior cell $C$ that is occupied and such that all cells directly north or directly east of $C$ are empty. The set of major cells (possibly empty) determines a (rotated) Dyck path of semilength $m-1$ with valley vertices at the major cells as illustrated in the figure below. (If there are no major cells, the Dyck path covers the boundary cells and has no valleys.)
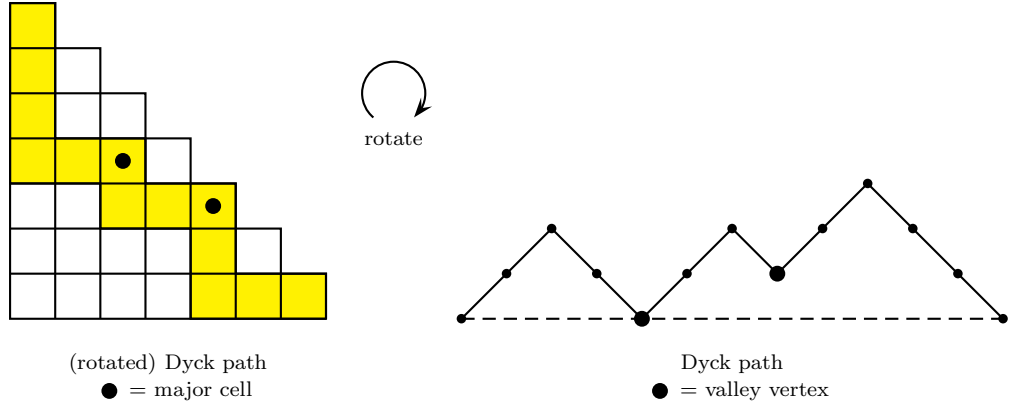
\begin{center}
\begin{pspicture}(-8.5,-1.2)(12,4)
\psset{unit=.6cm}
\pspolygon[fillstyle=solid,fillcolor=yellow](-10,7)(-10,3)(-8,3)(-8,2)(-6,2)(-6,0)(-3,0)(-3,1)(-5,1)(-5,3)(-7,3)(-7,4)(-9,4)(-9,7)(-10,7)
\psline(-10,7)(-10,0)(-3,0)
\psline(-10,7)(-9,7)(-9,0)
\psline(-10,6)(-8,6)(-8,0)
\psline(-10,5)(-7,5)(-7,0)
\psline(-10,4)(-6,4)(-6,0)
\psline(-10,3)(-5,3)(-5,0)
\psline(-10,2)(-4,2)(-4,0)
\psline(-10,1)(-3,1)(-3,0)
\qdisk(-7.5,3.5){3pt}
\qdisk(-5.5,2.5){3pt}
\psdots(0,0)(1,1)(2,2)(3,1)(4,0)(5,1)(6,2)(7,1)(8,2)(9,3)(10,2)(11,1)(12,0)
\psline(0,0)(2,2)(4,0)(6,2)(7,1)(9,3)(12,0)
\psline[linestyle=dashed](0,0)(12,0)
\qdisk(4,0){3pt}
\qdisk(7,1){3pt}
\psarcn[fillcolor=white,arrows=->,arrowsize=3pt 3](-1.5,5){.7}{230}{310}
\rput(-1.5,4){\textrm{\footnotesize rotate}}
\rput(-6.5,-1){\textrm{\footnotesize (rotated) Dyck path}}
\rput(-6,-1.6){\textrm{\footnotesize = major cell}}
\qdisk(-7.8,-1.6){3pt}
\rput(6,-1){\textrm{\footnotesize Dyck path}}
\rput(6.5,-1.6){\textrm{\footnotesize = valley vertex}}
\qdisk(4.4,-1.6){3pt}
\end{pspicture}
\end{center}
If $\pi$ avoids $R$, then condition (1) implies that all cells not on the Dyck path are empty, and condition (4) implies St($\al_{ij}$) is an $R$-avoider for all $i,j$. Conversely, if $n=i_m>i_{m-1}> \dots > i_1 \ge 1$ are given and we have a Dyck path in the associated cell diagram,
and an $R$-avoider $\pi_C$ is specified for each cell $C$ on the Dyck path, with the additional proviso $\pi_C \ne \emptyset$ for valley cells, then conditions (2) and (3) imply that an $R$-avoider with this Dyck path is uniquely determined.

It follows that an $R$-avoider $\pi$ avoids the pattern $\tau k$ where $\tau \in S_{k-1}$ if and only if all the subpermutations $\al_{ij}$ avoid $R$ and $\tau$. We use this observation in the next two results.
As an immediate consequence, we have
\begin{proposition}\label{prop238A3}
Let $\tau$ and $\tau'$ be two patterns in $S_{k-1}$.
If $F_{\{1342,1423,\tau\}}(x)=F_{\{1342,1423,\tau'\}}(x)$, then $F_{\{1342,1423,\tau k\}}(x)=F_{\{1342,1423,\tau'k\}}(x)$. \qed
\end{proposition}

We can now find a recurrence for avoiders of the pattern set $R \cup \{12 \cdots k\}$.

\begin{proposition}\label{prop238A4}
Let $T_k=\{1342,1423,12\cdots k\}$. Then
$$F_{T_k}(x)=\frac{1+(x-2)F_{T_{k-1}}(x)+\sqrt{\big(1+xF_{T_{k-1}}(x)\big)^2-4xF_{T_{k-1}}^2(x)}}{2\big(1-F_{T_{k-1}}(x)\big)}.$$
\end{proposition}
\begin{proof}
For brevity, set $F_{k} =F_{T_{k}}(x)$. So, for $m$ right-left maxima and an associated Dyck path of semilength $m-1$, the contribution to $F_k$ is $x^m$ for the right-left maxima, $F_{k-1} -1$ for each valley vertex, and $F_{k-1}$ for every other vertex. Let $\ell$ denote the number of peaks in the Dyck path, so that $\ell-1$ is the number of valleys. Recall that the Narayana number $N_{m,\ell}=\frac{1}{m}\binom{m}{\ell}\binom{m}{\ell-1}$ counts Dyck paths of semilength $m$ with $\ell$ peaks. Hence, summing over $m$,
\begin{eqnarray*}
F_{k} & = & 1+xF_{k-1}+\sum_{m\geq2}x^m \sum_{\ell=1}^{m-1}N_{m-1,\ell}\:(F_{k-1}-1)^{\ell-1} F_{k-1}^{2m-\ell} \\
& = & 1+xF_{k-1} +\frac{x F_{k-1}^{2}}{F_{k-1}-1} \sum_{m\ge 1}\sum_{\ell=1}^{m}N_{m,\ell}\:\big(xF_{k-1}^{2}\big)^{m}\left(1-\frac{1}{F_{k-1}}\right)^{\ell} \\
& = &  1+xF_{k-1} +\frac{x F_{k-1}^{2}}{F_{k-1}-1} N\big(xF_{k-1}^{2},1-1/F_{k-1}\big)\, ,
\end{eqnarray*}
where $
N(x,y):=\sum_{m\geq1}\sum_{\ell=1}^m N_{m,\ell}x^m y^\ell$
is the \gf the Narayana numbers.
It is known that
\[
N(x,y)=\frac{1-x(1+y)-\sqrt{(1-x(1+y))^2-4yx^2}}{2x}
\]
and the theorem follows.
\end{proof}

\subsubsection{$\mathbf{T=\{1234,1342,1423\}}$}
\begin{theorem}\label{th238A3}
Let $T=\{1234,1342,1423\}$. Then
$$F_T(x)=\frac{3-2x-\sqrt{1-4x}-\sqrt{2-16x+4x^2+(2+4x)\sqrt{1-4x}}}{2(1-\sqrt{1-4x})}\, .$$
\end{theorem}
\begin{proof}
Since $F_{\{1342,1423,123\}}(x)=F_{\{123\}}(x)=C(x)$,
we get by Prop. \ref{prop238A4} that
$$F_T(x)=1+xC(x)+\frac{xC^2(x)}{C(x)-1}N\big(xC^2(x),1-1/C(x)\big),$$
which, after some algebraic manipulation, agrees with the desired expression.
\end{proof}

\subsubsection{$\mathbf{T=\{1324,1342,1423\}}$}
\begin{theorem}\label{th238A4}
Let $T=\{1324,1342,1423\}$. Then
$$F_T(x)=\frac{3-2x-\sqrt{1-4x}-\sqrt{2-16x+4x^2+(2+4x)\sqrt{1-4x}}}{2(1-\sqrt{1-4x})}\, .$$
\end{theorem}
\begin{proof}
Since $F_{\{1342,1423,132\}}(x)=F_{\{132\}}(x)=C(x)$ and $F_{\{1342,1423,123\}}(x)=F_{\{123\}}(x)=C(x)$, we get by Prop. \ref{prop238A3} with $\tau=132$ and $\tau'=123$ that
$F_{\{1342,1423,1324\}}(x)=F_{\{1342,1423,1234\}}(x)$.
Apply Theorem \ref{th238A3}.
\end{proof}

\subsubsection{$\mathbf{T=\{1243,1342,1423\}}$}
\begin{theorem}\label{th238A5}
Let $T=\{1243,1342,1423\}$. Then
$$F_T(x)=\frac{3-2x-\sqrt{1-4x}-\sqrt{2-16x+4x^2+(2+4x)\sqrt{1-4x}}}{2(1-\sqrt{1-4x})}\, .$$
\end{theorem}
\begin{proof}

A permutation $\pi \in S_T(n)$ with $m\ge 2$ right-left maxima avoids $R$ and so the cell decomposition of $\pi$ has an associated Dyck path that covers all occupied cells. To also avoid 1243, all the Dyck path cells except the cells incident with a right-left maximum, that is, cells $C_{ij}$ with $i+j=m+1$, must avoid 12 for else some two right-left maxima would form the 43 of a 1243. Other cells need only avoid 1243. The cells $C_{ij}$ with $i+j=m+1$ consist of the extremities $C_{1m}$ and $C_{m1}$ together with all the low valleys in the Dyck path (a low valley is one incident with ground level, the line joining the path's endpoints). Suppose the Dyck path has $\ell$ low valleys and $h$ high valleys. The contribution of the right-left maxima is $x^m$. Since $F_{\{12\}}(x)=1/(1-x)$, the contributions of the $2m-1$ Dyck path cells are as follows. The two extremities contribute $F_T^2(x)$, the $\ell$ low valleys contribute $(F_T(x)-1)^{\ell}$, the $h$ high valleys contribute $\big(\frac{1}{1-x}-1\big)^h=\big(\frac{x}{1-x}\big)^h$, and the remaining cells contribute $\big(\frac{1}{1-x}\big)^{2m-3-\ell-h}$.

Let $M_{m,\ell,h}$ denote the number of Dyck paths of semilength $m$ containing $\ell$ low valleys and $h$ high valleys,
with \gf $M(x,y,z)=\sum_{m,\ell,h\geq0}M_{m,\ell,h}x^my^\ell z^h$. Then, by the first return decomposition of the Dyck paths, we obtain
$$M(x,1,z)=1+xM(x,1,z)+xzM(x,1,z)(M(x,1,z)-1)$$
and
$$M(x,y,z)=1+xM(x,1,z)+xyM(x,1,z)(M(x,y,z)-1)\, .$$
Thus,
$$M(x,y,z)=\frac{y-2z-1+x(1-y)(1-z)+(1-y)\sqrt{1-2x(1+z)+x^2(1-z)^2}}{(1-x)y+(xy-2)z-y\sqrt{1-2x(1+z)+x^2(1-z)^2}}\, .$$
Hence, summing over $m$ and over all Dyck paths gives
\[
F_T(x)=1 +xF_T(x)+\sum_{m\ge 2}\: \sum_{\ell,h \ge 0}M_{m-1,\ell,h}x^{m+h}F_T^2(x)\big(F_T(x)-1\big)^{\ell} \frac{1}{(1-x)^{2m-3-\ell}}.
\]
After several algebraic steps and solving for $F_T(x)$, one obtains the desired formula.
\end{proof}

The preceding theorem can be extended to the case
$T_k=\{1342,1423,\tau k(k-1)\}$ with $k\geq4$ as follows.
\begin{theorem}\label{th238A5A12k}
Let $k\geq4$ and $\tau\in S_{k-2}$. Let $T_k=\{1342,1423,
\tau k(k-1)\}$ and $T'_k=\{1342,1423,\tau\}$. Then
$$F_{T_k}(x)=\frac{(2-x)(1-t)-x^2F_{T'_k}^2(x)(1+(x-2)F_{T'_k}(x))+
\sqrt{2x(a-bt)}}{2(1-xF_{T'_k}^2(x)+x^2F_{T'_k}^3(x)-t)},$$
where
\begin{align*}
t&=\sqrt{(1-xF_{T'_k}^2(x))^2-x^2F_{T'_k}^3(x)(2-2xF_{T'_k}^2(x)+x^2F_{T'_k}^3(x))},\\
a&=(x-4)(1+x^4F_{T'_k}^6(x))+2xF_{T'_k}^2(x)(1+(1-x)F_{T'_k}^2(x)+x^2F_{T'_k}^3(x))+x^3F_{T'_k}^4(x),\\
b&=(4-x)(1+x^2F_{T'_k}^3(x))+x(2-x)F_{T'_k}^2(x).
\end{align*}
\end{theorem}
\begin{proof}
The proof follows the same lines as in the preceding theorem except that $F_T(x)$ is replaced by $F_{T_{k}}(x)$ and $F_{\{12\}}(x)$
is replaced by $F_{T'_{k}}(x)$. The details are left to the reader.
\end{proof}


\section*{Appendix}
\footnotesize\begin{longtable}[c]{|l|l|l|}
\caption{Sequences $\{|S_n(T)|\}_{n=5}^{16}$, where $T$ is one of the triples in a symmetry class, arranged in lex (increasing) order of counting sequence.\label{long3s}}\\ \hline
\multicolumn{3}{| c |}{Table 2}\\ \hline
No. &$T$ & $\{|S_n(T)|\}_{n=5}^{16}$ \\ \hline
\endfirsthead  \hline
\multicolumn{3}{|c|}{Continuation of Table \ref{long3s}}\\ \hline
No. &$T$ & $\{|S_n(T)|\}_{n=5}^{16}$ \\ \hline
\endhead \hline
\endfoot \hline
\multicolumn{3}{| c |}{End of Table}\\ \hline\hline
\endlastfoot
1&\{4321,3412,1234\}&   69,162,240,199,73,0,0,0,0,0,0,0\\\hline
2&\{4321,3142,1234\}&   69,164,252,221,85,0,0,0,0,0,0,0\\\hline
3&\{2143,4312,1234\}&   69,181,375,651,1009,1449,1971,2575,3261,4029,4879,5811\\\hline
4&\{4231,2143,1234\}&   69,190,446,927,1745,3036,4960,7701,11467,16490,23026,31355\\\hline
5&\{2143,3412,1234\}&   69,194,470,1009,1969,3562,6062,9813,15237,22842,33230,47105\\\hline
6&\{2134,3412,1432\}&   \\
&\{3412,1432,1234\}&    69,198,498,1121,2305,4402,7910,13509,22101,34854,53250,79137\\\hline
7&\{3421,4312,1234\}&   70,177,333,538,792,1095,1447,1848,2298,2797,3345,3942\\\hline
8&\{2431,4213,1234\}&   70,192,441,929,1870,3670,7097,13600,25907,49142,92911,175190\\\hline
9&\{2134,4312,1243\}&   70,195,458,942,1752,3016,4886,7539,11178,16033,22362,30452\\\hline
10&\{4213,1432,1234\}&  70,199,502,1232,2962,6970,16138,36982,84083,189918,426722,954884\\\hline 11&\{4231,1432,1234\}& 70,200,481,1004,1886,3270,5325,8246,12254,17596,24545,33400\\\hline
12&\{2341,4312,1324\}&  70,203,517,1187,2504,4921,9107,16009,26922,43567,68177,103591\\\hline
13&\{3214,1432,1234\}&     70,204,560,1617,4796,14249,41939,122658,358991,1053628,3095381,9089525\\\hline 14&\{4231,2134,1243\}&   70,205,536,1264,2722,5424,10122,17871,30102,48703,76108,115394\\\hline 15&\{2134,3412,1243\}&   70,207,541,1272,2747,5552,10672,19783,35804,63965,113903,203810\\\hline 16&\{2314,1432,4123\}&  70,210,589,1592,4218,11069,28932,75528,197165,514920,1345484,3517427\\\hline 17&\{2341,2143,4123\}&     70,212,597,1610,4248,11107,28966,75552,197251,515476,1348060,3527067\\\hline 18&\{2341,1432,4123\}&     70,212,611,1712,4712,12815,34576,92764,247819,659840,1752170,4642567\\\hline 19&\{2431,4312,1234\}&     71,200,465,929,1667,2766,4325,6455,9279,12932,17561,23325\\\hline
20&\{4312,1432,1234\}&  71,204,479,951,1687,2764,4269,6299,8961,12372,16659,21959\\\hline
21&\{4312,3142,1234\}&  71,208,526,1174,2370,4416,7714,12783,20277,31004,45946,66280\\\hline
22&\{2134,4312,1432\}&  71,209,533,1205,2473,4696,8372,14169,22959,35855,54251,79865\\\hline
23&\{2431,4132,1234\}&     71,209,545,1348,3270,7908,19201,46918,115407,285642,711031,1779289\\\hline 24&\{4231,3412,1234\}&   71,212,554,1289,2725,5326,9758,16941,28107,44864,69266,103889\\\hline 25&\{3412,4132,1234\}&    71,213,561,1317,2809,5536,10220,17865,29823,47867,74271,111897\\\hline 26&\{2134,4312,1342\}&   71,213,564,1340,2909,5860,11090,19911,34179,56447,90144,139782\\\hline 27&\{2314,4312,1432\}&   71,213,569,1389,3175,6927,14632,30238,61596,124335,249598,499492\\\hline 28&\{4231,3142,1234\}&     71,217,599,1514,3550,7801,16193,31956,60282,109214,190816,322679\\\hline 29&\{2143,4312,1324\}&     71,218,610,1585,3895,9186,21022,47061,103663,225618,486626,1042305\\\hline 30&\{4231,3412,1324\}&   71,218,614,1619,4065,9840,23168,53393,120995,270518,598218,1310943\\\hline 31&\{2314,4312,1342\}&   71,219,626,1698,4452,11428,28966,72907,182915,458590,1150877,2894324\\\hline 32&\{2134,1432,4123\}&     71,219,635,1776,4853,13068,34862,92438,244118,642947,1690256,4437947\\\hline 33&\{2134,3412,4132\}&     71,220,630,1697,4365,10842,26216,62071,144519,331928,753834,1695933\\\hline 34&\{2143,4132,1234\}&  71,220,646,1835,5095,13924,37627,100859,268756,713023,1885543,4974068\\\hline 35&\{2143,3412,1324\}&    71,222,648,1797,4807,12548,32236,82009,207529,524060,1323540,3348087\\\hline 36&\{3412,3124,1432\}&     71,222,652,1838,5053,13682,36697,97814,259585,686709,1812257,4773804\\\hline 37&\{3142,1432,1234\}&     71,229,726,2299,7296,23180,73648,233935,742924,2359143,7491146,23786672\\\hline 38&\{4321,1423,1234\}&  72,198,367,359,147,0,0,0,0,0,0,0\\\hline
39&\{4321,4123,1234\}&  72,205,396,400,185,0,0,0,0,0,0,0\\\hline
40&\{2341,4312,1234\}&  72,216,555,1252,2549,4787,8428,14079,22518,34722,51897,75510\\\hline
41&\{4312,1342,1234\}&  72,220,590,1409,3055,6118,11474,20373,34542,56304,88714,135713\\\hline 42&\{2341,4132,1234\}&   72,221,605,1517,3574,8065,17671,37953,80424,168885,352481,732581\\\hline 43&\{2314,4213,1432\}&     72,228,670,1864,5000,13099,33789,86239,218432,550107,1379348,3446817\\\hline 44&\{4213,1342,1234\}&     72,228,678,1929,5307,14203,37133,95179,239942,596587,1466529,3571386\\\hline 45&\{4213,2134,1432\}&     72,229,683,1954,5452,14974,40671,109509,292743,777810,2055833,5409187\\\hline 46&\{2341,4132,1324\}&    72,229,686,1972,5514,15131,40986,110013,293376,778678,2059646,5434009\\\hline 47&\{2413,4132,1234\}&    72,230,689,1970,5460,14833,39790,105890,280367,739878,1948186,5121973\\\hline 48&\{4312,3124,1342\}&    72,230,692,2004,5683,15948,44523,123924,344113,953353,2635064,7266192\\\hline 49&\{2341,1324,4123\}&    72,230,701,2113,6475,20468,66969,226027,782276,2760094,9880455,35758457\\\hline 50&\{2143,3412,4123\}& \\
&\{3412,1432,1324\}&    \\
&\{3412,1432,1243\}&    72,232,707,2066,5858,16257,44428,120076,321919,857942,2276454,6020541\\\hline 51&\{4213,3124,1432\}&    72,232,712,2116,6155,17629,49893,139851,388899,1074280,2950885,8066698\\\hline 52&\{1432,4123,1234\}&   72,232,717,2157,6370,18557,53490,152868,433781,1223511,3433182,9590277\\\hline 53&\{2134,4132,1243\}&   72,233,719,2146,6260,17968,50967,143278,399960,1110203,3067479,8442903\\\hline 54&\{3124,1432,1234\}&   72,233,739,2343,7458,23801,76016,242777,775265,2475513,7904587,25240597\\\hline 55&\{4213,3124,1342\}&  \\
&\{2143,1324,4123\}&    \\
&\{2143,4123,1234\}&    72,236,745,2286,6866,20285,59156,170712,488401,1387226,3916062,10996581\\\hline 56&\{2143,1342,4123\}&  \\
&\{3412,1432,4123\}&    72,237,761,2415,7626,24034,75689,238298,750179,2361533,7433917,23401274\\\hline 57&\{2143,1432,1234\}&  72,246,845,2901,9955,34165,117254,402409,1381046,4739681,16266344,55825262\\\hline 58&\{4321,1243,1234\}&   73,202,382,396,144,0,0,0,0,0,0,0\\\hline
59&\{4321,1324,1234\}&  73,215,484,669,334,0,0,0,0,0,0,0\\\hline
60&\{4312,4132,1234\}&  73,222,563,1226,2376,4213,6972,10923,16371,23656,33153,45272\\\hline
61&\{4312,1243,1234\}&  73,223,587,1356,2820,5395,9653,16355,26487,41299,62347,91538\\\hline
62&\{4231,4312,1234\}&     73,228,616,1460,3110,6082,11102,19155,31539,49924,76416,113626\\\hline 63&\{4312,1324,1234\}&   73,229,629,1521,3304,6578,12201,21353,35607,57007,88153,132293\\\hline 64&\{4312,3412,1234\}&   73,229,634,1562,3481,7132,13622,24531,42033,69031,109306,167680\\\hline 65&\{4213,4132,1234\}&  73,231,650,1668,3987,9030,19628,41333,84915,171087,339408,665004\\\hline 66&\{4231,4132,1234\}&     73,232,654,1639,3705,7678,14798,26841,46257,76324,121318,186699\\\hline 67&\{4312,1324,1243\}&  73,233,677,1819,4606,11171,26274,60471,137059,307245,683171,1509595\\\hline 68&\{4312,1342,1243\}&  73,234,691,1910,5019,12690,31147,74694,175843,407810,934179,2117958\\\hline 69&\{3412,1324,1234\}&  73,236,700,1919,4927,12006,28090,63705,141109,307088,659576,1402947\\\hline 70&\{4312,3124,1243\}&  73,236,705,1970,5224,13307,32866,79251,187523,437030,1005935,2291536\\\hline 71&\{4231,1243,1234\}&     73,237,702,1881,4577,10216,21158,41097,75561,132523,223134,362589\\\hline 72&\{3412,1324,1243\}&    73,237,711,1988,5253,13301,32673,78669,187230,443398,1050209,2497187\\\hline 73&\{4231,1324,1234\}&     73,238,714,1962,4957,11604,25390,52361,102533,191868,344970,598682\\\hline 74&\{3412,1243,1234\}&   73,238,718,2013,5301,13266,31886,74269,168841,376750,828726,1802901\\\hline 75&\{4231,1324,1243\}&  73,238,721,2044,5492,14178,35610,87938,215295,525787,1286294,3160692\\\hline 76&\{3412,1324,4123\}&     73,238,721,2046,5501,14158,35172,84895,200133,462714,1052727,2363200\\\hline 77&\{3412,3124,1243\}&     73,238,722,2054,5541,14323,35788,87043,207201,484772,1118334,2550164\\\hline 78&\{4312,1342,1324\}&    \\
&\{3412,3142,1234\}&    73,238,724,2075,5667,14892,37942,94273,229453,548954,1294440,3014775\\\hline 79&\{2134,4132,1234\}& 73,238,724,2078,5706,15161,39319,100168,251846,627046,1549898,3810125\\\hline 80&\{4312,1324,4123\}&    73,238,726,2101,5857,15926,42626,112997,297861,782666,2052958,5379953\\\hline 81&\{2431,4312,1324\}&    73,239,734,2131,5900,15697,40389,101047,246915,591507,1393504,3236723\\\hline 82&\{4312,3142,1243\}&    73,239,734,2133,5924,15859,41202,104433,259312,632815,1521646,3612653\\\hline 83&\{4312,3412,1243\}&    73,239,734,2134,5934,15918,41470,105470,262910,644350,1556478,3713022\\\hline 84&\{4231,1324,4123\}&    73,239,736,2158,6102,16813,45493,121567,322108,848654,2227722,5834253\\\hline 85&\{2314,4132,1432\}&    73,239,738,2178,6220,17351,47595,128985,346492,924788,2456502,6502017\\\hline 86&\{3412,4132,1324\}&    73,239,740,2194,6298,17653,48621,132199,356040,952154,2533014,6712221\\\hline 87&\{4312,3124,1432\}&    73,239,740,2199,6348,17947,49954,137372,374164,1011303,2716439,7259970\\\hline 88&\{4312,3412,1324\}&   73,240,744,2192,6192,16896,44800,115968,294144,733184,1800192,4362240\\\hline 89&\{3142,4132,1234\}&    73,240,746,2217,6371,17864,49202,133759,360175,963044,2561604,6787167\\\hline 90&\{3412,4132,1243\}&    73,240,748,2240,6525,18653,52640,147210,408957,1130398,3112172,8540753\\\hline 91&\{4213,1342,1243\}&   73,240,754,2309,6987,21036,63202,189723,569311,1708100,5124492,15373695\\\hline 92&\{2314,3124,1432\}&  73,240,759,2365,7369,23069,72495,228186,718341,2260566,7111650,22370236\\\hline 93&\{4312,3142,1324\}&  73,241,754,2252,6471,18003,48736,128878,333949,850061,2130078,5263536\\\hline 94&\{2134,4132,1432\}&    \\
&\{4132,1432,1234\}&    73,241,756,2276,6640,18915,52911,145951,398242,1077434,2895486,7740081\\\hline 95&\{2314,4132,1342\}&   73,241,757,2288,6724,19365,54959,154303,429733,1189430,3276306,8990037\\\hline 96&\{2134,4132,1342\}&   73,241,759,2305,6806,19652,55725,155688,429719,1174344,3183298,8571979\\\hline 97&\{2341,4312,4123\}&   73,241,766,2399,7514,23648,74706,236352,747770,2364773,7475960,23631523\\\hline 98&\{2134,3124,1432\}&  73,241,768,2415,7587,23905,75507,238759,755088,2387570,7548085,23860518\\\hline 99&\{4231,3142,1324\}&  73,242,762,2290,6610,18434,49922,131842,340738,864258,2156546,5304322\\\hline 100&\{4312,1342,4123\}&   73,242,772,2409,7439,22872,70204,215345,660375,2024866,6208416,19035179\\\hline 101&\{3124,4132,1342\}& 73,243,777,2408,7288,21661,63471,183877,527761,1503086,4252938,11966373\\\hline 102&\{2413,3142,1234\}& 73,243,785,2504,7968,25389,81033,258873,827263,2643616,8447300,26990489\\\hline 103&\{2314,1342,4123\}& 73,243,785,2511,8073,26312,87257,294603,1011602,3526519,12456315,44495535\\\hline 104&\{2134,4132,1423\}&   73,244,782,2415,7232,21122,60455,170228,473014,1300271,3543000,9584730\\\hline 105&\{4213,2134,1342\}&  73,244,787,2468,7570,22809,67727,198664,576775,1659914,4741254,13454541\\\hline 106&\{2143,3412,1423\}& 73,244,790,2505,7839,24320,74998,230243,704359,2148620,6538740,19859175\\\hline 107&\{4213,3412,1342\}& 73,244,794,2553,8179,26192,83906,268883,861815,2762484,8855204,28385839\\\hline 108&\{3124,1432,4123\}& \\
&\{1432,1324,4123\}&    73,245,795,2508,7732,23393,69687,204939,596215,1718714,4915914,13966077\\\hline 109&\{2143,3412,1243\}& 73,245,797,2530,7878,24153,73109,218929,649609,1912298,5590446,16243437\\\hline 110&\{2134,3142,1432\}& 73,245,804,2617,8511,27709,90283,294231,958826,3124175,10178664,33160777\\\hline 111&\{2143,3142,1234\}&    73,247,821,2704,8868,29030,94960,310531,1015359,3319829,10854379,35488838\\\hline 112&\{2134,1432,1243\}&   \\
&\{3142,1432,4123\}&    73,250,853,2911,9938,33931,115849,395534,1350437,4610679,15741842,53746011\\\hline 113&\{2134,1432,1234\}&  73,250,861,2967,10220,35203,121263,417710,1438865,4956391,17073052,58810751\\\hline 114&\{4312,1423,1234\}& 74,237,668,1667,3750,7743,14898,27033,46698,77369,123672,191639\\\hline 115&\{4231,1423,1234\}& 74,245,744,2068,5296,12608,28150,59412,119341,229477,424478,758491\\\hline 116&\{4312,4123,1243\}&  74,246,763,2227,6191,16567,43026,109110,271384,664236,1603813,3827381\\\hline 117&\{3124,4132,1234\}&   74,247,769,2247,6238,16649,43132,109257,272073,668704,1626916,3926643\\\hline 118&\{3412,1423,1234\}&   74,248,780,2309,6483,17407,45028,112921,275964,660030,1550320,3586449\\\hline 119&\{4312,1432,1324\}&   74,248,784,2355,6785,18897,51177,135358,350788,893038,2237998,5530485\\\hline 120&\{4132,1423,1234\}&   74,248,787,2389,7013,20079,56417,156250,427914,1161571,3130892,8391305\\\hline 121&\{3412,4123,1243\}&  74,249,792,2394,6941,19479,53323,143275,379721,996456,2596798,6735913\\\hline 122&\{4213,1432,1324\}&   74,249,797,2451,7318,21380,61449,174378,489827,1364499,3774779,10381722\\\hline 123&\{4132,1342,1234\}& 74,249,798,2459,7351,21457,61434,173120,481461,1324409,3610321,9768290\\\hline 124&\{2341,4132,4123\}&  74,249,804,2540,7977,25106,79327,251328,797094,2527977,8014590,25401277\\\hline 125&\{2341,4123,1243\}& 74,249,804,2551,8139,26500,88531,303112,1059129,3759584,13505901,48965424\\\hline 126&\{2431,4213,1324\}&   \\
&\{3142,4123,1234\}&74,251,817,2570,7872,23621,69746,203317,586561,1677746,4764474,13447701\\\hline 127&\{2413,4312,1342\}& \\
&\{3142,4123,1243\}&    74,253,840,2728,8719,27541,86221,268047,828661,2550116,7818174,23893803\\\hline 128&\{2341,3142,4123\}& 74,253,843,2772,9080,29759,97686,321033,1055596,3471365,11415280,37535830\\\hline 129&\{2413,3124,1432\}&   \\
&\{2143,3124,1432\}&    74,253,845,2791,9188,30246,99639,328422,1082797,3570197,11771589,38812310\\\hline 130&\{3412,3124,1342\}&   74,254,856,2867,9614,32368,109432,371221,1262278,4298922,14655296,49991239\\\hline 131&\{2134,1342,4123\}&  74,254,858,2889,9775,33371,115135,401538,1414821,5032193,18049762,65227297\\\hline 132&\{3142,1324,4123\}&  74,255,857,2815,9063,28677,89389,275034,836689,2520128,7524372,22291317\\\hline 133&\{2143,3124,1423\}& 74,255,863,2891,9638,32068,106627,354480,1178459,3917863,13025489,43305654\\\hline 134&\{2134,4123,1243\}&  74,255,866,2927,9923,33898,116940,407638,1435430,5102259,18290193,66060912\\\hline 135&\{1432,4123,1243\}&  74,256,876,2987,10182,34726,118492,404441,1380670,4713644,16093028,54944551\\\hline 136&\{4213,1342,4123\}& 74,256,880,3025,10406,35805,123197,423881,1458425,5017929,17264954,59402739\\\hline 137&\{3124,1432,1342\}& 74,257,881,2995,10132,34182,115143,387538,1303745,4384933,14746009,49585430\\\hline 138&\{2134,3142,1243\}& 74,257,883,3015,10258,34826,118075,399978,1354163,4582981,15506825,52460374\\\hline 139&\{2143,3124,1342\}& 74,257,883,3019,10306,35174,120063,409878,1399367,4777721,16312213,55693546\\\hline 140&\{3124,1432,1243\}& 74,257,886,3050,10505,36206,124833,430474,1484526,5119597,17655746,60888801\\\hline 141&\{2143,1423,1234\}& 74,259,905,3163,11058,38664,135193,472724,1652965,5779907,20210571,70670238\\\hline 142&\{1432,1342,4123\}& 74,260,913,3206,11258,39533,138822,487480,1711809,6011098,21108254,74122629\\\hline 143&\{4312,4123,1234\}& 75,248,735,1952,4697,10378,21320,41163,75363,131808,221561,359742\\\hline 144&\{4231,4123,1234\}&   75,253,774,2130,5314,12169,25895,51756,98034,177282,307933,516327\\\hline 145&\{4312,1423,1243\}&   75,255,813,2443,6985,19175,50917,131555,332257,823263,2007005,4825051\\\hline 146&\{4132,1243,1234\}&   75,256,826,2535,7474,21370,59718,164082,445266,1197326,3198035,8499466\\\hline 147&\{4132,1324,1234\}&  75,258,842,2614,7787,22466,63273,175044,477897,1291997,3467411,9254394\\\hline 148&\{2134,4132,1324\}&  75,258,845,2649,8019,23630,68216,193861,544312,1514024,4180488,11476203\\\hline 149&\{3412,4123,1234\}& 75,259,849,2638,7817,22275,61539,166007,439844,1150070,2978785,7665397\\\hline 150&\{4312,4132,1324\}&  75,259,852,2669,7997,23043,64190,173677,458255,1183139,2997544,7470237\\\hline 151&\{4312,1324,1423\}&  75,259,853,2684,8120,23782,67845,189493,520359,1409742,3778514,10042552\\\hline 152&\{4231,2341,4123\}& 75,259,862,2808,9090,29489,96076,314011,1027749,3364559,11012071,36033146\\\hline 153&\{4231,1324,1423\}&   75,260,864,2756,8485,25365,74021,211814,596506,1658102,4560087,12431775\\\hline 154&\{4312,1342,1423\}& 75,260,869,2817,8920,27745,85113,258256,776717,2319093,6882432,20321017\\\hline 155&\{3124,4132,1243\}& 75,261,876,2839,8923,27329,81923,241257,700150,2007431,5698047,16039035\\\hline 156&\{4132,1324,4123\}& 75,261,876,2840,8934,27399,82259,242599,704816,2021818,5737262,16130049\\\hline 157&\{4213,1342,1324\}& \\
&\{3124,4132,1432\}&    75,261,877,2852,9020,27877,84533,252331,743389,2166062,6252642,17905365\\\hline 158&\{3412,1324,1423\}& 75,261,879,2879,9232,29148,90995,281730,866917,2655218,8103324,24660429\\\hline 159&\{3412,1423,1243\}& 75,262,889,2938,9500,30183,94559,292940,899443,2742038,8312058,25083465\\\hline 160&\{4312,1432,1342\}& 75,262,890,2949,9575,30590,96486,301269,933171,2872102,8794946,26822901\\\hline 161&\{4312,4132,1342\}& 75,262,891,2964,9700,31374,100639,320949,1019396,3228687,10206180,32219494\\\hline 162&\{3412,1342,4123\}&  75,262,893,2992,9925,32747,107743,353949,1161732,3810960,12497234,40973185\\\hline 163&\{3412,3124,1423\}&  75,262,894,3011,10120,34213,116864,404013,1413582,5000943,17866417,64375380\\\hline 164&\{2341,4123,1423\}& 75,262,895,3022,10188,34524,118030,407754,1423886,5023900,17895739,64296859\\\hline 165&\{4312,3124,1423\}& 75,262,896,3033,10261,34906,119771,415012,1452361,5130997,18286959,65698315\\\hline 166&\{3412,3142,1324\}& \\
&\{3412,3142,1243\}&    75,263,901,3024,9980,32489,104585,333549,1055497,3318014,10371474,32261565\\\hline 167&\{3142,3124,1432\}&  75,263,904,3066,10324,34652,116179,389443,1305592,4377595,14679474,49227937\\\hline 168&\{3124,1432,1423\}& 75,264,914,3127,10621,35932,121324,409301,1380417,4655382,15700590,52954137\\\hline 169&\{3142,1423,1234\}& 75,264,918,3176,10978,37964,131362,454692,1574092,5449596,18867020,65319484\\\hline 170&\{3142,1324,1234\}& \\
&\{3142,1243,1234\}&    75,265,925,3201,11017,37793,129393,442497,1512225,5165953,17643457,60250113\\\hline 171&\{3124,1342,4123\}& \\
&\{1342,1324,4123\}&    75,265,926,3216,11152,38741,135126,473872,1672151,5939232,21234409,76406414\\\hline 172&\{2143,4132,1324\}& 75,265,927,3229,11253,39355,138362,489440,1742576,6244395,22516585,81673947\\\hline 173&\{4213,1342,1423\}& \\
&\{4132,1342,4123\}&    75,265,929,3249,11362,39746,139060,486549,1702349,5956172,20839367,72912441\\\hline 174&\{4312,3412,1432\}& \\
&\{4312,3412,1342\}&    \\
&\{4312,3142,1432\}&    \\
&\{4312,3142,1342\}&    \\
&\{3412,4132,1432\}&    \\
&\{3412,4132,1342\}&    75,266,935,3263,11326,39155,134955,464094,1593231,5462447,18709694,64035275\\\hline 175&\{2413,1342,4123\}& 75,266,939,3311,11676,41183,145273,512466,1807791,6377231,22496580,79359907\\\hline 176&\{3412,4132,1423\}& 75,266,939,3315,11737,41732,149104,535469,1932998,7013680,25574106,93689214\\\hline 177&\{3412,3142,1432\}& \\
&\{3412,1432,1423\}&    75,267,948,3363,11928,42306,150051,532203,1887627,6695070,23746197,84223446\\\hline 178&\{3124,4132,1423\}& 75,267,948,3367,11988,42842,153783,554624,2009904,7318260,26768537,98339812\\\hline 179&\{2134,1432,1423\}& 75,267,950,3384,12065,43034,153524,547744,1954328,6973114,24880601,88776363\\\hline 180&\{2431,4132,1324\}& 75,267,951,3407,12309,44867,164891,610347,2273020,8508804,31991549,120734511\\\hline
181&\{2143,1324,1234\}& 75,268,958,3425,12245,43778,156514,559565,2000543,7152292,25570698,91419729\\\hline 182&\{3412,3124,4132\}&  75,268,961,3467,12591,46012,169088,624478,2316582,8627816,32247951,120920851\\\hline 183&\{4132,4123,1234\}&   76,263,843,2501,6941,18245,45928,111721,264482,612707,1394929,3131269\\\hline 184&\{4213,4132,1324\}&   76,270,927,3074,9886,30985,95064,286558,851203,2497550,7252494,20874861\\\hline 185&\{2341,4123,1234\}& 76,270,929,3118,10354,34472,116097,397167,1380884,4872188,17405889,62819962\\\hline 186&\{4132,4123,1243\}& 76,270,930,3114,10196,32820,104283,328048,1023854,3175395,9797833,30104416\\\hline 187&\{3124,4132,1324\}&  76,271,938,3146,10252,32583,101368,309697,931708,2766374,8121630,23612985\\\hline 188&\{2341,3412,4123\}&   76,273,964,3356,11587,39866,137055,471326,1621698,5581897,19216642,66160957\\\hline 189&\{2143,2134,1432\}& 76,273,971,3439,12172,43098,152649,540730,1915445,6785029,24034177,85134498\\\hline 190&\{2413,3124,1342\}& 76,274,977,3449,12086,42141,146469,508098,1760610,6096937,21106816,73058238\\\hline 191&\{2134,3142,1423\}& \\
&\{3142,3124,1243\}&    \\
&\{3142,1342,1234\}&    \\
&\{3124,1432,1324\}&    76,274,978,3463,12201,42869,150415,527426,1848905,6480722,22715293,79617891\\\hline 192&\{4132,1423,1243\}& 76,274,979,3479,12351,43951,157081,564409,2039465,7410650,27070098,99369477\\\hline 193&\{2413,4132,1324\}& 76,275,989,3539,12631,45066,161021,576887,2074166,7488003,27150233,98878251\\\hline 194&\{3124,4123,1243\}&  76,275,989,3541,12660,45316,162694,586506,2124192,7730537,28267633,103834509\\\hline 195&\{1324,4123,1243\}&    76,275,989,3544,12696,45578,164194,593966,2158090,7875503,28862235,106203597\\\hline 196&\{4213,3142,1342\}&   \\
&\{2134,1432,1324\}&    \\
&\{3412,1342,1423\}&    \\
&\{3142,1342,4123\}&    76,275,991,3563,12800,45976,165141,593184,2130737,7653715,27492557,98754742\\\hline 197&\{4312,3142,1423\}&  76,275,991,3566,12848,46426,168390,613252,2242584,8233836,30347562,112259358\\\hline 198&\{1342,4123,1234\}&    76,275,991,3566,12850,46458,168686,615340,2255101,8301270,30684958,113860149\\\hline 199&\{1342,4123,1243\}&    76,275,993,3593,13068,47838,176277,653538,2436158,9124352,34315674,129523198\\\hline 200&\{2143,3124,1243\}&    76,276,1001,3626,13126,47501,171876,621876,2250001,8140626,29453126,106562501\\\hline 201&\{3142,1324,1243\}&  \\
&\{3124,1342,1423\}&76,276,1002,3641,13261,48451,177651,653753,2414426,8947576,33266626,124062001\\\hline 202&\{1432,1423,1234\}&    76,277,1012,3702,13553,49642,181885,666542,2442922,8954133,32821408,120310377\\\hline 203&\{3142,1432,1324\}&  \\
&\{3124,1423,1234\}&     76,277,1015,3743,13893,51874,194693,733983,2777748,10547615,40169157,153377405\\\hline 204&\{3124,1342,1243\}&  76,277,1016,3756,13994,52491,197987,750185,2853359,10888249,41666366,159841363\\\hline 205&\{1432,1324,1234\}&  76,278,1019,3734,13678,50100,183514,672230,2462490,9020556,33043996,121046420\\\hline 206&\{1432,1243,1234\}&   76,278,1021,3756,13827,50916,187512,690593,2543444,9367525,34500756,127067006\\\hline 207&\{2134,1423,1243\}&   76,278,1026,3818,14308,53932,204273,776859,2964716,11348261,43552751,167535302\\\hline 208&\{3124,1342,1234\}&  76,278,1029,3859,14642,56067,216174,837832,3260369,12728853,49828647,195502526\\\hline 209&\{3142,1432,1243\}&  76,279,1043,3979,15464,61035,243956,985332,4015149,16486978,68152082,283379950\\\hline 210&\{4132,1324,1243\}&  77,283,1032,3740,13522,48930,177564,646908,2367121,8699706,32108614,118975273\\\hline 211&\{1324,4123,1234\}&   77,284,1041,3789,13730,49679,179906,653083,2378702,8696754,31921462,117624497\\\hline 212&\{3142,4132,1324\}&   77,285,1053,3875,14212,52021,190301,696532,2553047,9377034,34525630,127466481\\\hline 213&\{4132,1324,1423\}&   77,285,1054,3889,14330,52800,194748,719602,2664989,9894443,36831886,137465657\\\hline 214&\{3412,4123,1423\}&   77,285,1055,3905,14476,53812,200709,751206,2820944,10625962,40138957,152012381\\\hline 215&\{2143,2134,1243\}&\\
&\{2143,1243,1234\}&\\
&\{3142,3124,1342\}&\\
&\{1432,4123,1423\}&     77,286,1066,3977,14841,55386,206702,771421,2878981,10744502,40099026,149651601\\\hline 216&\{2143,3412,3142\}& 77,286,1066,3978,14858,55556,208012,780045,2930085,11025510,41560770,156938430\\\hline
217&\{4132,1342,1243\}& 77,286,1067,3992,14976,56338,212517,803758,3047409,11580777,44103581,168294630\\\hline 218&\{3142,3124,1423\}&  \\
&\{3142,1324,1423\}&    \\
&\{3124,1423,1243\}&     77,286,1067,3993,14992,56488,213600,810449,3084733,11774727,45061101,172844990\\\hline 219&\{3412,3142,1423\}&  77,286,1068,4006,15093,57104,216875,826448,3158726,12104591,46494761,178964400\\\hline 220&\{2413,4132,1342\}&  77,286,1069,4018,15182,57636,219701,840422,3224664,12405795,47838633,184854955\\\hline 221&\{2413,3142,1324\}& \\
&\{2143,3142,1324\}&    \\
&\{2143,1324,1423\}&    \\
&\{3142,4132,1243\}&    \\
&\{3142,4123,1423\}&    \\
&\{4132,1432,1243\}&    \\
&\{4132,1342,1324\}&     77,287,1079,4082,15522,59280,227240,873886,3370030,13027730,50469890,195892565\\\hline 222&\{4312,3412,1423\}&  77,287,1080,4094,15611,59811,230048,887674,3434510,13319262,51756304,201467116\\\hline 223&\{3142,1423,1243\}&  77,287,1082,4128,15945,62330,246328,982977,3956136,16041373,65473465,268790735\\\hline 224&\{4132,1342,1423\}&  77,288,1091,4172,16069,62240,242152,945536,3703095,14539109,57204767,225484743\\\hline 225&\{2413,3142,1243\}&  77,288,1093,4202,16341,64187,254313,1015163,4078777,16481961,66940960,273115842\\\hline 226&\{2143,3142,1423\}&     77,288,1093,4203,16359,64377,255857,1025889,4145966,16873475,69105368,284618324\\\hline 227&\{2143,1432,1324\}&     77,289,1103,4261,16603,65100,256466,1014107,4021836,15988827,63691619,254145940\\\hline 228&\{3412,3142,4123\}&     77,289,1107,4322,17162,69137,281917,1161404,4826652,20211146,85192214,361185371\\\hline 229&\{2413,3142,4123\}&    \\
&\{2143,1342,1423\}&    \\
&\{2134,1342,1423\}&     77,290,1118,4398,17595,71385,293042,1215035,5081259,21408350,90786332,387212538\\\hline 230&\{4123,1243,1234\}&     78,294,1108,4165,15638,58762,221324,836330,3171916,12074924,46131496,176825773\\\hline 231&\{1324,4123,1423\}&  78,297,1143,4419,17119,66386,257621,1000407,3887666,15119991,58856167,229312425\\\hline 232&\{4123,1423,1234\}&     78,297,1144,4433,17238,67184,262276,1025202,4011660,15712335,61590780,241610745\\\hline 233&\{2143,1324,1243\}&\\
&\{2134,1324,1243\}&\\
&\{2134,1243,1234\}&\\
&\{3142,4132,1432\}&\\
&\{3142,4132,1342\}&\\
&\{3142,4132,1423\}&\\
&\{3142,1342,1324\}&\\
&\{3124,1342,1324\}&\\
&\{3124,1324,1423\}&\\
&\{4132,1432,1324\}&\\
&\{4132,4123,1423\}&\\
&\{1342,4123,1423\}&     78,298,1157,4539,17936,71251,284188,1137076,4561093,18333337,73816489,297635750\\\hline 234&\{3412,3142,4132\}&    \\
&\{3142,1342,1243\}&78,299,1172,4677,18947,77746,322545,1350906,5704822,24265651,103872254,447146683\\\hline 235&\{3412,4132,4123\}&    \\
&\{3142,1432,1423\}&    \\
&\{3124,1243,1234\}&     78,301,1197,4875,20235,85294,364131,1571212,6841633,30025137,132668839,589726354\\\hline 236&\{3124,4123,1423\}&   \\
&\{4132,1432,1342\}&    \\
&\{1432,1324,1423\}&    \\
&\{1324,1423,1243\}&    \\
&\{1423,1243,1234\}&79,309,1237,5026,20626,85242,354080,1476368,6173634,25873744,108628550,456710589\\\hline 237&\{1432,1324,1243\}&     79,310,1251,5150,21517,90921,387595,1663936,7183750,31158310,135661904,592558096\\\hline 238&\{2413,3142,1423\}&   \\
&\{4312,3412,3142\}&    \\
&\{1342,1324,1423\}&    \\
&\{1342,1423,1243\}&    \\
&\{1342,1423,1234\}&79,310,1251,5151,21536,91137,389510,1678565,7284975,31811311,139661231,616097345\\\hline 239&\{2413,3412,3142\}&    \\
&\{4312,3412,4132\}&    \\
&\{3412,3142,1342\}&    \\
&\{3142,1432,1342\}&    \\
&\{3142,1342,1423\}&    \\
&\{3124,1324,1243\}&    \\
&\{1432,1423,1243\}&    \\
&\{1324,1423,1234\}&    \\
&\{4123,1423,1243\}&79,311,1265,5275,22431,96900,424068,1876143,8377299,37704042,170870106,779058843\\\hline 240&\{4312,3412,4123\}&79,313,1290,5475,23764,105001,470738,2136022,9791501,45275765,210931962,989153896\\\hline 241&\{1324,1243,1234\}&80,322,1346,5783,25372,113174,511649,2338988,10793251,50205607,235156609,1108120540\\\hline
242&\{1432,1342,1423\}&80,322,1347,5798,25512,114236,518848,2384538,11068567,51817118,244370806,1159883685\\\hline
\end{longtable}
\end{document}